\documentclass[reqno,centertags,11pt,twoside]{article}
\usepackage[margin=1in]{geometry}
\usepackage{amsfonts,amsmath,amssymb,amsthm}
\usepackage{mathtools,thmtools,thm-restate}
\usepackage{tikz,graphicx,caption,subcaption}

\usepackage{newtxtext}
\usepackage[varvw]{newtxmath}
\DeclareMathAlphabet{\mathbb}{U}{msb}{m}{n}
\usepackage[english]{babel}
\usepackage{accents}

\usepackage{fancyhdr}
\fancypagestyle{plain}{\fancyhf{}}
\DeclarePairedDelimiter{\paren}{(}{)}
\DeclarePairedDelimiter{\brac}{[}{]}
\DeclarePairedDelimiter{\curly}{ \{ }{ \} }
\DeclarePairedDelimiter{\abs}{\lvert}{\rvert}
\DeclarePairedDelimiter{\ceil}{\lceil}{\rceil}

\newcommand{\R}{\mathbb{R}}
\newcommand{\T}{\mathbb{T}}

\DeclareMathOperator{\sgn}{sgn}
\newcommand{\eps}{\varepsilon}
\renewcommand{\phi}{\varphi}
\renewcommand{\kappa}{\varkappa}
\newcommand{\pre}{^{-1}}
\newcommand{\ov}[1]{\overline{#1}}
\newcommand{\wt}[1]{\widetilde{#1}}
\DeclareMathOperator{\supp}{supp}

\let\div\relax
\DeclareMathOperator{\div}{div}
\newcommand{\grad}{\nabla}
\newcommand{\lap}{\Delta}

\DeclarePairedDelimiter{\nm}{\|}{\|}
\DeclarePairedDelimiter{\gen}{\langle}{\rangle}
\newcommand*{\dH}{\accentset{\mbox{\large\bfseries .}}{H}}

\usepackage{cite}
\usepackage[hypertexnames = false]{hyperref}
\usepackage{cleveref,amsrefs}

\usepackage[shortlabels]{enumitem}
\numberwithin{equation}{section}
\newcommand\numberthis{\addtocounter{equation}{1}\tag{\theequation}}
\theoremstyle{plain}

\newtheorem{theorem}{Theorem}[section]
\crefname{theorem}{theorem}{theorems}
\newtheorem{lemma}[theorem]{Lemma}
\crefname{lemma}{lemma}{lemmas}
\newtheorem{proposition}[theorem]{Proposition}
\crefname{proposition}{proposition}{propositions}
\newtheorem{corollary}[theorem]{Corollary}
\crefname{corollary}{corollary}{corollaries}
\newtheorem{definition}[theorem]{Definition}
\crefname{definition}{definition}{definitions}
\newtheorem{remark}[theorem]{Remark}
\crefname{remark}{remark}{remarks}
\newtheorem{hypotheses}[theorem]{Hypotheses}
\crefname{hypotheses}{hypotheses}{hypotheses}

\newtheorem{assumption}{Assumption}
\crefname{assumption}{assumption}{assumptions}
\newtheorem{normalization}{Normalization}
\crefname{normalization}{normalization}{normalizations}

\begin{document}
\title{\bfseries \large \MakeUppercase{A Nash stratification inequality and global regularity for a chemotaxis-fluid system on general 2D domains}}
\author{\normalsize \MakeUppercase{Alexander Kiselev and Naji A. Sarsam}}
\date{}
\maketitle

\begin{abstract}
Incompressible fluid advection has been shown to facilitate singularity suppression in various differential equations, often by mixing-enhanced diffusion or by dimension-reduction effects. To aid with the study of such scenarios, we prove a refinement of the classical Nash inequality,
\[
\|f-f_M\|_{L^2}^2 \lesssim \|f-f_M\|_{L^1}^{8/7}\|\nabla f\|_{L^2}^{6/7} + \|\partial_1 f\|_{\accentset{\mbox{\large\bfseries .}}{H}_0^{-1}}^{2\vartheta} \|f-f_M\|_{L^1}^{1-\vartheta}\|\nabla f\|_{L^2}^{1-\vartheta},
\]
for $f \in H^1$ with mean $f_M$ over a smooth bounded planar domain under the main constraint of having connected horizontal cross-sections. The first term on the right-hand side follows the classical Nash scaling for a formal dimension of $3/2$. The second term introduces a mixing norm that measures how far $f$ is from being stratified. The proof provides an explicit exponent $0 < \vartheta \ll 1$.

As an application, we study the 2D parabolic-elliptic Patlak--Keller--Segel (PKS) chemotaxis model over the aforementioned large class of bounded domains. This aggregation-diffusion equation is well-known to produce finite-time singularity formation for large-mass data in finite domains. Using the above Nash stratification inequality, we prove that the 2D PKS equation becomes globally regular when actively coupled via buoyancy to a fluid obeying Darcy's law for incompressible porous media flows. This result holds for arbitrarily large $C^\infty$ initial data, far from perturbative regimes, and for arbitrarily weak coupling strength. Moreover, the spatial domain can have ``bottle neck'' regions and boundary segments of large curvature. The argument conceptualizes and generalizes recent work by Hu, Yao, and the first author on the analogous result for the periodic channel.
\end{abstract}

\begin{small}
\quotation
\vspace{8pt}

\vspace{8pt}

\noindent\textbf{AK.} Department of Mathematics, Duke University, Durham, NC 27708, USA.\\
\textbf{Email.} \href{mailto:alexander.kiselev@duke.edu}{\tt{alexander.kiselev@duke.edu}}\vspace{8pt}

\noindent\textbf{NAS.} Department of Mathematics, Duke University, Durham, NC 27708, USA.\\
\textbf{Email.} \href{mailto:naji.sarsam@duke.edu}{\tt{naji.sarsam@duke.edu}}
\endquotation
\end{small}
\vfill

\newpage
\tableofcontents
\vfill
\newpage

\section{Introduction}
\label{sec:Intro}

For functions $f \in H^1(\Omega)$ on an appropriate bounded domain $\Omega \subseteq \R^d$, the classical Nash inequality states that
\begin{equation}\label{eq:secIntro_Nash}
\nm{f-f_M}_2^{1+\frac{2}{d}} \le C_{\mathrm{N},d} \nm{f-f_M}_1^{\frac{2}{d}}\nm{\grad f}_2,
\end{equation}
where $f_M$ denotes the mean of $f$ over $\Omega$ and $C_{\mathrm{N},d} := C_{\mathrm{N},d}(\Omega) > 0$ denotes the associated universal constant. The analogue of (\ref{eq:secIntro_Nash}) for full Euclidean space $\R^d$, with $f-f_M$ replaced by $f$, was first established by Nash \cites{N58} via a simple Fourier transform argument. In that celebrated work, Nash was interested in proving the H\"{o}lder continuity of solutions to general coefficient parabolic equations. He used his name-sake inequality to establish an a-priori $L^2$ bound on fundamental solutions to such equations, which he then upgraded to an important $L^\infty$ estimate. The Nash inequality (\ref{eq:secIntro_Nash}) has since become an indispensable tool in the derivation of such a-priori estimates for many partial differential equations, with its generalization being the Gagliardo--Nirenberg--Sobolev interpolation inequality. The inequality may be particularly helpful for physical equations where the $L^1$ norm of a solution, representing total mass, is conserved, while the $L^2$ norm may grow. However, the utility of (\ref{eq:secIntro_Nash}) in such cases depends on the criticality of the equation being studied with respect to the scaling of the inequality.

An illustrative example is provided by the parabolic-elliptic Patlak--Keller--Segel (PKS) equation, for which the inequality (\ref{eq:secIntro_Nash}) barely fails to establish the global regularity of smooth solutions in two dimensions. In turn, there exist large-mass smooth solutions that can develop singularities in finite time, while all small-mass smooth solutions are globally regular. This phenomenon has garnered significant attention in the recent literature, which we review below. It serves as the motivating example for this work, as we provide an anisotropic refinement of the classical Nash inequality (\ref{eq:secIntro_Nash}) in the case of $d =2$. This refinement leads to the proof of global regularity for a system where fluid advection is coupled to the PKS equation via a buoyancy/gravitational force.

\subsection{Previous literature on the PKS equation}
\label{subsec:BackChemo}

The Patlak--Keller--Segel (PKS) equation is a family of aggregation-diffusion equations that model the biological phenomenon of chemotaxis. Therein, a scalar density $\rho(x,t)$, representing a bacteria population for example, evolves under diffusion as well as transport up the gradient of a scalar $c(x,t)$ that is coupled to $\rho(x,t)$. The density $c(x,t)$ represents a chemical concentration produced by the bacteria and to which the bacteria are attracted \cites{KS71,P53}. The literature on this subject is quite expansive and so we may only provide a representative summary of works in what follows.

The general form of the PKS equation on an open bounded connected domain $\Omega \subseteq \R^d$ is given by
\begin{equation}\label{eq:secIntro_PKSGeneral}
\begin{cases}
\partial_t \rho - \lap \rho + \div(\rho \grad c) = 0, \\
\partial_n \rho |_{\partial \Omega} = 0, \quad \partial_n c |_{\partial \Omega} = 0, \\
\rho(\cdot, 0) = \rho_0 \ge 0,
\end{cases}
\end{equation}
which is also supplied with a second equation, of reaction-diffusion type, relating the density $c(x,t)$ to $\rho(x,t)$ in a way specific to the biological context being modeled. There are two important physical considerations that are often ensured. First, any solution $\rho(x,t)$ should remain nonnegative for all times of existence. Second, the $L^1$ norm of $\rho(x,t)$, corresponding to the total mass of bacteria, should be finite and constant for all times (unless a specific growth or decay term is added to the equation). To that end, the initial condition $\rho_0(x)$ in (\ref{eq:secIntro_PKSGeneral}) is assumed nonnegative and the Neumann boundary conditions $\partial_n \rho |_{\partial \Omega} = \partial_n c |_{\partial \Omega} = 0$ are imposed. These considerations motivate studying the criticality of the PKS equation (\ref{eq:secIntro_PKSGeneral}) with respect to the $L^\infty_t L^1_x$ norm of $\rho(x,t)$. To do so, one must first specify the equation determining $c(x,t)$ from $\rho(x,t)$.

The possibly most studied version of the PKS equation is the parabolic-elliptic form where
\begin{equation}\label{eq:secIntro_PKSeqC}
-\lap c = \rho-\rho_M,
\quad
\rho_M := \frac{1}{|\Omega|}\int_\Omega \rho(x,t)\, dt,
\end{equation}
with $\rho_M$ denoting the spatial mean of the density. Equation (\ref{eq:secIntro_PKSeqC}) was introduced by \cites{N73,JL92} in order to study the simplified physical scenario where the chemical $c(x,t)$ is produced and diffuses on a much faster time scale than the other important processes represented by (\ref{eq:secIntro_PKSGeneral}). The parabolic-elliptic PKS system (\ref{eq:secIntro_PKSGeneral}, \ref{eq:secIntro_PKSeqC}) indeed preserves the nonnegativity and total mass of a classical solution for all times of existence; see \Cref{subsec:PKSIPMPreliminaries}. Moreover, given any nonnegative solution $\rho(x,t)$ of (\ref{eq:secIntro_PKSGeneral}, \ref{eq:secIntro_PKSeqC}), it holds that
\begin{equation}\label{eq:secIntro_PKSscaling}
\rho_\mu(x,t) := \frac{1}{\mu^2} \rho\paren*{\frac{x}{\mu},\ \frac{t}{\mu^2}}
\end{equation}
is also formally a solution for any scalar $\mu > 0$ in the $\R^d$ setting.

The scaling (\ref{eq:secIntro_PKSscaling}) is $L^\infty_t L^1_x$-subcritical when $d = 1$; hence the conservation of mass leads to the global regularity of smooth solutions to (\ref{eq:secIntro_PKSGeneral}, \ref{eq:secIntro_PKSeqC}) in one dimension \cite{N95}. The opposite situation occurs in dimension $d = 3$ where the scaling (\ref{eq:secIntro_PKSscaling}) is instead $L^\infty_t L^1_x$-supercritical and $L^\infty_t L^{3/2}_x$-critical. Consequently, solutions with arbitrarily small mass are known to blow-up in finite time while solutions initially small in $L^{3/2}$ do exist globally in time \cites{CGMN23,CPZ04,HMV97,HMV98}. In the remaining case of dimension $d =2$, the scaling (\ref{eq:secIntro_PKSscaling}) is $L^\infty_t L^1_x$-critical and singularity formation is possible. Indeed, on bounded planar domains, \cites{JL92,N95,N01} were among the first to show that all smooth initial data with sufficiently small mass lead to globally regular solutions while there do exist smooth initial data with large mass that develop a finite-time singularity. Moreover, \cites{HV96a,HV96b} gave examples of such initial data which specifically converge to a Dirac delta measure in finite time. Since these works, the study of the two-dimensional parabolic-elliptic PKS system (\ref{eq:secIntro_PKSGeneral}, \ref{eq:secIntro_PKSeqC}) has received considerable study by the literature with a very notable achievement including the establishment of a sharp mass threshold for the setting of the full plane, as comprehensively summarized by the detailed surveys \cites{H03,H04} and the text \cite{P07}. We also note that ``chemotactic collapse'' to a two-dimensional Dirac mass is now considered to be the ``generic'' way that singularities may form under (\ref{eq:secIntro_PKSGeneral}, \ref{eq:secIntro_PKSeqC}) --- again see \cites{H03,H04,P07} --- and we point to \cites{CTMN22a,CTMN22b,RS14} as a small sample of the recent state of the art.

The possibility of blow-up in two dimensions is reflected by the classical Nash inequality (\ref{eq:secIntro_Nash}). Indeed, any nonnegative solution $\rho(x, t)$ to (\ref{eq:secIntro_PKSGeneral}, \ref{eq:secIntro_PKSeqC}) in two dimensions obeys the following easily proven a-priori estimate:
\begin{equation}\label{eq:secIntro_AprioriVariance}
\frac{d}{dt}\nm{\rho - \rho_M}_2^2 + \nm{\grad \rho}_2^2
\le C\nm{\rho-\rho_M}_2^4,
\end{equation}
for times $t$ where the variance $\nm{\rho(\cdot, t)-\rho_M}_2$ is sufficiently large; see \Cref{subsec:Variance}. Here, $C = C(\Omega)$ denotes a universal constant depending only on the domain. Combining this estimate with (\ref{eq:secIntro_Nash}) gives
\[
\frac{d}{dt}\nm{\rho - \rho_M}_2^2
\le
\paren*{4C_{\mathrm{N},2}^2C\nm{\rho_0}_1^2-1}\nm{\grad \rho}_2^2.
\]
Thus, if the initial mass $\nm{\rho_0}_1$ is sufficiently small, we obtain boundedness in time of the solution's variance, which is a subcritical quantity in two dimensions with respect to the scaling (\ref{eq:secIntro_PKSscaling}). However, we also see that if the initial mass is sufficiently large, we obtain no such control. So, the scaling of the classical Nash inequality (\ref{eq:secIntro_Nash}) in two-dimensions appears just barely too weak to ensure global regularity of the PKS system (\ref{eq:secIntro_PKSGeneral}, \ref{eq:secIntro_PKSeqC}). One quickly observes that, if the exponent on the $L^1$ norm were at all larger than the exponent on the $\dH^1$ norm in (\ref{eq:secIntro_Nash}) in the two-dimensional case, then the solution's variance would have to be bounded for all time and in turn global regularity would hold. This is an important observation for our work, as we will see below in \Cref{thm:secIntro_NashTotal,thm:secIntro_GWP}.

\subsection{Previous literature on PKS-fluid systems}
\label{subsec:BackChemoFluids}

Since singularity formation in the two-dimensional PKS system is relatively well understood, the system (\ref{eq:secIntro_PKSGeneral}, \ref{eq:secIntro_PKSeqC}) provides a perfect playground for studying various singularity suppression mechanisms. One such mechanism, that is of considerable interest biologically as well as theoretically, is when there exists a background incompressible fluid within which the bacteria reside. To model this, a term of the form $A u \cdot \grad \rho$ can be added to the left-hand side of the equation for $\rho$ in (\ref{eq:secIntro_PKSGeneral}). Here, $u(x,t)$ represents the fluid velocity field and $A > 0$ is a constant that may be interpreted as representing the strength of fluid advection.

The works \cites{BH17,HT19,KX16,IXZ21} treat the system (\ref{eq:secIntro_PKSGeneral}, \ref{eq:secIntro_PKSeqC}) in two dimensions with the addition of passive fluid advection $A u\cdot \grad \rho$. This means that the fluid velocity is a-priori fixed, with no dependence on $\rho$. Each of \cites{BH17,HT19,KX16,IXZ21} identifies a class of fluid velocities $u(x,t)$ for which the following result holds: given any regular initial data $\rho_0$ (with some results requiring a mass smallness condition), there exists a sufficiently large coupling constant $A > 0$ depending on $\rho_0$ and $u$ for which the corresponding solution $\rho(x,t)$ exists in a classical strong sense for all time. We summarize the types of fluid flows considered; see the recent survey \cite{FMN25} for a comprehensive overview of such results for various nonlinear equations including the PKS system.

The chronological first result \cite{KX16} considers the mixing or ``relaxation enhancing'' fluid velocities constructed in \cites{ACM19,CKRZ08,YZ17}. The fluid mixing causes the solution $\rho(x,t)$ to have large gradients, leading to an enhanced dissipation effect that allows diffusion to triumph over chemotactic aggregation. The initial data $\rho_0$ can be an arbitrary sufficiently regular function, and the global regularity holds for sufficiently strong flow: $A \geq A_0(\rho_0).$ The same result also holds in three dimensions. The work \cite{BH17} instead considers regular shear flows $u(x,t)$ which effectively lower the dimension of the dynamics by one for very large $A.$ Thus, the two-dimensional setting requires no mass smallness condition on $\rho_0$ while the three-dimensional setting develops a critical mass threshold. We also note that \cite{H18} extends this aforementioned shearing regularization mechanism to the parabolic-parabolic PKS system, where the equation (\ref{eq:secIntro_PKSeqC}) is replaced by a parabolic advection-diffusion equation. Another type of flow is considered in \cite{HT19}: the specific time-stationary hyperbolic flow $u(x_1,x_2) := (-x_1,x_2)$. The authors show that even this simple fluid flow induces a ``fast splitting scenario'' which causes the critical mass threshold for blow-up in two-dimensions to double from the original $8\pi$ to $16 \pi$, as long as $A$ is very large. The follow-up work \cite{HTZ22} identifies this fast splitting scenario in a general class of parabolic equations with chemotaxis-type and ignition-type nonlinearities. Lastly, the paper \cite{IXZ21} derives a general ``dissipation time'' criterion for suppression of finite-time blow up that is applicable more broadly, and applies it to analysis of the PKS system with some cellular flows --- obtaining global regularity if $A$ exceeds some threshold that depends on the initial data.

A more delicate problem is presented by situations where,
instead of advection by a passive fluid, the fluid velocity $u(x,t)$ is active, that is, depends on the solution $\rho(x,t)$. Thus, $u(x,t)$ will be specified to solve a particular fluid equation with the addition of a forcing/coupling term that involves $\rho(x,t)$. Typically, this force represents buoyancy/gravity, with a coupling constant that we denote by $g \neq 0$. Given the nonlinear and nonlocal nature of fluid equations, such actively coupled PKS-fluid systems are much more challenging to study than the passive counterparts. Consequently, there are many works on many variations of these systems, and so we focus here on those that address the question of global regularity for classical solutions (the precise regularity of which varies from work to work).

The global regularity of small initial data is treated by \cites{CKL14,DLM10,FLM10,L12} for many different PKS-fluid systems. The need for small initial data is underscored by the recent work \cite{LT25}, which provides an example of finite-time blow-up in three dimensions for the specific system of (\ref{eq:secIntro_PKSGeneral}, \ref{eq:secIntro_PKSeqC}) coupled to the Navier--Stokes equations via buoyancy. Also treating the same PKS--Navier--Stokes system, \cites{H23,ZZZ21} demonstrate the global regularity of large solutions which are instead ``small'' in the sense of being sufficiently close perturbations of the Couette shear flow; \cite{LXX25} does the same for perturbations about the Poiseuille flow. The shear flow for these three results is also required to have sufficiently large amplitude $A$, thus providing a dimension reduction mechanism via shear-induced enhanced dissipation and shear flow stability. On the other hand, \cites{H24,HK24} establish global regularity of arbitrary size classical solutions to (\ref{eq:secIntro_PKSGeneral}, \ref{eq:secIntro_PKSeqC}) when actively coupled to the Navier--Stokes flow (with Neumann boundary conditions) and Stokes--Boussinesq flow (with Dirichlet boundary conditions) via buoyancy/gravity, respectively. However, the gravitational constant $g$ is required to be sufficiently large depending on the initial data.

We should note that many of these aforementioned papers on actively coupled PKS-fluid systems do not only focus on global regularity for classical solutions. They also treat other questions such as stability, long-time behavior, and the existence of weak solutions. Indeed, there is a very large literature on weak solutions to such systems, which we unfortunately cannot summarize here. We instead only point to the following two influential works \cites{LL11,L10} as a small sample, as they were among the first to treat the local and global existence of appropriate weak solutions for PKS-fluid systems.

There are fewer works treating the question of global regularity for actively coupled PKS-fluid systems away from the cases of small initial data, perturbative data, or sufficiently large $A$, $g$ regimes. These include the interesting results of \cites{TW16,W12,W21} which establish global regularity of arbitrary classical solutions to various PKS-fluid systems, although both the fluid equations and the chemotaxis equations being considered may not allow singularity formation when not coupled together. The first example where the uncoupled chemotaxis and fluid equations may form singularities was recently provided by Hu, Yao, and the first author \cite{HKY25}. This system is one where $\rho(x,t)$ solves (\ref{eq:secIntro_PKSGeneral}, \ref{eq:secIntro_PKSeqC}) with the addition of advection by a fluid solving the incompressible porous media (IPM) equation. We will focus on this system for the remainder of this article.

\subsection{The PKS-IPM system}
\label{subsec:BackChemoIPM}

In this article, we consider the PKS-IPM system set on an open bounded connected domain $\Omega \subseteq \R^2$:
\begin{equation}\label{eq:secIntro_PKSIPM}
\begin{cases}
\partial_t \rho + u \cdot \grad \rho - \lap \rho + \div(\rho \grad c) = 0, \\
-\lap c = \rho - \rho_M \\
u = -\grad p - (0, g\rho), \quad \div u = 0, \\
\partial_n \rho |_{\partial \Omega} = 0, \quad \partial_n c |_{\partial \Omega} = 0, \quad u \cdot n|_{\partial \Omega} = 0, \\
\rho(\cdot, 0) = \rho_0 \ge 0.
\end{cases}
\end{equation}
Here, $g \neq 0$ is the gravitational/buoyancy constant; $n(x)$ denotes the outer unit normal to $x \in \partial \Omega$; and $p(x,t)$ is the scalar pressure which ensures that the fluid velocity $u(x,t)$ must always remain incompressible. We note that for $g > 0$, the coupling force can be interpreted as the fluid being weighed down by gravity. Meanwhile, the choice of $g < 0$ can represent an upward force felt by the bacteria density due to the bacteria's buoyancy in the fluid. Throughout this article, the sign of $g$ will not affect any of the arguments. We also note that the third line of (\ref{eq:secIntro_PKSIPM}) is often called Darcy's law. Importantly, it is equivalent to the following Biot--Savart law:
\begin{equation}\label{eq:secIntro_BiotSavart}
u = g\grad^\perp(-\lap_D)\pre \partial_1 \rho,
\end{equation}
where $(-\lap_D)\pre$ denotes the inverse of the Dirichlet Laplacian; see \Cref{subsec:PKSIPMPreliminaries}. This Biot--Savart law is crucial for understanding the dynamics of the PKS-IPM system. Its first consequence is that we need only only require an initial condition on the density $\rho$ for (\ref{eq:secIntro_PKSIPM}) to be locally well-posed; see \Cref{subsec:Variance}.

The PKS-IPM system retains many of the original properties of the PKS equation (\ref{eq:secIntro_PKSGeneral}, \ref{eq:secIntro_PKSeqC}) without fluid advection. Indeed, for classical solutions, (\ref{eq:secIntro_PKSIPM}) preserves the nonnegativity and total mass for all times of existence. Moreover, the same a-priori variance estimate (\ref{eq:secIntro_AprioriVariance}) holds due to the incompressibility of the fluid velocity. However, unlike the classical the PKS equation (\ref{eq:secIntro_PKSGeneral}, \ref{eq:secIntro_PKSeqC}), the PKS-IPM system (\ref{eq:secIntro_PKSIPM}) is actually $L^\infty_t L^2_x$ critical. To that end, one can show that if a smooth $C^\infty$ solution to (\ref{eq:secIntro_PKSIPM}) blows-up at a finite time $0 < T_* < \infty$, then the variance $\nm{\rho(\cdot, t) - \rho_M}_2^2$ must tend to infinity at time $T_*$. Nonetheless, the classical Nash inequality (\ref{eq:secIntro_Nash}) for $d = 2$ along with the a-priori variance estimate still fall short of establishing global regularity of (\ref{eq:secIntro_PKSIPM}) for the same reason as discussed in \Cref{subsec:BackChemo}. All of these aforementioned properties are verified in \Cref{subsec:PKSIPMPreliminaries,subsec:Variance}.

Working specifically on the periodic strip $\T \times (0, \pi)$, the article \cite{HKY25} established global well-posedness of nonnegative $C^\infty$ solutions to (\ref{eq:secIntro_PKSIPM}) for any fixed constant $g \neq 0$, regardless of the choice of initial data $\rho_0$. Here, $\T := [-\pi, \pi)$ denotes the flat circle representing periodic boundary conditions in the first spatial variable $x_1$. This was the first result showing global regularity for a chemotaxis-fluid system far away from any smallness or perturbative regimes, where both the chemotaxis equation (the PKS system (\ref{eq:secIntro_PKSGeneral}, \ref{eq:secIntro_PKSeqC}); see \Cref{subsec:BackChemo}) and the fluid equation (the IPM equation; see \cites{KY23,CM24,D25}) admit singular behavior of solutions. Indeed, \cite{HKY25} shows that, on the periodic strip, the solution's variance must remain bounded for all time, giving global regularity due to $L^\infty_t L^2_x$ criticality. Therefore, the addition of fluid advection obeying Darcy's law completely regularizes the original PKS system (\ref{eq:secIntro_PKSGeneral}, \ref{eq:secIntro_PKSeqC}) on $\T \times (0,\pi)$.

Let us provide some loose intuition as to why the solution's variance must remain bounded, regardless of the magnitude of $g \neq 0$. Consider the case where a smooth solution $\rho$ to (\ref{eq:secIntro_PKSIPM}) begins to evolve towards a finite-time singularity, due to the chemotactic term $\div(\rho \grad c)$. The works \cites{CTMN22a,CTMN22b,HV96a,HV96b,RS14} on the original parabolic-elliptic PKS equation (\ref{eq:secIntro_PKSGeneral}, \ref{eq:secIntro_PKSeqC}) then suggest that the solution's mass should be large and concentrating like an approximation of a two-dimensional Dirac delta singularity, as depicted in \Cref{fig:secIntro_PKSIPMStrip1}. Once sufficiently concentrated, the force of gravity --- as seen in the Darcian Biot--Savart law (\ref{eq:secIntro_BiotSavart}) --- causes the fluid velocity to stretch the density concentration, as shown in \Cref{fig:secIntro_PKSIPMStrip2}. The stretched concentration then reaches the boundary, along which the fluid must move tangentially. \Cref{fig:secIntro_PKSIPMStrip3} shows that the fluid near the boundary spreads out and redistributes the density $\rho$, in a nonlinear manner. The expected effect is to make the density ``nearly stratified", or ``nearly one-dimensional", at least where it is large. Since the solutions of the PKS equation do not blow up in one dimension \cite{N95}, this can lead to global regularity.

\tikzset{every picture/.style={line width=0.75pt}}
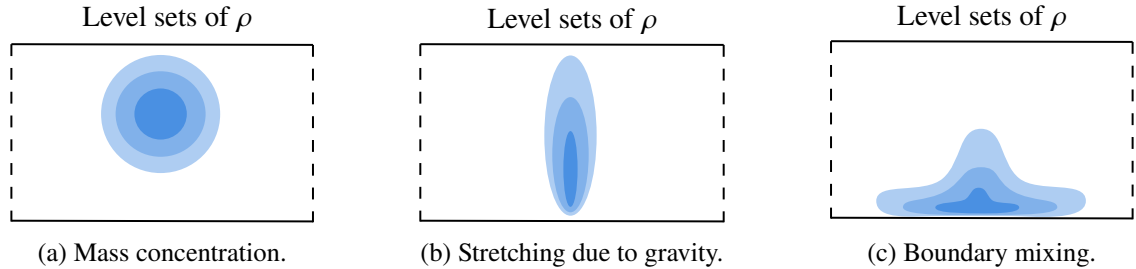
\begin{figure}[!ht]%
\centering
\begin{subfigure}{.33\textwidth}%
\captionsetup{width=0.9\textwidth}
\centering
\begin{tikzpicture}[x=0.75pt,y=0.75pt,yscale=-0.9,xscale=0.9]
\draw    (91.76,149.79) -- (260.76,150.19) ;
\draw  [dash pattern={on 4.5pt off 4.5pt}]  (91.76,50.79) -- (91.76,150.27) ;
\draw    (91.76,50.79) -- (260.76,51.19) ;
\draw  [dash pattern={on 4.5pt off 4.5pt}]  (260.76,51.19) -- (260.76,150.67) ;
\draw  [draw opacity=0][fill={rgb, 255:red, 74; green, 144; blue, 226 }  ,fill opacity=0.45 ] (142.01,89.96) .. controls (142.01,71.67) and (156.95,56.84) .. (175.38,56.84) .. controls (193.82,56.84) and (208.76,71.67) .. (208.76,89.96) .. controls (208.76,108.24) and (193.82,123.07) .. (175.38,123.07) .. controls (156.95,123.07) and (142.01,108.24) .. (142.01,89.96) -- cycle ;
\draw  [draw opacity=0][fill={rgb, 255:red, 74; green, 144; blue, 226 }  ,fill opacity=0.45 ] (150.16,89.96) .. controls (150.16,76.71) and (161.46,65.98) .. (175.38,65.98) .. controls (189.31,65.98) and (200.6,76.71) .. (200.6,89.96) .. controls (200.6,103.2) and (189.31,113.94) .. (175.38,113.94) .. controls (161.46,113.94) and (150.16,103.2) .. (150.16,89.96) -- cycle ;
\draw  [draw opacity=0][fill={rgb, 255:red, 74; green, 144; blue, 226 }  ,fill opacity=1 ] (160.71,89.96) .. controls (160.71,82.04) and (167.28,75.63) .. (175.38,75.63) .. controls (183.49,75.63) and (190.06,82.04) .. (190.06,89.96) .. controls (190.06,97.87) and (183.49,104.28) .. (175.38,104.28) .. controls (167.28,104.28) and (160.71,97.87) .. (160.71,89.96) -- cycle ;
\draw (130,28) node [anchor=north west][inner sep=0.75pt]   [align=left] {Level sets of $\rho$};
\end{tikzpicture}
\caption{Mass concentration.}
\label{fig:secIntro_PKSIPMStrip1}
\end{subfigure}%
\begin{subfigure}{.33\textwidth}%
\captionsetup{width=0.9\textwidth}
\centering
\begin{tikzpicture}[x=0.75pt,y=0.75pt,yscale=-0.9,xscale=0.9]
\draw    (80.76,148.79) -- (250.76,149.19) ;
\draw  [dash pattern={on 4.5pt off 4.5pt}]  (80.76,49.79) -- (80.76,149.27) ;
\draw    (80.76,49.79) -- (250.76,50.19) ;
\draw  [dash pattern={on 4.5pt off 4.5pt}]  (250.76,50.19) -- (250.76,149.67) ;
\draw  [draw opacity=0][fill={rgb, 255:red, 74; green, 144; blue, 226 }  ,fill opacity=0.45 ] (164.91,56.11) .. controls (173.01,56.11) and (179.57,76.23) .. (179.57,101.05) .. controls (179.57,125.87) and (173.01,145.99) .. (164.91,145.99) .. controls (156.81,145.99) and (150.25,125.87) .. (150.25,101.05) .. controls (150.25,76.23) and (156.81,56.11) .. (164.91,56.11) -- cycle ;
\draw  [draw opacity=0][fill={rgb, 255:red, 74; green, 144; blue, 226 }  ,fill opacity=0.45 ] (164.91,79.47) .. controls (170.54,79.47) and (175.11,93.93) .. (175.11,111.77) .. controls (175.11,129.61) and (170.54,144.07) .. (164.91,144.07) .. controls (159.28,144.07) and (154.71,129.61) .. (154.71,111.77) .. controls (154.71,93.93) and (159.28,79.47) .. (164.91,79.47) -- cycle ;
\draw  [draw opacity=0][fill={rgb, 255:red, 74; green, 144; blue, 226 }  ,fill opacity=1 ] (164.91,98.52) .. controls (167.11,98.52) and (168.89,108.03) .. (168.89,119.77) .. controls (168.89,131.51) and (167.11,141.02) .. (164.91,141.02) .. controls (162.71,141.02) and (160.93,131.51) .. (160.93,119.77) .. controls (160.93,108.03) and (162.71,98.52) .. (164.91,98.52) -- cycle ;
\draw (116.71,28) node [anchor=north west][inner sep=0.75pt]   [align=left] {Level sets of $\rho $};
\end{tikzpicture}
\caption{Stretching due to gravity.}
\label{fig:secIntro_PKSIPMStrip2}
\end{subfigure}%
\begin{subfigure}{.33\textwidth}%
\captionsetup{width=0.9\textwidth}
\centering
\begin{tikzpicture}[x=0.75pt,y=0.75pt,yscale=-0.9,xscale=0.9]
\draw    (79.76,148.79) -- (248.76,149.19) ;
\draw  [dash pattern={on 4.5pt off 4.5pt}]  (79.76,49.79) -- (79.76,149.27) ;
\draw    (79.76,49.79) -- (248.76,50.19) ;
\draw  [dash pattern={on 4.5pt off 4.5pt}]  (248.76,50.19) -- (248.76,149.67) ;
\draw  [draw opacity=0][fill={rgb, 255:red, 74; green, 144; blue, 226 }  ,fill opacity=0.45 ] (142.48,126.66) .. controls (153.16,117.3) and (151.02,98.84) .. (163.84,98.91) .. controls (176.65,98.98) and (174.52,118.15) .. (185.19,126.66) .. controls (195.87,135.17) and (222.57,127.37) .. (222.57,139.36) .. controls (222.57,151.34) and (191.6,147.87) .. (180.92,147.87) .. controls (170.24,147.87) and (151.02,147.87) .. (138.21,147.87) .. controls (125.4,147.87) and (105.11,150.49) .. (105.11,139.36) .. controls (105.11,128.23) and (131.8,136.02) .. (142.48,126.66) -- cycle ;
\draw  [draw opacity=0][fill={rgb, 255:red, 74; green, 144; blue, 226 }  ,fill opacity=0.45 ] (143.55,134.08) .. controls (154.23,128.12) and (153.16,120.46) .. (163.84,120.46) .. controls (174.52,120.46) and (171.31,128.12) .. (181.99,134.08) .. controls (192.67,140.04) and (206.55,137.83) .. (206.55,142.93) .. controls (206.55,148.04) and (172.38,147.19) .. (162.77,147.19) .. controls (153.16,147.19) and (120.06,148.18) .. (120.06,142.93) .. controls (120.06,137.69) and (132.87,140.04) .. (143.55,134.08) -- cycle ;
\draw  [draw opacity=0][fill={rgb, 255:red, 74; green, 144; blue, 226 }  ,fill opacity=1 ] (155.3,138.24) .. controls (158.5,135.68) and (158.5,131.8) .. (162.77,131.97) .. controls (167.04,132.14) and (167.04,136.84) .. (170.24,139.09) .. controls (173.45,141.33) and (185.19,139.67) .. (185.19,143.07) .. controls (185.19,146.48) and (169.18,146.44) .. (162.77,146.34) .. controls (156.36,146.24) and (138.21,146.48) .. (138.21,143.07) .. controls (138.21,139.67) and (152.09,140.79) .. (155.3,138.24) -- cycle ;
\draw (115.24,28) node [anchor=north west][inner sep=0.75pt]   [align=left] {Level sets of $\rho $};
\end{tikzpicture}
\caption{Boundary mixing.}
\label{fig:secIntro_PKSIPMStrip3}
\end{subfigure}%
\caption{Cartoon of how regularization of the solution to (\ref{eq:secIntro_PKSIPM}) occurs on the periodic channel $\T \times (0, \pi)$ whenever the mass begins to concentrate in the form of an approximate two-dimensional Dirac delta.}
\label{fig:secIntro_PKSIPMStrip}
\end{figure}

Of course, this intuitive argument is far from the proof, and it appears to be quite difficult to use it in a rigorous way. The proof in \cite{HKY25} does not directly follow the above intuition; it is difficult to control the interaction of nonlinear aggregation and transport. The key takeaway is that the horizontal mixing phenomenon is crucial for the global regularity to hold without requiring small initial data $\rho_0$ or large gravity $g$. This leads us to two remarks. First, the nonlinear combination of gravitational stretching and horizontal mixing highly suggests that the solution to (\ref{eq:secIntro_PKSIPM}) on the periodic strip likely does not have any simple limiting dynamics and may perhaps be turbulent. Second, the work of \cites{H24} considers a system very closely related to (\ref{eq:secIntro_PKSIPM}) except for the velocity $u$ being formally more regular than $\rho$ in both variables $x_1$, $x_2$. Correspondingly, \cites{H24} establishes global regularity in the case of sufficiently large coupling constant $g$. This can be considered as further evidence of the importance for the fluid velocity $u$ in the PKS-IPM system to have the same formal regularity as $\rho$ in $x_1$ direction to be effective in arresting blow up. However, it should be noted that it remains an open problem as to whether the system considered by \cites{H24} indeed admits singularity formation in the regime of very small $g$.

We emphasize that all of the above discussion on \cite{HKY25} pertains to the specific domain of the periodic strip $\T \times (0, \pi)$. The original goal of this work is to extend the global regularity result of (\ref{eq:secIntro_PKSIPM}) to a much larger class of general bounded domains $\Omega \subseteq \R^2$, with an example domain depicted by \Cref{fig:secIntro_PKSIPMGeneralBoundary}. Note that the top and bottom portions of the boundary are not flat but instead may exhibit very large curvature. Thus, if we consider the same thought experiment as we did in \Cref{fig:secIntro_PKSIPMStrip} for the periodic strip, it is not a-priori clear at all if there is enough room for the boundary fluid effects to spread out a large concentration of the solution $\rho$. Moreover, the example domain in \Cref{fig:secIntro_PKSIPMGeneralBoundary} also contains multiple ``bottle necks'', meaning horizontal cross sections of very short length. In such regions, one may similarly fear that horizontal mixing may not succeed in suppressing singularity formation.

\tikzset{every picture/.style={line width=0.75pt}}
\begin{figure}[!ht]%
\centering
\begin{tikzpicture}[x=0.75pt,y=0.75pt,yscale=-1,xscale=1]
\draw   (283.14,5.57) .. controls (316.64,5.66) and (335.79,17.55) .. (335.79,28.22) .. controls (335.79,38.89) and (286.99,34.87) .. (286.99,47.28) .. controls (286.99,59.69) and (360.19,62) .. (360.19,81.21) .. controls (360.19,100.42) and (317.81,92) .. (307.54,103.89) .. controls (297.26,115.78) and (349.92,126.43) .. (349.92,144.57) .. controls (349.92,162.71) and (307.54,159.77) .. (295.98,178.98) .. controls (284.42,198.19) and (292.13,229.92) .. (283.14,230.23) .. controls (274.15,230.53) and (285.71,197.12) .. (270.3,178.98) .. controls (254.89,160.84) and (215.22,162.64) .. (215.22,145.57) .. controls (215.22,128.5) and (269.01,115.08) .. (258.74,102.55) .. controls (248.46,90.02) and (206.09,98.86) .. (206.09,80.72) .. controls (206.09,62.58) and (279.29,60.08) .. (279.29,47.28) .. controls (279.29,34.47) and (231.77,37.83) .. (231.77,27.16) .. controls (231.77,16.49) and (249.64,5.48) .. (283.14,5.57) -- cycle ;
\draw  [draw opacity=0][fill={rgb, 255:red, 74; green, 144; blue, 226 }  ,fill opacity=0.45 ] (260.02,33.41) .. controls (260.02,31.64) and (261.46,30.21) .. (263.22,30.21) -- (304.34,30.21) .. controls (306.1,30.21) and (307.54,31.64) .. (307.54,33.41) -- (307.54,33.41) .. controls (307.54,35.17) and (306.1,36.61) .. (304.34,36.61) -- (263.22,36.61) .. controls (261.46,36.61) and (260.02,35.17) .. (260.02,33.41) -- cycle ;
\draw  [draw opacity=0][fill={rgb, 255:red, 74; green, 144; blue, 226 }  ,fill opacity=0.45 ] (265.6,33.77) .. controls (265.6,32.98) and (266.24,32.34) .. (267.03,32.34) -- (302.69,32.34) .. controls (303.48,32.34) and (304.13,32.98) .. (304.13,33.77) -- (304.13,33.77) .. controls (304.13,34.56) and (303.48,35.21) .. (302.69,35.21) -- (267.03,35.21) .. controls (266.24,35.21) and (265.6,34.56) .. (265.6,33.77) -- cycle ;
\draw  [draw opacity=0][fill={rgb, 255:red, 74; green, 144; blue, 226 }  ,fill opacity=1 ] (274.59,33.77) .. controls (274.59,33.21) and (275.05,32.75) .. (275.61,32.75) -- (294.12,32.75) .. controls (294.68,32.75) and (295.14,33.21) .. (295.14,33.77) -- (295.14,33.77) .. controls (295.14,34.33) and (294.68,34.79) .. (294.12,34.79) -- (275.61,34.79) .. controls (275.05,34.79) and (274.59,34.33) .. (274.59,33.77) -- cycle ;
\draw  [draw opacity=0][fill={rgb, 255:red, 74; green, 144; blue, 226 }  ,fill opacity=0.45 ] (208.24,82.49) .. controls (209.89,78.82) and (221.81,77.35) .. (234.87,79.21) .. controls (247.94,81.08) and (257.19,85.56) .. (255.54,89.23) .. controls (253.89,92.9) and (241.97,94.37) .. (228.91,92.51) .. controls (215.85,90.65) and (206.59,86.16) .. (208.24,82.49) -- cycle ;
\draw  [draw opacity=0][fill={rgb, 255:red, 74; green, 144; blue, 226 }  ,fill opacity=0.45 ] (210.17,83.31) .. controls (209.89,80.89) and (217.67,80.07) .. (227.54,81.48) .. controls (237.41,82.88) and (245.64,85.99) .. (245.91,88.41) .. controls (246.18,90.83) and (238.4,91.65) .. (228.53,90.24) .. controls (218.66,88.84) and (210.44,85.73) .. (210.17,83.31) -- cycle ;
\draw  [draw opacity=0][fill={rgb, 255:red, 74; green, 144; blue, 226 }  ,fill opacity=1 ] (216.72,84.77) .. controls (216.56,83.38) and (219.83,82.74) .. (224.02,83.34) .. controls (228.21,83.94) and (231.73,85.54) .. (231.89,86.93) .. controls (232.05,88.31) and (228.78,88.95) .. (224.59,88.35) .. controls (220.4,87.76) and (216.87,86.15) .. (216.72,84.77) -- cycle ;
\draw  [draw opacity=0][fill={rgb, 255:red, 74; green, 144; blue, 226 }  ,fill opacity=0.45 ] (283.14,150.44) .. controls (292.13,150.44) and (315.24,150.44) .. (298.55,172.85) .. controls (281.85,195.26) and (289.56,228.34) .. (283.14,228.09) .. controls (276.72,227.85) and (286.99,197.39) .. (269.01,174.99) .. controls (251.03,152.58) and (274.15,150.44) .. (283.14,150.44) -- cycle ;
\draw  [draw opacity=0][fill={rgb, 255:red, 74; green, 144; blue, 226 }  ,fill opacity=0.45 ] (283.76,168.47) .. controls (292.75,168.47) and (290.76,176.47) .. (287.76,187.47) .. controls (284.76,198.47) and (286.99,227.27) .. (283.14,227.03) .. controls (279.29,226.78) and (283.14,197.67) .. (279.29,188.86) .. controls (275.43,180.05) and (274.77,168.47) .. (283.76,168.47) -- cycle ;
\draw  [draw opacity=0][fill={rgb, 255:red, 74; green, 144; blue, 226 }  ,fill opacity=1 ] (283.45,200.32) .. controls (284.16,200.33) and (284.66,206.07) .. (284.58,213.15) .. controls (284.49,220.23) and (283.85,225.97) .. (283.14,225.96) .. controls (282.43,225.95) and (281.92,220.21) .. (282.01,213.13) .. controls (282.1,206.05) and (282.74,200.32) .. (283.45,200.32) -- cycle ;
\draw (304.33,184.29) node [anchor=north west][inner sep=0.75pt]   [align=left] {Level sets of $\rho$};
\end{tikzpicture}
\caption{Potential scenarios of chemotactic singularity formation in a domain with non-flat boundary geometry.}
\label{fig:secIntro_PKSIPMGeneralBoundary}
\end{figure}
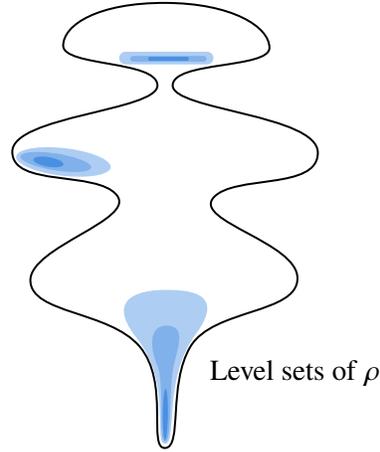

To address these difficulties, we establish a general anisotropic refinement of the classical Nash inequality (\ref{eq:secIntro_Nash}) for a large class of planar bounded domains.  This inequality provides extra control whenever the density is well-mixed in the horizontal direction or, in other words, close to being a stratified function. Furthermore, we can understand the stratification of a solution to the PKS-IPM system (\ref{eq:secIntro_PKSIPM}) via the introduction of a potential energy, see \Cref{subsec:Potential}, whose derivative relates the measurement of stratification to the solution's variance and $\dH^1$ norm. Altogether, we may combine the Nash stratification inequality, the potential energy estimates, and the a-priori variance estimate (\ref{eq:secIntro_AprioriVariance}) in a similar, but more subtle, manner to the discussion above in \Cref{subsec:BackChemo}. Doing so, we obtain uniform in time boundedness of the solution's variance and hence global well-posedness. This approach replaces and conceptualizes many arguments in \cite{HKY25} that heavily rely on the flatness of the periodic strip boundary including Fourier side arguments.

\subsection{Main results and organization}
\label{subsec:Results}

We begin in \Cref{sec:Setting} with a rigorous definition of the large class of two-dimensional domains that we treat in this article. The assumptions are chosen to be compatible with the important phenomena of stretching by gravity and horizontal mixing. Briefly summarized here, we take $\Omega$ to be an open, bounded, simply-connected subset of $\R^2$ with smooth boundary. Most importantly, we further impose that each horizontal cross-section of $\Omega$ be a connected interval. We do not believe that this assumptions is essential; however, treating more general cases introduces technicalities that we decided not to pursue here. We also focus on the harder case where the length of the horizontal cross-sections of $\Omega$ go to zero near the top and bottom of the domain, as in the case of the unit disk. If this does not occur, then the top and bottom portions of the domain boundary are flat which leads to significant simplifications; see \Cref{rem:secIntro_DomainAssumptions}.

On such a domain $\Omega$, we also provide in \Cref{sec:Setting} an explicit construction of the orthogonal projection, in an $L^2$ sense, of a smooth function $f \in C^\infty(\ov{\Omega})$ onto the space of ``stratified functions'': those which depend only on the vertical variable $x_2$. We denote the projection of $f$ by $\ov{f}(x_2)$ and we let $\wt{f}(x_1, x_2) := f(x_1, x_2) - \ov{f}(x_2)$. The following orthogonality relation holds:
\[
\nm{f-f_M}_2^2 = \nm{\ov{f}-f_M}_2^2 + \nm{\wt{f}}_2^2,
\]
where we note that $f_M$ is the mean of $f$ over the domain $\Omega$. In turn, one may consider the proportions
\[
0 \le \frac{\nm{\ov{f}-f_M}_2^2}{\nm{f-f_M}_2^2} \le 1,
\quad
0 \le \frac{\nm{\wt{f}}_2^2}{\nm{f-f_M}_2^2} \le 1,
\]
which sum to one. If the left-hand quantity is close to one, the function $f(x_1,x_2)$ is ``highly stratified''. If instead the right-hand quantity is close to one, the function $f(x_1, x_2)$ is ``highly unstratified''. This gives a rigorous method for discussing ``how stratified'' a smooth function $f \in C^\infty(\Omega)$ is. Lastly, \Cref{sec:Setting} also contains a summary of the function spaces and notation used throughout the entirety of this article.

The above framework suggests that, for functions which are highly stratified, one may expect an improvement in the scaling of the classical Nash inequality (\ref{eq:secIntro_Nash}). Indeed, in \Cref{sec:Nash}, we prove \Cref{thm:secNash_NashOnStratified} which provides control of $\nm{\ov{f}-f_M}_2$ by $\nm{f-f_M}_1$ and $\nm{\grad f}_2$ with a scaling corresponding to the formal choice of $d = 3/2$. However, it is unreasonable to expect an improved scaling for functions that are far from being stratified. Instead, we show in \Cref{thm:secNash_NashOnUnstratified} that $\nm{\wt{f}}_2$ obeys a Nash-type inequality with control by $\nm{f-f_M}_1$, $\nm{\grad f}_2$, and also the mixing norm $\dH_0\pre$ of $\partial_1 f$. Combining these Nash inequalities on $\ov{f}$ and $\wt{f}$, we are able to prove our first main theorem.

\begin{restatable}{theorem}{NashTotal}
\label{thm:secIntro_NashTotal}
Fix a domain $\Omega \subseteq \R^2$ satisfying \Cref{as:secSett_1,as:secSett_2,as:secSett_3}, as stated in \Cref{subsec:Domain}, and fix $f \in C^\infty(\ov{\Omega})$. Then,
\[
\nm{f-f_M}_2^2
\le
C_0 \nm{f-f_M}_1^{\frac{8}{7}}\nm{\grad f}_2^{\frac{6}{7}}
+ C_0 \nm{\partial_1 f}_{\dH_0\pre}^{\frac{8}{987}} \nm{f-f_M}_1^{\frac{416}{987}}\nm{f-f_M}_2^{\frac{54}{47}}\nm{\grad f}_2^{\frac{416}{987}},
\]
where $C_0 = C_0(\Omega) > 0$ is a universal constant. It immediately follows that
\[
\nm{f-f_M}_2^2
\le
C_0 \nm{f-f_M}_1^{\frac{8}{7}}\nm{\grad f}_2^{\frac{6}{7}}
+ C_1 \nm{\partial_1 f}_{\dH_0\pre}^{\frac{8}{987}} \nm{f-f_M}_1^{\frac{983}{987}}\nm{\grad f}_2^{\frac{983}{987}},
\]
by the classical Nash inequality (\ref{eq:secIntro_Nash}), where $C_1 = C_1(\Omega) > 0$.
\end{restatable}

\Cref{thm:secIntro_NashTotal} provides the desired quantitative improvement of the two-dimensional classical Nash inequality (\ref{eq:secIntro_Nash}). We conclude \Cref{sec:Nash} with a discussion on the domain assumptions, on generalizing the inequality, and on the sharpness of the inequality; see \Cref{rem:secIntro_DomainAssumptions}.

In \Cref{sec:GWP}, we begin with a review of the easily verified properties of solutions to the PKS-IPM system (\ref{eq:secIntro_PKSIPM}). In \Cref{subsec:Variance}, we prove the a-priori variance inequality (\ref{eq:secIntro_AprioriVariance}), see \Cref{cor:secGWP_AprioriVarianceIneq}, and that the variance $\nm{\rho(\cdot, t) - \rho_M}_2^2$ is a control quantity for global regularity; see \Cref{thm:secGWP_LWP}. We then introduce the potential energy of a solution in \Cref{subsec:Potential}. It is easily observed that the potential energy is bounded and that its time derivative relates the $\dH_0\pre$ norm of $\rho$, the $\dH^1$ norm of $\rho$, and the variance. Collecting the Nash stratification inequality \Cref{thm:secIntro_NashTotal}, the a-priori variance inequality (\ref{eq:secIntro_AprioriVariance}), and the potential energy estimates provides a system of differential inequalities that, with a short ODE style argument, leads to our second main result.

\begin{restatable}{theorem}{GWP}
\label{thm:secIntro_GWP}
Fix a domain $\Omega \subseteq \R^2$ satisfying \Cref{as:secSett_1,as:secSett_2,as:secSett_3}, as stated in \Cref{subsec:Domain}, and fix any choice of gravitational constant $g \neq 0$. Then, for any nonnegative initial data $\rho_0 \in C^\infty(\ov{\Omega})$ with $\partial_n \rho|_{\partial \Omega} = 0$, the corresponding unique solution $\rho(x,t)$ to (\ref{eq:secIntro_PKSIPM}) is nonnegative and exists for all time in the space $C^\infty(\ov{\Omega} \times [0,\infty))$. Furthermore, the variance supremum satisfies
\[
\sup_{0 \le t < \infty} \nm{\rho(\cdot, t) - \rho_M}_2^2
\le C_2 \max\curly*{\nm{\rho_0 - \rho_M}_2^2,\, C_3},
\]
where $C_2, C_3 > 0$ are finite constants depending only on $\Omega$, $|g|$, and $\rho_M$.
\end{restatable}

Thus, the global regularity result of \cite{HKY25} for the periodic strip indeed generalizes to a much larger class of bounded domains, including the disk. This reveals the PKS-IPM system (\ref{eq:secIntro_PKSIPM}) to be an example where regularization by fluid flow is quite a robust phenomenon, despite the domain $\Omega$ being allowed to have boundary regions of very large curvature or horizontal cross-sections with thin width. The study of fluid advection as a regularization mechanism relates to many active areas of research. For now, we conclude this section with some remarks on the main theorems.

\begin{remark}
We emphasize that $f$ does not need to satisfy any sign or boundary conditions for \Cref{thm:secIntro_NashTotal} to hold true, nor need it be a solution of a particular differential equation. Also, establishing an analogue of this inequality in dimension $d \ge 3$ remains an open and interesting question.
\end{remark}

\begin{remark}\label{rem:secIntro_DomainAssumptions}
In the above, we have focused on domains with connected cross-sections with length going to zero when approaching the top and bottom of the domain. Recall that these assumptions will be more precisely stated in \Cref{sec:Setting}. The cases where the top or bottom domain is flat follows more easily, which we discuss in more detail in \Cref{subsubsec:NashDiscussion}. There, we also remark upon other ways of extending the assumptions in \Cref{sec:Setting} --- many of which are imposed simply for the nicer presentation of arguments. Moreover, we do not claim sharpness of the constants nor exponents in the inequalities given by \Cref{thm:secIntro_NashTotal}. We also discuss such questions in \Cref{subsubsec:NashDiscussion}.
\end{remark}

\begin{remark}
The unit disk is an important example of a domain satisfying \Cref{as:secSett_1,as:secSett_2,as:secSett_3}, stated in \Cref{subsec:Domain}. We note that many finite-time blow-up solutions of the parabolic-elliptic PKS equation (\ref{eq:secIntro_PKSGeneral},\ref{eq:secIntro_PKSeqC}) have been constructed on the disk; see for example \cite{HV96b}. Moreover, these singularity formation scenarios can occur on more general domains as well \cites{JL92,N95,N01}. Therefore, \Cref{thm:secIntro_GWP} indeed offers an example of singularity suppression due to fluid advection.
\end{remark}

\begin{remark}\label{rem:secIntro_OtherContexts}
The $\dH_0\pre$ seminorm of $\partial_1 f$ naturally appears when studying the potential energy associated to a solution of the incompressible porous media equation, as first noticed in the work \cite{KY23} on small scale creation for the IPM equation. This led the authors of \cite{HKY25} on the PKS-IPM system for the periodic strip to consider potential energy as well, as a key object in the analysis. We do the same, see \Cref{subsec:Potential}, in order to establish the boundedness of the solution's variance in \Cref{thm:secIntro_GWP}. Many other very recent works have also leveraged similar potential energy structures, including \cites{BJPW25,KPY25,P25} for the IPM and Boussinesq equations. We expect then that \Cref{thm:secIntro_NashTotal} (along with the inequalities given by \Cref{thm:secNash_NashOnStratified} for $\ov{f}$ and \Cref{thm:secNash_NashOnUnstratified} for $\wt{f}$) may turn out to be useful when working in those settings.
\end{remark}

\begin{remark}
\Cref{thm:secIntro_GWP} is the first result, to the authors' knowledge, that provides a robust example of singularity suppression by active fluid advection in regards to boundary geometry of the domain. Indeed, we emphasize that the initial data to (\ref{eq:secIntro_PKSIPM}) may be arbitrarily large; the coupling constant $g \neq 0$ may be arbitrarily small; the limiting dynamics are not suspected to be simple; and the domain may contain arbitrarily thin cross sections or boundary regions of arbitrarily large curvature. Interestingly, the usage of \Cref{thm:secIntro_NashTotal} reduces the proof of \Cref{thm:secIntro_GWP} to a system of four integro-differential inequalities that can be handled with quite general ODE arguments: see \Cref{prop:secGWP_GeneralArgument}.
\end{remark}

\section{Setting and preliminaries}
\label{sec:Setting}

The goal for \Cref{sec:Setting} is to provide a framework for decomposing any smooth function $f(x_1,x_2)$ into the sum of a stratified function $\ov{f}(x_2)$ and a function $\wt{f}(x_1,x_2)$ that is mean-zero on each horizontal cross-section of the domain. To do so, we require some assumptions on the domain $\Omega$ that are explained in detail in \Cref{subsec:Domain}. Then, \Cref{subsec:FuncDecomp} provides the formal construction of $\ov{f}$ and $\wt{f}$ along with some basic properties that will be used repeatedly throughout the remainder of this work. Lastly, the commonly used function spaces along with other important notation are summarized in  \Cref{subsec:Notation}, for the reader's ease of reference.

\subsection{Domain assumptions and normalizations}
\label{subsec:Domain}

The theorems listed in \Cref{subsec:Results} hold true for a domain $\Omega \subseteq \R^2$ satisfying the following assumptions.

\begin{assumption}\label{as:secSett_1}
$\Omega$ is a nonempty, open, bounded, and connected subset of $\R^2$ with smooth boundary $\partial \Omega$. We denote the lower-most height and upper-most height of the domain $\Omega$ by
\[
h_L := \inf\curly{x_2\ |\ x \in \Omega},
\quad
h_U := \sup\curly{x_2\ |\ x \in \Omega},
\]
where, of course, $-\infty < h_L < h_U < \infty$.
\end{assumption}

\begin{assumption}\label{as:secSett_2}
Given $\Omega$ satisfying \Cref{as:secSett_1}, there further exist smooth charts, denoted $F_\ell, F_r \in C^\infty((h_L, h_U)) \cap C([h_L,h_U])$, that partition the boundary $\partial \Omega$ into left and right sides. Rigorously, we mean that
\[
\Omega = \curly*{x \in \R^2\ |\ h_L < x_2 < h_U,\ F_\ell(x_2) < x_1 < F_r(x_2)},
\]
and also that
\[
a := F_\ell(h_L) = F_r(h_L),
\quad
b := F_\ell(h_U) = F_r(h_U),
\quad
F_\ell(x_2) < F_r(x_2),
\]
for all $h_L < x_2 < h_U$. Note that the derivatives $F_\ell'(x_2)$ and $F_r'(x_2)$ are finite for all $h_L < x_2 < h_U$ while the derivatives $F_\ell'(h_L)$, $F_r'(h_L)$, $F_\ell'(h_U)$, $F_r'(h_U)$ must all be infinite.
\end{assumption}

\begin{definition}
Given $\Omega$ satisfying \Cref{as:secSett_1,as:secSett_2}, we denote the horizontal cross-section of $\Omega$ at height $h_L < x_2 < h_U$ by
\[
I_{x_2} := \{x_1 \in \R\ |\ F_\ell(x_2) < x_1 < F_r(x_2) \}.
\]
Note $I_{x_2}$ is a connected interval of length $|I_{x_2}| = F_r(x_2)-F_\ell(x_2)$, which is smoothly varying as $x_2$ varies in $(h_L,h_U).$
\end{definition}

\begin{definition}
Suppose $\Omega \subseteq \R^2$ satisfies \Cref{as:secSett_1,as:secSett_2}. Fix any $0 < \eps < (h_U-h_L)/2$. We define
\[
L_\eps := \curly{x \in \Omega\ |\ x_2 \in (h_L,h_L+\eps)},
~
U_\eps := \curly{x \in \Omega\ |\ x_2 \in (h_U-\eps, h_U)},
~
\Omega_\eps := \Omega \setminus \paren*{L_\eps \cup U_\eps}.
\]
Thus, $L_\eps$ and $U_\eps$ are the lower and upper $\eps$-caps of $\Omega$ while $\Omega_\eps$ is the bulk of the domain.
\end{definition}

\begin{assumption}\label{as:secSett_3}
Given $\Omega$ satisfying \Cref{as:secSett_1,as:secSett_2}, there further exists $0 < \eps_* < (h_U-h_L)/2$ sufficiently small so that the smooth chart $x_2 = F_L(x_1) \in C^\infty$ for the lower $\eps_*$-cap boundary $\partial L_{\eps_*} \cap \partial \Omega$ admits the Taylor expansion
\[
F_L(x_1) = h_L + K_L (x_1-a)^{n_L} + O((x_1-a)^{n_L+1}),
\]
for an even integer $n_L \ge 2$ and some constant $K_L > 0$. Similarly, the smooth chart $x_2 = F_U(x_1) \in C^\infty$ for the upper $\eps_*$-cap boundary $\partial U_{\eps_*} \cap \partial \Omega$ admits the Taylor expansion
\[
F_U(x_1) = h_U-K_U (x_1-b)^{n_U} + O((x_1-b)^{n_U+1}),
\]
for an even integer $n_U \ge 2$ and some constant $K_U > 0$.
\end{assumption}

Every result in \Cref{sec:Nash} will hold for $\Omega \subseteq \R^2$ satisfying \Cref{as:secSett_1,as:secSett_2,as:secSett_3}. However, we will usually also impose the following further normalizations on $\Omega$, without loss of generality, in order to simplify proofs.

\begin{normalization}\label{norm:secSett_h}
Suppose $\Omega$ satisfies \Cref{as:secSett_1,as:secSett_2,as:secSett_3}. By translation of $\Omega$, we may set $h_L = 0$ and $a = 0$. We also relabel $h := h_U > 0$. This will be useful as it gives
\[
\int_\Omega f(x)\, dx = \int_0^h \int_{I_{x_2}} f(x)\, dx_1\, dx_2
\]
for any reasonable function $f$ on $\Omega$, among many other similar integral decompositions.
\end{normalization}

\begin{normalization}\label{norm:secSett_eps}
Given $\Omega$ satisfying \Cref{as:secSett_1,as:secSett_2,as:secSett_3} and \Cref{norm:secSett_h}, we may further shrink $\eps_*$ so that $0 < \eps_* < h/4$ and moreover we have that
\[
|I_{\eps_*}| = \min\curly*{ |I_{x_2}|\ |\  \eps_* \le x_2 \le 3h/4},
\quad
|I_{h-\eps_*}| = \min\curly*{ |I_{x_2}|\ |\ h/4 \le x_2 \le h-\eps_*}.
\]
We again shrink $\eps_*$ to ensure $0 < \eps_* < 1$, so that $\eps_*^{1/n_U},\ \eps_*^{1/n_L} \ge \eps_*^{1/2}$. For a final time, we shrink $\eps_*$ so that the map $x_2 \mapsto |I_{x_2}|$ is increasing on $(0,\eps_*)$ and decreasing on $(h-\eps_*, h)$. This is possible due to Assumption \ref{as:secSett_3}. In particular, there exists a finite constant $K_* = K_*(\Omega) > 0$ (that can be expressed in terms of $n_L$, $K_L$, $n_U$, $K_U$) such that
\begin{equation}\label{eq:secSett_DecayOfIx2}
|I_{x_2}| = F_r(x_2) - F_\ell(x_2) \ge K_*\eps^{1/2},
\end{equation}
for any $0 < \eps \le \eps_*$ and any $\eps < x_2 < h-\eps$.
\end{normalization}

Lastly, we define some further constants that will be useful for estimates across \Cref{sec:Nash}.

\begin{definition}\label{def:secSett_K*R*}
Given $\Omega$ satisfying \Cref{as:secSett_1,as:secSett_2,as:secSett_3} as well as \Cref{norm:secSett_h,norm:secSett_eps}, we define two further constants associated to the bulk
\[ \Omega_{\eps_*/4} = \left\{ x=(x_1,x_2) \in \Omega\ \left|\ \frac{\eps_*}{4} \leq x_2 \leq h - \frac{\eps_*}{4} \right. \right\}. \]
We write
\begin{equation}\label{aux0307a}
R_*
:=
\frac{
\max \curly*{|I_{x_2}|\ |\ x=(x_1,x_2) \in \Omega_{\eps_*/4}}
}
{
\min\curly*{ |I_{x_2}|\ |\ x=(x_1,x_2) \in \Omega_{\eps_*/4}}
}
\end{equation}
to denote the ratio of the longest and shortest horizontal cross-sections of $\Omega_{\eps_*/4}$. We also write
\begin{equation}\label{aux128a}
M_* :=
\max\curly{
1,\ |F_r'(x_2)|,\ |F_\ell'(x_2)|\ |\
x=(x_1,x_2) \in \Omega_{\eps_*/4}
}
\end{equation}
to denote the maximal derivative of the smooth charts for the bulk boundary $\partial \Omega_{\eps_*/4} \cap \partial \Omega$. Note that it holds $1 \le R_* = R_*(\Omega) < \infty$ and also $1 \le M_* = M_*(\Omega) < \infty$.
\end{definition}

\subsection{Function decomposition into stratified and unstratified components}
\label{subsec:FuncDecomp}

We work under the following hypotheses in \Cref{subsec:FuncDecomp}.

\begin{hypotheses}
\label{hyp:subsecFuncDecomp}
We fix a domain $\Omega \subseteq \R^2$ satisfying \Cref{as:secSett_1,as:secSett_2,as:secSett_3} as well as \Cref{norm:secSett_h,norm:secSett_eps} and fix a smooth function $f \in C^\infty(\ov{\Omega})$. We use the notation defined in \Cref{subsec:Domain}.
\end{hypotheses}

We are interested in providing a notion for ``how stratified'' the function $f(x_1,x_2)$, of two variables, may be. To do so, we should associate to $f(x_1,x_2)$ another function $\ov{f}(x_2)$ that is defined on $\Omega$ but depends only on the second variable. A natural way of doing so is to average out the first variable as follows.

\begin{definition}
Under \Cref{hyp:subsecFuncDecomp}, we define the stratified component of $f \in C^\infty(\ov{\Omega})$ to be the (single-variable) function $\ov{f} \in C^\infty((0,h))$ where
\[
\ov{f}(x_2) := \frac{1}{|I_{x_2}|}\int_{I_{x_2}}f(x_1,x_2)\, dx_1,
\]
for all $0 < x_2 < h$. Note that $|I_{x_2}| > 0$ for any $0 < x_2 < h$ by \Cref{as:secSett_2}. Also note that, generally, $\ov{f} \notin C^\infty(\ov{\Omega})$ as it can have non-smooth behavior near $x_2=0$ and $x_2=h$.
\end{definition}

The proofs of the following and a few subsequent lemmas are immediate and will be omitted.

\begin{lemma}
Under \Cref{hyp:subsecFuncDecomp}, the following hold true.
\begin{enumerate}[label={(\alph*)}]
\item If $f(x_1, x_2) = f(x_2)$ originally depends only on the second variable, then $\ov{f} = f$. In particular, $\ov{c} = c$ for any constant $c \in \R$.
\item $\ov{f} \in L^\infty(\Omega) \subseteq L^2(\Omega)$ with $|\ov{f}(x_2)| \le \nm{f}_\infty$ for all $0 < x_2 < h$.
\item For any $0 < x_2 < h$,
\[
\int_{I_{x_2}} \ov{f}(x_2)\, dx_1 = \int_{I_{x_2}} f(x_1,x_2)\, dx_1,
\quad
\int_\Omega \ov{f}(x_2)\, dx = \int_\Omega f(x_1,x_2)\, dx.
\]
We note however that the equality
\[
\int_{0}^{h} \ov{f}(x_2)\, dx_2 = \int_\Omega f(x_1,x_2)\, dx
\]
does not necessarily hold true, which is a reason why we typically treat $\ov{f}$ as a function defined on $\Omega$.
\item The map $f \mapsto \ov{f}$ from $C^\infty(\ov{\Omega})$ into $L^2(\Omega)$, where both function spaces are endowed with the standard $L^2$ inner product, is a continuous and orthogonal linear projection.
\end{enumerate}
\end{lemma}

The fact that $f \mapsto \ov{f}$ is an orthogonal projection gives credence to the interpretation of $\ov{f}(x_2)$ being the ``closest'' stratified function to $f(x_1,x_2)$. Moreover, it encourages the following definition and observations.

\begin{definition}
Under \Cref{hyp:subsecFuncDecomp}, we define the unstratified component of $f$ to be the function $\wt{f} \in C^\infty(\Omega)$ where
\[
\wt{f}(x_1,x_2) := f(x_1,x_2) - \ov{f}(x_2)
\]
for any $(x_1, x_2) \in \Omega$.
\end{definition}

\begin{lemma}
Under \Cref{hyp:subsecFuncDecomp}, the following hold true.
\begin{enumerate}[label={(\alph*)}]
\item If $f(x_1, x_2) = f(x_2)$ originally depends only on the second variable, then $\wt{f} = 0$. In particular, $\wt{c} = 0$ for any constant $c \in \R$.
\item $\wt{f} \in L^\infty(\Omega) \subseteq L^2(\Omega)$ with $|\wt{f}(x_2)| \le |f(x_1,x_2)| + |\ov{f}(x_2)| \le 2\nm{f}_\infty$ for all $x \in \Omega$.
\item For any $0 < x_2 < h$,
\[
\int_{I_{x_2}} \wt{f}(x_2)\, dx_1 = 0,
\quad
\int_\Omega \wt{f}(x_2)\, dx = 0.
\]
\item The map $f \mapsto \wt{f}$ from $C^\infty(\ov{\Omega})$ into $L^2(\Omega)$, where both function spaces are endowed with the standard $L^2$ inner product, is also a continuous and orthogonal linear projection.
\end{enumerate}
\end{lemma}

We have an $L^2$-orthogonal decomposition $f(x_1,x_2) = \ov{f}(x_2) + \wt{f}(x_1,x_2)$, where $\ov{f}$ is stratified and $\wt{f}$ can be considered ``unstratified''. Here, a function being unstratified should be understood as that function having zero-mean over each horizontal cross-section of $\Omega$. The orthogonality of this decomposition is crucial for the remainder of this article, with it implying the following fact.

\begin{definition}
Under \Cref{hyp:subsecFuncDecomp}, we write
\[
f_M := \frac{1}{|\Omega|}\int_\Omega f(x)\, dx
\]
to denote the mean of $f$ over $\Omega$.
\end{definition}

\begin{lemma}
Under \Cref{hyp:subsecFuncDecomp}, the following $L^2(I_{x_2})$-orthogonal decomposition holds
\[
\nm{f(\cdot,x_2)-f_M}_{L^2(I_{x_2})}^2 = |I_{x_2}||\ov{f}(x_2)-f_M|^2 + \nm{\wt{f}(\cdot,x_2)}_{L^2(I_{x_2})}^2
\]
for any $0 < x_2 < h$. It follows that
\[
\nm{f-f_M}_{L^2(\Omega')}^2 = \nm{\ov{f}-f_M}_{L^2(\Omega')}^2+\nm{\wt{f}}_{L^2(\Omega')}^2,
\]
where $\Omega'$ may be set equal to the full domain $\Omega$, the $\eps$-caps $L_\eps$, $U_\eps$, or the $\eps$-bulk $\Omega_\eps$, for any $0 < \eps < \eps_*$.
\end{lemma}

This motivates, for a non-constant function $f$, the quantities
\[
0 \le \frac{\nm{\ov{f}-f_M}_2^2}{\nm{f-f_M}_2^2} \le 1,
\quad
0 \le \frac{\nm{\wt{f}}_2^2}{\nm{f-f_M}_2^2} \le 1,
\]
as methods of measuring the ``proportion'' of $f$ that is stratified and unstratified. Indeed, these proportions lie between $0$ and $1$ and always add up to $1$. Furthermore, these quantities are invariant under the transformation $f \mapsto c_1 f + c_2$ for constants $c_1,c_2 \in \R$. Thus, we can think of $f$ being ``highly stratified'' when $\nm{\ov{f}-f_M}_2^2/\nm{f-f_M}_2^2$ is large and $\nm{\wt{f}}_2^2/\nm{f-f_M}_2^2$ is small. Analogously, $f$ is ``highly unstratified'' when $\nm{\ov{f}-f_M}_2^2/\nm{f-f_M}_2^2$ is small and $\nm{\wt{f}}_2^2/\nm{f-f_M}_2^2$ is large.

We conclude by recording here some further useful observations concerning averages of $f$, $\ov{f}$, and $\wt{f}$.

\begin{lemma}
Under \Cref{hyp:subsecFuncDecomp}, the following hold true.
\begin{enumerate}[label={(\alph*)}]
\item In general, $\nm{f_M}_1 = |\Omega||f_M| \le \nm{f}_1$ and $\nm{\ov{f}}_1 \le \nm{f}_1$. If $f \ge 0$ is originally nonnegative, then $\nm{f_M}_1 = \nm{f}_1$ and $\nm{\ov{f}}_1 = \nm{f}_1$.
\item It holds $\nm{\wt{f}}_1 = \nm{f-\ov{f}}_1 \le 2\nm{f}_1$ and hence $\nm{\wt{f}}_1 \le 2\nm{f-f_M}_1$, since $\wt{c} = 0$ for any constant $c \in \R$.
\end{enumerate}
\end{lemma}

We note that the proofs of the lemmas in these sections are all immediate. Moreover, as these observations are used so frequently throughout the remainder of this work, we will not explicitly cite them.

\subsection{Function spaces and notation}
\label{subsec:Notation}

In this article, we work with a spatial domain $\Omega \subseteq \R^2$ that is always open, bounded, and simply connected with closure and boundary denoted by $\ov{\Omega}$ and $\partial \Omega$, respectively. The Cartesian coordinates of a point $x \in \Omega$ are denoted by $(x_1, x_2)$, with partial derivatives in those directions denoted by $\partial_1, \partial_2$. The outward pointing unit normal at a point $x \in \partial \Omega$ is denoted by $n(x)$, with the normal derivative at that point denoted by $\partial_n$. The perpendicular gradient is $\grad^\perp := (-\partial_2, \partial_1)$. The Laplacian on $\Omega$ is written as $\lap := \partial_1^2 + \partial_2^2$ with Dirichlet inverse $(-\lap_D)\pre$ and Neumann inverse $(-\lap_N)\pre$. The two-dimensional Lebesgue measure of any subset $A$ of $\Omega$ is denoted $|A|$, while the one-dimensional analogue for or any subset $S$ of heights $[h_L, h_U]$ is also denoted as $|S|$.

We now list the function spaces and corresponding (semi-)norms that are frequently used throughout the article. We drop $\Omega$ in the notation of any (semi-)norms. Whenever working with functions defined over anything other than $\Omega$, such as $I_{x_2}$, $U_\eps$, $L_\eps$, or $\Omega_\eps$, we will write the relevant subset in the notation for the (semi-)norms. The space of infinitely differentiable functions which extend to the boundary is written as $C^\infty(\ov{\Omega})$. We write $L^p(\Omega)$ with norm $\nm{f}_p$ to denote the standard $p$-Lebesgue space for any $1 \le p \le \infty$, with the $L^2(\Omega)$ inner product written as $\gen{f, g}$. The standard $L^2(\Omega)$-based Sobolev space of once weakly differentiable functions is denoted by $H^1(\Omega)$. If only the weak derivative $f$ is known to be in $L^2(\Omega)$, and not necessarily the function itself, we write $f \in \dH^1(\Omega)$ with the corresponding homogeneous seminorm being given by $\nm{f}_{\dH^1} := \nm{\grad f}_2$. If $f \in \dH^1(\Omega)$ has vanishing trace, we instead define the Laplacian-based seminorm
\[
\nm{f}_{\dH_0^1}^2 := \int_\Omega f(x) (-\lap) f(x)\, dx = \nm{(-\lap_D)^{1/2}f}_2^2.
\]
In turn, we may define the negative Sobolev norm:
\[
\nm{f}_{\dH_0\pre}^2 := \int_\Omega f(x) (-\lap_D)\pre f(x)\, dx = \nm{(-\lap_D)^{-1/2}f}_2^2
\]
for any $f \in L^2(\Omega)$. This negative Sobolev norm satisfies the following important variational or dual formulation.

\begin{lemma}
We have
\begin{equation}\label{eq:secSett_HMixVariational}
\nm{f}_{\dH_0\pre} = \sup_{\psi \in \dH_0^1\setminus\{0\}} \frac{|\gen{f,\psi}|}{\nm{\psi}_{\dH_0^1}},
\end{equation}
for any $f \in C^\infty(\ov{\Omega})$.
\end{lemma}

The proof is straightforward and will be omitted.

All of the results in \Cref{sec:Nash} and some results in \Cref{sec:GWP} hold for general functions we denote by $f \in C^\infty(\ov{\Omega})$. For such $f$, we often write
\[
\Gamma := \nm{f-f_M}_2^2,
\quad
\Lambda := \Gamma^{-3/2} \sup_{x_2 \in [0,h]} \nm{\wt{f}(\cdot,x_2)}_{L^2(I_{x_2})}^2.
\]
We use $0 < \kappa \le 1$ to denote a constant lower-bounding either the stratified proportion $\nm{\ov{f}-f_M}_2^2\Gamma\pre$ or the unstratified proportion $\nm{\wt{f}}_2^2\Gamma\pre$ and we use $\lambda$ to denote a parameter such that either $\lambda < \Lambda$, or $\Lambda < \lambda\pre$, or $\lambda \le \Lambda \le \lambda\pre$. When proving \Cref{thm:secNash_NashOnUnstratified}, we will use parameters $\alpha$, $\beta$, $\sigma$, $\eta$ depending on $\kappa$, $\lambda$, $\Lambda$ whose definitions can be found in \Cref{subsubsec:NashUnstratHardCasesPart1} and again in \Cref{subsubsec:NashUnstratHardCasesPart2}.

Any solution to the PKS-IPM equation (\ref{eq:secIntro_PKSIPM}) is denoted by $\rho$, with corresponding chemical concentration $c$ and fluid velocity $u$ with internal pressure $p$. The constants arising from the $d$-dimensional classical Nash inequality (\ref{eq:secIntro_Nash}) are denoted $C_{\mathrm{N},d}$. Throughout \Cref{sec:Nash,sec:GWP}, we often label universal constants $C_0$, $C_1$, $C_2$, and so on, which may only depend on the domain $\Omega$ and may change from line to line. Sometimes we suppress constants, writing $A \lesssim B$ to mean there exists a universal constant $C > 0$ such that $A \le CB$. If that implicit universal constant depends on some parameter(s) $s$, we instead write $A \lesssim_s B$.

\section{Anisotropic Nash inequalities that quantify stratification}
\label{sec:Nash}

Recall that the dynamics of a solution to the PKS-IPM system (\ref{eq:secIntro_PKSIPM}) have a preferential direction: the vertical, due to the force of gravity. This motivates the goal for \Cref{sec:Nash}, as we now prove two anisotropic refinements of the two-dimensional classical Nash inequality (\ref{eq:secIntro_Nash}) that are designed to better capture ``how vertically stratified'' a function $f \in C^\infty(\ov{\Omega})$ is. \Cref{subsec:StratNash} provides a refinement for the stratified component $\ov{f}$ while \Cref{subsec:UnstratNash} provides a refinement for the unstratified component $\wt{f}$. Finally, in \Cref{subsec:TotalNash}, we combine the inequalities on $\ov{f}$ and $\wt{f}$ to obtain \Cref{thm:secIntro_NashTotal} and conclude with a discussion of the generality and sharpness of the results. In particular, we comment on how the proofs below can be simplified when the top and bottom portions of $\partial\Omega$ are both flat. Throughout this section, we frequently work under the following hypotheses.

\begin{hypotheses}
\label{hyp:secNash}
We fix a domain $\Omega \subseteq \R^2$ satisfying \Cref{as:secSett_1,as:secSett_2,as:secSett_3} as well as \Cref{norm:secSett_h,norm:secSett_eps}. We also fix a function $f \in C^\infty(\ov{\Omega})$ for which we write
\begin{equation}\label{aux129a}
\Gamma := \nm{f-f_M}_2^2, \qquad \Lambda := \Gamma^{-3/2} \sup_{x_2 \in [0,h]} \nm{\wt{f}(\cdot,x_2)}_{L^2(I_{x_2})}^2.
\end{equation}
We use the notation defined in \Cref{sec:Setting}.
\end{hypotheses}

\subsection{Nash inequality on the stratified component}
\label{subsec:StratNash}

We now work to establish the following inequality:
\[
\nm{\ov{f}-f_M}_2^{8/3} \lesssim_\Omega \nm{f-f_M}_2^{1/3} \nm{f-f_M}_1^{4/3} \nm{\grad f}_2,
\]
which is precisely stated below in \Cref{thm:secNash_NashOnStratified}. This inequality can be understood as an improvement of the scaling for the two-dimensional classical Nash inequality (\ref{eq:secIntro_Nash}) with a function-dependent coefficient that worsens when $f$ is far from being stratified; see \Cref{rem:secIntro_NashStratInterp1,rem:secIntro_NashStratInterp2}.

\subsubsection{Proof of theorem and remarks}
\label{subsubsec:StratNashProof}

Recall the orthogonal decomposition
\[
\nm{f-f_M}_2^2 = \nm{\ov{f}-f_M}_2^2 + \nm{\wt{f}}_2^2.
\]
This motivates studying the concentration of the quantity $\nm{\ov{f}-f_M}_2^2$ within the domain. That is, an a-priori bound of the form
\[
\nm{\ov{f}-f_M}_{L^2(\Omega')}^2 \gtrsim \nm{f-f_M}_2^2
\]
for some subset $\Omega' \subseteq \Omega$ gives quantitative information on ``how stratified'' $f$ is in the region $\Omega'$. One can exploit such information to prove a Nash-type inequality. We do so by considering first the concentration of $\nm{\ov{f}-f_M}_2^2$ to the upper or lower caps of the domain.

\begin{restatable}{proposition}{StratCaps}%
\label{prop:secNash_StratCaps}%
Under \Cref{hyp:secNash}, suppose that
\[
\nm{\ov{f}-f_M}_{L^2(L_\eps)}^2 \ge \kappa \Gamma/4,
\]
for some fixed $0 < \kappa \le 1$ and $0 < \eps < \kappa \eps_*/32$. Then, we have
\[
\kappa \Gamma^{1/2} \eps^{-1} \le C_4\nm{\partial_2 f}_2,
\]
for some universal constant $C_4 > 0$, independent even from $\Omega$. The same result holds (with the same universal constant $C_4$) upon replacing the lower $\eps$-cap $L_\eps$ by the upper $\eps$-cap $U_\eps$.
\end{restatable}

If instead $\nm{\ov{f}-f_M}_2^2$ is concentrated in the bulk of the domain, the following inequality is true.

\begin{restatable}{proposition}{StratBulk}
\label{prop:secNash_StratBulk}
Under \Cref{hyp:secNash}, suppose that
\[
\kappa\pre \Gamma\pre \eps^{-1/2} \nm{f-f_M}_1^2 \le \frac{hK_*}{32 R_*}, \qquad \nm{\ov{f}-f_M}_{L^2(\Omega_\eps)}^2 \ge \kappa\Gamma/2,
\]
for some fixed $0 < \kappa \le 1$ and $0 < \eps < \eps_*$. Then
\[
\kappa^{3/2} \Gamma^{3/2} \eps^{1/2} \le C_5\nm{f-f_M}_1^2\nm{\grad f}_2,
\]
for some universal constant $C_5 = C_5(\Omega) > 0$.
\end{restatable}

We reserve the proofs of \Cref{prop:secNash_StratCaps} and \Cref{prop:secNash_StratBulk} to \Cref{subsubsec:NashStratProps}. For now, we leverage these propositions by choosing $\eps > 0$ so that the two inequalities match to give the following desired result.

\begin{theorem}
\label{thm:secNash_NashOnStratified}
Assume \Cref{hyp:secNash}. Then
\begin{equation}\label{aux129b}
\nm{\ov{f}-f_M}_2^{8/3} \le C_6 \nm{f-f_M}_2^{1/3} \nm{f-f_M}_1^{4/3} \nm{\grad f}_2,
\end{equation}
where $C_6 = C_6(\Omega) > 0$ is a universal constant.
\end{theorem}

\begin{proof}
Note that, of course, the desired inequality holds trivially if $\nm{\ov{f}-f_M}_2 = 0$. Suppose then that $0 < \nm{\ov{f}-f_M}_2$ and define
\[
\kappa := \nm{\ov{f}-f_M}_2^2 \Gamma\pre = \nm{\ov{f}-f_M}_2^2\nm{f-f_M}_2^{-2}.
\]
Then $0 < \kappa \le 1$ and trivially $\nm{\ov{f}-f_M}_2^2 \ge \kappa \Gamma$. Moreover, for every $\eps > 0$, exactly one of the following holds true: either
\begin{equation}\label{eq:secNash_StratFCapConcentrated}
\nm{\ov{f}-f_M}_{L^2(L_\eps)}^2 \ge \kappa \Gamma/4
\quad \text{or} \quad
\nm{\ov{f}-f_M}_{L^2(U_\eps)}^2 \ge \kappa \Gamma/4,
\end{equation}
or
\begin{equation}\label{eq:secNash_StratFBulkConcentrated}
\nm{\ov{f}-f_M}_{L^2(\Omega_\eps)}^2 \ge \kappa \Gamma/2.
\end{equation}
In the case of (\ref{eq:secNash_StratFCapConcentrated}), we will employ \Cref{prop:secNash_StratCaps} to obtain a Nash inequality. If instead we have (\ref{eq:secNash_StratFBulkConcentrated}), we will employ \Cref{prop:secNash_StratBulk}. In order to obtain the same Nash inequality no matter whether (\ref{eq:secNash_StratFCapConcentrated}) or (\ref{eq:secNash_StratFBulkConcentrated}) holds, we will define
\[
\eps := \kappa^{-1/3}\Gamma^{-2/3}\nm{f-f_M}_1^{4/3} > 0
\]
(note that if $\|f-f_M\|_1=0$, then the inequality hods trivially).
Moreover, we need a smallness condition on $\eps$ in order to invoke \Cref{prop:secNash_StratCaps,prop:secNash_StratBulk}.

To that end, suppose that
\begin{equation}\label{aux129c}
\eps \le \kappa \min\left\{\frac{\eps_*}{32}, \frac{hK_*}{32R_*}\right\}.
\end{equation}
In particular, $\eps \le \kappa \eps_*/32$ is satisfied and implies $\eps \leq \eps_*$. Therefore, if (\ref{eq:secNash_StratFCapConcentrated}) holds, we may invoke \Cref{prop:secNash_StratCaps} to obtain
\[
\kappa^{4/3}\Gamma^{7/6}\nm{f-f_M}_1^{-4/3} = \kappa\Gamma^{1/2} \eps\pre \le C_4\nm{\grad f}_2,
\]
which is exactly \eqref{aux129b}. On the other hand, we also verify that
\[
\kappa\pre \Gamma\pre \eps^{-1/2} \nm{f-f_M}_1^2 = \kappa^{-5/6} \Gamma^{-2/3}\nm{f-f_M}_1^{4/3} = \kappa^{-1/2}\eps \le \frac{hK_*}{32R_*},
\]
where the right-most inequality holds due to \eqref{aux129c} and as $0 < \kappa \le 1$. Thus, if instead (\ref{eq:secNash_StratFBulkConcentrated}) holds, we then invoke \Cref{prop:secNash_StratBulk} to obtain
\[
\kappa^{4/3} \Gamma^{7/6}\nm{f-f_M}_1^{2/3} = \kappa^{3/2}\Gamma^{3/2} \eps^{1/2} \le C_5\nm{f-f_M}_1^2\nm{\grad f}_2,
\]
which is exactly \eqref{aux129b}. Altogether, the proof is complete under the smallness assumption \eqref{aux129c} on $\eps$.

It only remains to consider the case where
\begin{equation}\label{aux129d}
\kappa^{-1/3}\Gamma^{-2/3}\nm{f-f_M}_1^{4/3} =: \eps > \kappa \min\left\{\frac{\eps_*}{32}, \frac{hK_*}{32R_*}\right\}.
\end{equation}
Here, we employ the classical Nash inequality (\ref{eq:secIntro_Nash}) for dimension $d = 2$ to obtain
\[
\Gamma
\le
C_{\mathrm{N},2} \nm{f-f_M}_1 \nm{\grad f}_2.
\]
Multiplying both sides by
\[
\kappa^{1/3}\Gamma^{1/6} < \paren*{\max\left\{\frac{32}{\eps_*}, \frac{32R_*}{hK_*}\right\}}^{1/4} \nm{f-f_M}_1^{1/3},
\]
we get
\[
\kappa^{4/3}\Gamma^{7/6}
\le \kappa^{1/3}\Gamma^{7/6}
\le C_{\mathrm{N},2} \paren*{\max\left\{\frac{32}{\eps_*}, \frac{32R_*}{hK_*}\right\}}^{1/4} \nm{f-f_M}_1^{4/3} \nm{\grad f}_2,
\]
which implies \eqref{aux129b}, thus completing the proof in the case where \eqref{aux129d} holds.

Indeed, collecting our work gives the inequality
\[
\nm{\ov{f}-f_M}_2^{8/3} \le C_6 \nm{f-f_M}_1^{4/3} \nm{\grad f}_2 \nm{f-f_M}_2,
\]
with the following value for the universal constant:
\[
C_6 = C_6(\Omega) := \max\left\{C_4, C_5, C_{\mathrm{N},2} \paren*{\max\left\{\frac{32}{\eps_*}, \frac{32R_*}{hK_*}\right\}}^{1/4} \right\}.
\]
\end{proof}

\begin{remark}
Suppose one does not like the appearance of the $L^2$ norm on the right hand side of the inequality in \Cref{thm:secNash_NashOnStratified}, for whatever reason. One can simply invoke the classical Nash inequality (\ref{eq:secIntro_Nash}) to obtain
\[
\nm{\ov{f}-f_M}_2^{8/3} \le C_6 C_{\mathrm{N},2}^{1/6} \nm{f-f_M}_1^{9/6} \nm{\grad f}_2^{7/6}.
\]
\end{remark}

\begin{remark}\label{rem:secIntro_NashStratInterp1}
\Cref{thm:secNash_NashOnStratified} implies that
\[
\nm{f-f_M}_2^{7/3} \le C_6 \paren*{\frac{\nm{f-f_M}_2^2}{\nm{\ov{f}-f_M}_2^2}}^{4/3}\nm{f-f_M}_1^{4/3} \nm{\grad f}_2.
\]
Thus, by viewing the term $C_6 (\nm{f-f_M}_2^2/\nm{\ov{f}-f_M}_2^2)^{4/3} \ge C_6$ as a function-dependent coefficient, \Cref{thm:secNash_NashOnStratified} provides an improved scaling corresponding to a formal dimension of $d = 3/2$ for ``somewhat stratified '' functions when compared to the classical Nash inequality (\ref{eq:secIntro_Nash}). The inequality gives less control when the function is less stratified, as reflected by the function-dependent coefficient, with it degenerating whenever $\ov{f}$ is constantly equal to $f_M$.
\end{remark}

\begin{remark}\label{rem:secIntro_NashStratInterp2}
Moreover, upon rearranging, we equivalently have that
\[
\paren*{\frac{\nm{\ov{f}-f_M}_2^2}{\nm{f-f_M}_2^2}}^{4/3} \le C_6 \frac{1}{\nm{f-f_M}_2^{7/3}} \nm{f-f_M}_1^{4/3}\nm{\grad f}_2.
\]
Thus, \Cref{thm:secNash_NashOnStratified} can just as well be interpreted as providing a way to control how stratified the function $f$ is in the sense of the proportion $\nm{\ov{f}-f_M}_2^2/\nm{f-f_M}_2^2$. The negative power of the variance of $f$ on the right-hand side is necessary since the stratified proportion $\nm{\ov{f}-f_M}_2^2/\nm{f-f_M}_2^2$ may be close to $1$ even when the variance $\nm{f-f_M}_2^2$ is small in magnitude.
\end{remark}

\subsubsection{Proofs of auxiliary propositions}
\label{subsubsec:NashStratProps}

We now prove \Cref{prop:secNash_StratCaps} and \Cref{prop:secNash_StratBulk}. The strategy for these arguments can be abstracted as follows. An a-priori bound of the form $\nm{\ov{f}-f_M}_{L^2(\Omega')}^2 \gtrsim \Gamma$, for some subset $\Omega' \subseteq \Omega$, will be seen to imply a drop in the cross-sectional $L^1$ or $L^2$ norm of $f-f_M$ over the region $\Omega'$. This drop will depend on $\Gamma$ and $\nm{\ov{f}-f_M}_{L^2(\Omega')}^2\Gamma\pre$. The fundamental theorem of calculus then leads to a lower bound on the $L^2$ norm of $\grad f$ over the region where the drop occurs. This bound will depend on the magnitude of the drop itself and also on the area of the region $\Omega'$. However, the magnitude of $|\Omega'|$ varies significantly depending on whether $\Omega'$ lies in the $\eps$-caps $L_\eps$, $U_\eps$, or in the bulk $\Omega_\eps$. This forces us to consider those two situations separately and explains why \Cref{prop:secNash_StratCaps} and \Cref{prop:secNash_StratBulk} provide inequalities with different scalings. These different scalings were matched by an appropriate choice of $\eps$ in the proof of \Cref{thm:secNash_NashOnStratified}.

First, we present the case of stratification in the $\eps$-caps.

\begin{proof}[Proof of \Cref{prop:secNash_StratCaps}]
We only consider the case of the lower $\eps$-cap $L_\eps$, as the argument below can be repeated analogously with obvious modifications if $L_\eps$ is replaced by $U_\eps$.

We begin with the $L^2$ orthogonal decomposition over $L_\eps$,
\[
\nm{f-f_M}_{L^2(L_\eps)}^2 = \nm{\ov{f}-f_M}_{L^2(L_\eps)}^2+\nm{\wt{f}}_{L^2(L_\eps)}^2 \ge \kappa\Gamma/4.
\]
We must be able to fix $x_2' \in (0, \eps)$ such that
\begin{equation}\label{eq:secNash_lowerboundL2section}
\int_{I_{x_2'}} |f(x_1, x_2')-f_M|^2\, dx_1 \ge \frac{\kappa\Gamma}{4\eps},
\end{equation}
as otherwise we would obtain the contradiction
\[
\nm{f-f_M}_{L^2(L_\eps)}^2 = \int_0^\eps \int_{I_{x_2}} |f(x_1, x_2)-f_M|^2\, dx_1\, dx_2 < \kappa\Gamma/4.
\]
Similarly, we can fix $x_2'' \in (\eps, 16\eps/\kappa)$ such that
\begin{equation}\label{eq:secNash_upperboundL2section}
\int_{I_{x_2''}} |f(x_1, x_2'')-f_M|^2\, dx_1 \le \frac{\kappa\Gamma}{8\eps},
\end{equation}
as otherwise we would again obtain a contradiction:
\[
\Gamma = \nm{f-f_M}_2^2 \ge \int_\eps^{16\eps/\kappa} \int_{I_{x_2}} |f(x_1, x_2)-f_M|^2\, dx_1\, dx_2 > (2-\kappa/8)\Gamma > \Gamma.
\]
Together, (\ref{eq:secNash_lowerboundL2section}) and (\ref{eq:secNash_upperboundL2section}) imply that
\[
\paren*{ \int_{I_{x_2'}} |f(x_1,x_2')-f_M|^2\, dx_1 }^{1/2} - \paren*{ \int_{I_{x_2''}} |f(x_1,x_2'')-f_M|^2\, dx_1 }^{1/2} \ge \frac{2-\sqrt{2}}{4} \paren*{\frac{\kappa \Gamma}{\eps}}^{1/2}.
\]
Since $\eps < \kappa\eps_*/32$, we know from \Cref{as:secSett_3} and \Cref{norm:secSett_eps} that $I_{x_2'} \subseteq I_{x_2} \subseteq I_{x_2''}$ for all $x_2' \le x_2 \le x_2''$. In turn, the above inequality implies that
\begin{align*}
\frac{2-\sqrt{2}}{4} \paren*{\frac{\kappa \Gamma}{\eps}}^{1/2} &\le \paren*{ \int_{I_{x_2'}} |f(x_1,x_2')-f_M|^2\, dx_1 }^{1/2} - \paren*{ \int_{I_{x_2'}} |f(x_1,x_2'')-f_M|^2\, dx_1 }^{1/2} \\
&\le \paren*{ \int_{I_{x_2'}} |f(x_1,x_2'')-f(x_1,x_2')|^2\, dx_1 }^{1/2}.
\end{align*}
Note that the domain of integration is $I_{x_2'}$ instead of $I_{x_2''}$ for the right-most integral on the first line; its value can only decline if we shrink the domain. Invoking the fundamental theorem of calculus (again using that $I_{x_2'} \subseteq I_{x_2} \subseteq I_{x_2''}$ for all $x_2' \le x_2 \le x_2''$) and the Cauchy--Schwarz inequality, we obtain
\begin{align*}
\frac{2-\sqrt{2}}{4} \paren*{\frac{\kappa \Gamma}{\eps}}^{1/2} &\le \paren*{ \int_{I_{x_2'}} |f(x_1,x_2'')-f(x_1,x_2')|^2\, dx_1 }^{1/2} \\
&= \paren*{ \int_{I_{x_2'}} \abs*{\int_{x_2'}^{x_2''} \partial_2 f(x_1,x_2)\, dx_2 }^2\, dx_1 }^{1/2} \\
&\le (x_2''-x_2')^{1/2} \paren*{ \int_{I_{x_2'}} \int_{x_2'}^{x_2''} |\partial_2 f(x_1,x_2)|^2\, dx_2 \, dx_1 }^{1/2} \\
&\le (16\eps/\kappa)^{1/2} \nm{\partial_2 f}_2.
\end{align*}
Upon rearranging the above inequality, the proof is complete with the universal constant $C_4$ given by $C_4 = 16/(2-\sqrt{2})$.
\end{proof}

We next consider the case of stratification in the $\eps$-bulk $\Omega_\eps$. Since $\eps > 0$ is a variable that has to be chosen in the proof of \Cref{thm:secNash_NashOnStratified}, the subset $\Omega_\eps$ may still include horizontal cross-sections with much smaller lengths than those in $\Omega_{\eps_*}$, where $\eps_*$ is the fixed constant from \Cref{norm:secSett_eps}. Thus, we will need to consider two subcases: one focused on the set difference $\Omega_\eps \setminus \Omega_{\eps_*}$, and another focused on $\Omega_{\eps_*}$ where we have a fixed lower bound on the lengths of horizontal cross-sections. The following lemma considers the first subcase (which will be used to prove \Cref{prop:secNash_StratBulk}).

\begin{lemma}\label{lem:secNash_Sit1StratBulk}
Under \Cref{hyp:secNash}, suppose that
\[
\kappa\pre \Gamma\pre\eps^{-1/2}\nm{f-f_M}_1^2 \le \frac{h K_*}{32 R_*},
\quad
\nm{f(\cdot, x_2') - f_M}_{L^1(I_{x_2'})} \ge \frac{\kappa \Gamma |I_{x_2'}|}{2\nm{f-f_M}_1}
\]
for some fixed constants $0 < \kappa \le 1$ and $0 < \eps < \eps_*$ and some fixed height $x_2' \in [\eps, \eps_*]$. Then we have
\[
\kappa^{3/2} \Gamma^{3/2} \eps^{1/2} \le 2^{7/2}K_*\pre \nm{f-f_M}_1^2\nm{\grad f}_2.
\]
\end{lemma}

\begin{proof}
We may define $\delta := 8\nm{f-f_M}_1^2(\kappa\Gamma|I_{x_2'}|)\pre$ which satisfies $\delta \le h/(4R_*) < h/2$ by \Cref{norm:secSett_eps}, inequality (\ref{eq:secSett_DecayOfIx2}), the hypothesis of this lemma, and as $R_* \ge 1$ by (\ref{aux0307a}). In turn, there must exist $x_2''$ such that
\[
x_2' < x_2'' \le x_2' + \delta < \eps_*+h/2 \le 3h/4, \quad \int_{I_{x_2''}} \abs{f(x_1,x_2'')-f_M}\, dx_1 \le \frac{\kappa\Gamma|I_{x_2'}|}{4\nm{f-f_M}_1},
\]
lest we obtain the contradiction
\[
\nm{f-f_M}_1 \ge \int_{x_2'}^{x_2'+\delta} \int_{I_{x_2}} \abs{f(x_1,x_2)-f_M}\, dx_1\, dx_2 > 2\nm{f-f_M}_1.
\]
Importantly, by \Cref{norm:secSett_eps} and by the fact $x_2' < x_2'' \le 3h/4$, it holds that $|I_{x_2'}| \le |I_{x_2}|$ for all $x_2' \le x_2 \le x_2''$. This allows us to explicitly construct the following family of disjoint smooth curves
\[
\Phi_t(x_1) :=
\paren*{
x_1 + \frac{F_r(x_2'+t)-F_\ell(x_2'+t)}{2}-\frac{F_r(x_2')-F_\ell(x_2')}{2},\
x_2' + t
}
\]
where $x_1 \in I_{x_2'}$ and $0 \le t \le x_2'' - x_2'$, meaning that the curves start at the cross-section $I_{x_2'} \times \{x_2'\}$ and end at $I_{x_2''} \times \{x_2''\}$. Note that $\Phi_t$ simply shifts up the interval $I_{x_2'} \times \{x_2'\}$ to be a centered subinterval of $I_{x_2'+t} \times \{x_2' + t\}$, and hence $\Phi_t(x_1) \in \Omega$ for all valid choices of $t$ and $x_1 \in I_{x_2'}$; see \Cref{fig:secNash_NashStratCurves1}.

\tikzset{every picture/.style={line width=0.75pt}}
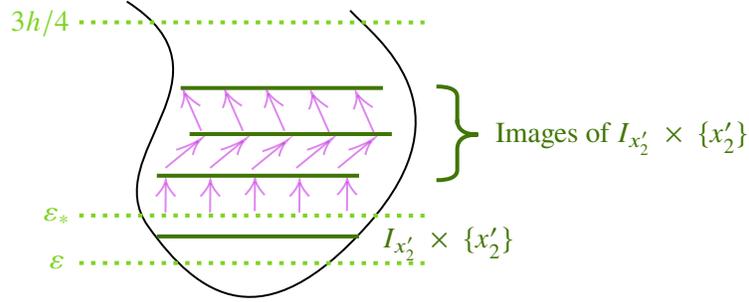
\begin{figure}[!ht]%
\centering
\begin{tikzpicture}[x=0.75pt,y=0.75pt,yscale=-.75,xscale=1.5]
\draw    (77,9) .. controls (116.09,59.09) and (67.09,112.09) .. (84.09,160.09) .. controls (101.09,208.09) and (123.09,238.09) .. (158.09,159.09) .. controls (193.09,80.09) and (161.09,34.09) .. (152.09,15.09) ;
\draw [color={rgb, 255:red, 65; green, 117; blue, 5 }  ,draw opacity=1 ][line width=1.5]    (87.09,167.09) -- (155.09,167.09) ;
\draw [color={rgb, 255:red, 126; green, 211; blue, 33 }  ,draw opacity=1 ][line width=1.5]  [dash pattern={on 1.69pt off 2.76pt}]  (61,185) -- (177.09,185.09) ;
\draw [color={rgb, 255:red, 126; green, 211; blue, 33 }  ,draw opacity=1 ][line width=1.5]  [dash pattern={on 1.69pt off 2.76pt}]  (61,153) -- (177.09,153.09) ;
\draw [color={rgb, 255:red, 189; green, 16; blue, 224 }  ,draw opacity=0.5 ]   (90,151) -- (90.08,131.23) ;
\draw [shift={(90.09,129.23)}, rotate = 90.23] [color={rgb, 255:red, 189; green, 16; blue, 224 }  ,draw opacity=0.5 ][line width=0.75]    (10.93,-3.29) .. controls (6.95,-1.4) and (3.31,-0.3) .. (0,0) .. controls (3.31,0.3) and (6.95,1.4) .. (10.93,3.29)   ;
\draw [color={rgb, 255:red, 189; green, 16; blue, 224 }  ,draw opacity=0.5 ]   (105,151) -- (105.08,131.23) ;
\draw [shift={(105.09,129.23)}, rotate = 90.23] [color={rgb, 255:red, 189; green, 16; blue, 224 }  ,draw opacity=0.5 ][line width=0.75]    (10.93,-3.29) .. controls (6.95,-1.4) and (3.31,-0.3) .. (0,0) .. controls (3.31,0.3) and (6.95,1.4) .. (10.93,3.29)   ;
\draw [color={rgb, 255:red, 189; green, 16; blue, 224 }  ,draw opacity=0.5 ]   (120,150) -- (120.08,130.23) ;
\draw [shift={(120.09,128.23)}, rotate = 90.23] [color={rgb, 255:red, 189; green, 16; blue, 224 }  ,draw opacity=0.5 ][line width=0.75]    (10.93,-3.29) .. controls (6.95,-1.4) and (3.31,-0.3) .. (0,0) .. controls (3.31,0.3) and (6.95,1.4) .. (10.93,3.29)   ;
\draw [color={rgb, 255:red, 189; green, 16; blue, 224 }  ,draw opacity=0.5 ]   (151,150) -- (151.08,130.23) ;
\draw [shift={(151.09,128.23)}, rotate = 90.23] [color={rgb, 255:red, 189; green, 16; blue, 224 }  ,draw opacity=0.5 ][line width=0.75]    (10.93,-3.29) .. controls (6.95,-1.4) and (3.31,-0.3) .. (0,0) .. controls (3.31,0.3) and (6.95,1.4) .. (10.93,3.29)   ;
\draw [color={rgb, 255:red, 189; green, 16; blue, 224 }  ,draw opacity=0.5 ]   (135,151) -- (135.08,131.23) ;
\draw [shift={(135.09,129.23)}, rotate = 90.23] [color={rgb, 255:red, 189; green, 16; blue, 224 }  ,draw opacity=0.5 ][line width=0.75]    (10.93,-3.29) .. controls (6.95,-1.4) and (3.31,-0.3) .. (0,0) .. controls (3.31,0.3) and (6.95,1.4) .. (10.93,3.29)   ;
\draw [color={rgb, 255:red, 65; green, 117; blue, 5 }  ,draw opacity=1 ][line width=1.5]    (87.09,126.23) -- (155.09,126.23) ;
\draw [color={rgb, 255:red, 189; green, 16; blue, 224 }  ,draw opacity=0.5 ]   (90,121) -- (101.04,102.94) ;
\draw [shift={(102.09,101.23)}, rotate = 121.44] [color={rgb, 255:red, 189; green, 16; blue, 224 }  ,draw opacity=0.5 ][line width=0.75]    (10.93,-3.29) .. controls (6.95,-1.4) and (3.31,-0.3) .. (0,0) .. controls (3.31,0.3) and (6.95,1.4) .. (10.93,3.29)   ;
\draw [color={rgb, 255:red, 189; green, 16; blue, 224 }  ,draw opacity=0.5 ]   (105,121) -- (116.04,102.94) ;
\draw [shift={(117.09,101.23)}, rotate = 121.44] [color={rgb, 255:red, 189; green, 16; blue, 224 }  ,draw opacity=0.5 ][line width=0.75]    (10.93,-3.29) .. controls (6.95,-1.4) and (3.31,-0.3) .. (0,0) .. controls (3.31,0.3) and (6.95,1.4) .. (10.93,3.29)   ;
\draw [color={rgb, 255:red, 189; green, 16; blue, 224 }  ,draw opacity=0.5 ]   (119,122) -- (130.04,103.94) ;
\draw [shift={(131.09,102.23)}, rotate = 121.44] [color={rgb, 255:red, 189; green, 16; blue, 224 }  ,draw opacity=0.5 ][line width=0.75]    (10.93,-3.29) .. controls (6.95,-1.4) and (3.31,-0.3) .. (0,0) .. controls (3.31,0.3) and (6.95,1.4) .. (10.93,3.29)   ;
\draw [color={rgb, 255:red, 189; green, 16; blue, 224 }  ,draw opacity=0.5 ]   (134,121) -- (145.04,102.94) ;
\draw [shift={(146.09,101.23)}, rotate = 121.44] [color={rgb, 255:red, 189; green, 16; blue, 224 }  ,draw opacity=0.5 ][line width=0.75]    (10.93,-3.29) .. controls (6.95,-1.4) and (3.31,-0.3) .. (0,0) .. controls (3.31,0.3) and (6.95,1.4) .. (10.93,3.29)   ;
\draw [color={rgb, 255:red, 189; green, 16; blue, 224 }  ,draw opacity=0.5 ]   (148,121) -- (159.04,102.94) ;
\draw [shift={(160.09,101.23)}, rotate = 121.44] [color={rgb, 255:red, 189; green, 16; blue, 224 }  ,draw opacity=0.5 ][line width=0.75]    (10.93,-3.29) .. controls (6.95,-1.4) and (3.31,-0.3) .. (0,0) .. controls (3.31,0.3) and (6.95,1.4) .. (10.93,3.29)   ;
\draw [color={rgb, 255:red, 65; green, 117; blue, 5 }  ,draw opacity=1 ][line width=1.5]    (98.09,98.23) -- (166.09,98.23) ;
\draw [color={rgb, 255:red, 65; green, 117; blue, 5 }  ,draw opacity=1 ][line width=1.5]    (95.09,67.23) -- (163.09,67.23) ;
\draw [color={rgb, 255:red, 189; green, 16; blue, 224 }  ,draw opacity=0.5 ]   (102,96) -- (96.52,71.18) ;
\draw [shift={(96.09,69.23)}, rotate = 77.54] [color={rgb, 255:red, 189; green, 16; blue, 224 }  ,draw opacity=0.5 ][line width=0.75]    (10.93,-3.29) .. controls (6.95,-1.4) and (3.31,-0.3) .. (0,0) .. controls (3.31,0.3) and (6.95,1.4) .. (10.93,3.29)   ;
\draw [color={rgb, 255:red, 189; green, 16; blue, 224 }  ,draw opacity=0.5 ]   (116,96) -- (110.52,71.18) ;
\draw [shift={(110.09,69.23)}, rotate = 77.54] [color={rgb, 255:red, 189; green, 16; blue, 224 }  ,draw opacity=0.5 ][line width=0.75]    (10.93,-3.29) .. controls (6.95,-1.4) and (3.31,-0.3) .. (0,0) .. controls (3.31,0.3) and (6.95,1.4) .. (10.93,3.29)   ;
\draw [color={rgb, 255:red, 189; green, 16; blue, 224 }  ,draw opacity=0.5 ]   (130,96) -- (124.52,71.18) ;
\draw [shift={(124.09,69.23)}, rotate = 77.54] [color={rgb, 255:red, 189; green, 16; blue, 224 }  ,draw opacity=0.5 ][line width=0.75]    (10.93,-3.29) .. controls (6.95,-1.4) and (3.31,-0.3) .. (0,0) .. controls (3.31,0.3) and (6.95,1.4) .. (10.93,3.29)   ;
\draw [color={rgb, 255:red, 189; green, 16; blue, 224 }  ,draw opacity=0.5 ]   (145,97) -- (139.52,72.18) ;
\draw [shift={(139.09,70.23)}, rotate = 77.54] [color={rgb, 255:red, 189; green, 16; blue, 224 }  ,draw opacity=0.5 ][line width=0.75]    (10.93,-3.29) .. controls (6.95,-1.4) and (3.31,-0.3) .. (0,0) .. controls (3.31,0.3) and (6.95,1.4) .. (10.93,3.29)   ;
\draw [color={rgb, 255:red, 189; green, 16; blue, 224 }  ,draw opacity=0.5 ]   (160,97) -- (154.52,72.18) ;
\draw [shift={(154.09,70.23)}, rotate = 77.54] [color={rgb, 255:red, 189; green, 16; blue, 224 }  ,draw opacity=0.5 ][line width=0.75]    (10.93,-3.29) .. controls (6.95,-1.4) and (3.31,-0.3) .. (0,0) .. controls (3.31,0.3) and (6.95,1.4) .. (10.93,3.29)   ;
\draw [color={rgb, 255:red, 126; green, 211; blue, 33 }  ,draw opacity=1 ][line width=1.5]  [dash pattern={on 1.69pt off 2.76pt}]  (61,23) -- (177.09,23.09) ;
\draw  [color={rgb, 255:red, 65; green, 117; blue, 5 }  ,draw opacity=1 ][line width=1.5]  (181.09,129.23) .. controls (185.76,129.23) and (188.09,126.9) .. (188.09,122.23) -- (188.09,108.57) .. controls (188.09,101.9) and (190.42,98.57) .. (195.09,98.57) .. controls (190.42,98.57) and (188.09,95.24) .. (188.09,88.57)(188.09,91.57) -- (188.09,73.23) .. controls (188.09,68.56) and (185.76,66.23) .. (181.09,66.23) ;
\draw (50,178) node [anchor=north west][inner sep=0.75pt]  [color={rgb, 255:red, 126; green, 211; blue, 33 }  ,opacity=1 ]  {$\varepsilon$};
\draw (48,146) node [anchor=north west][inner sep=0.75pt]  [color={rgb, 255:red, 126; green, 211; blue, 33 }  ,opacity=1 ]  {$\varepsilon_{*}$};
\draw (162,159) node [anchor=north west][inner sep=0.75pt]  [color={rgb, 255:red, 65; green, 117; blue, 5 }  ,opacity=1 ]  {$I_{x_{2} '} \ \times \ \{x_{2} '\}$};
\draw (37,13) node [anchor=north west][inner sep=0.75pt]  [color={rgb, 255:red, 126; green, 211; blue, 33 }  ,opacity=1 ]  {$3h/4$};
\draw (200,88) node [anchor=north west][inner sep=0.75pt]  [color={rgb, 255:red, 65; green, 117; blue, 5 }  ,opacity=1 ] [align=left] {Images of $I_{x_{2} '} \ \times \ \{x_{2} '\}$};
\end{tikzpicture}
\caption{Cartoon depiction of images of $I_{x_2'} \times \{x_2'\}$ under $\Phi_t$, with the arrows representing $\partial_t \Phi_t$. This construction of $\Phi_t$ preserves the length of $I_{x_2'} \times \{x_2'\}$.}
\label{fig:secNash_NashStratCurves1}
\end{figure}

We denote the image of the curves $t \mapsto \Phi_t(x_1)$ by $\phi_{x_1} := \{\Phi_t(x_1)\ |\ 0 \le t \le x_2'' - x_2'\}$, which are indexed by $x_1 \in I_{x_2'}$. This allows us to integrate over a portion of the domain $\Omega$ as follows:
\begin{align} \nonumber
\int_{I_{x_2'}} \int_{\phi_{x_1}} |\grad f(z)|\, dz\, dx_1
& \nonumber = \int_{I_{x_2'}} \int_0^{x_2''-x_2'} |\grad f(\Phi_t(x_1))||\partial_t \Phi_t(x_1)|\, dt\, dx_1 \\
& \nonumber \ge \int_{I_{x_2'}} \abs*{\int_0^{x_2''-x_2'} \grad f(\Phi_t(x_1))\cdot\partial_t \Phi_t(x_1)\, dt} dx_1 \\
& \label{aux129f} = \int_{I_{x_2'}} \abs*{f(\Phi_{x_2''-x_2'}(x_1))-f(\Phi_0(x_1)) } dx_1.
\end{align}
Applying the reverse triangle inequality and the fact that the image of $x_1 \mapsto \Phi_{x_2''-x_2'}(x_1)$ is a subinterval of $I_{x_2''} \times \{x_2''\}$ that has the same length as $I_{x_2'} \times \{x_2'\}$, we obtain
\begin{align*}
\int_{I_{x_2'}} \int_{\phi_{x_1}} |\grad f(z)|\, dz\, dx_1
&\ge \int_{I_{x_2'}} \abs*{f(x_1,x_2') - f_M}\, dx_1 - \int_{I_{x_2'}}\abs*{f(\Phi_{x_2''-x_2'}(x_1))-f_M}\, dx_1 \\
&\ge \int_{I_{x_2'}} \abs*{f(x_1,x_2') - f_M}\, dx_1  - \int_{I_{x_2''}} \abs*{f(x_1,x_2'')-f_M}\, dx_1 \\
&\ge \frac{\kappa\Gamma|I_{x_2'}|}{4\nm{f-f_M}_1},
\end{align*}
where the final inequality holds by the hypothesis on $x_2'$ and by the choice of $x_2''$. It follows by the Cauchy--Schwarz inequality that
\begin{align*}
\int_\Omega |\grad f(x)|^2\, dx &\ge \int_{I_{x_2'}} \int_{\phi_{x_1}} |\grad f(z)|^2\, dz\, dx_1 \\
&\ge \frac{1}{\delta|I_{x_2'}|} \paren*{\int_{I_{x_2'}} \int_{\phi(x_1)} |\grad f(z)|\, dz\, dx_1}^2 \\
&\ge \frac{\kappa^3\Gamma^3|I_{x_2'}|^2}{128\nm{f-f_M}_1^4},
\end{align*}
where we have used the smoothness of the boundary charts $F_r, F_\ell$ to compute:
\[
\int_{I_{x_2'}} \int_{\phi_{x_1}} 1\, dz\, dx_1
= \abs*{\{\Phi_t(x_1)\ |\ x_1 \in I_{x_2'},\ 0 \le t \le x_2''-x_2' \}}
\le \delta|I_{x_2'}|.
\]
Then, (\ref{eq:secSett_DecayOfIx2}) completes the proof.
\end{proof}

The following lemma considers the second subcase which will be used to prove \Cref{prop:secNash_StratBulk}. The proof is very similar to that of \Cref{lem:secNash_Sit1StratBulk}; we need only choose a different family of curves to integrate over.

\begin{lemma}\label{lem:secNash_Sit2StratBulk}
Under \Cref{hyp:secNash}, suppose that
\[
\kappa\pre \Gamma\pre \eps^{-1/2} \nm{f-f_M}_1^2 \le \frac{h K_* }{32 R_*},
\quad
\nm{f(\cdot, x_2') - f_M}_{L^1(I_{x_2'})} \ge \frac{\kappa \Gamma |I_{x_2'}|}{2\nm{f-f_M}_1}
\]
for some fixed constants $0 < \kappa \le 1$ and $0 < \eps < \eps_*$ and some fixed height $x_2' \in [\eps_*,h/2]$. Then we have
\[
\kappa^{3/2} \Gamma^{3/2} \eps^{1/2} \le 2^{7/2}R_*K_*\pre \nm{f-f_M}_1^2\nm{\grad f}_2.
\]
\end{lemma}

\begin{proof}
We may define $\delta := 8R_*\nm{f-f_M}_1^2(\kappa\Gamma|I_{x_2'}|)\pre$ which satisfies $\delta \le h/4$ by \Cref{norm:secSett_eps}, inequality (\ref{eq:secSett_DecayOfIx2}) and the hypothesis of this lemma. In turn, there must exist $x_2''$ such that
\[
x_2' < x_2'' \le x_2' + \delta \le h/2+h/4 = 3h/4, \quad \int_{I_{x_2''}} \abs{f(x_1,x_2'')-f_M}\, dx_1 \le \frac{\kappa\Gamma|I_{x_2'}|}{4R_*\nm{f-f_M}_1},
\]
as otherwise we obtain the contradiction
\[
\nm{f-f_M}_1 \ge \int_{x_2'}^{x_2'+\delta} \int_{I_{x_2}} \abs{f(x_1,x_2)-f_M}\, dx_1\, dx_2 > 2\nm{f-f_M}_1.
\]
We can again explicitly construct a family of disjoint smooth curves
\begin{align*}
\Phi_t(x_1) &:= \paren*{
F_\ell(x_2'+t) + \paren*{x_1-F_\ell(x_2')}\frac{|I_{x_2'+t}|}{|I_{x_2'}|},\ x_2' + t
}
\end{align*}
where $x_1 \in I_{x_2'}$ and $0 \le t \le x_2'' - x_2'$, meaning that the curves start at the cross-section $I_{x_2'} \times \{x_2'\}$ and end at $I_{x_2''} \times \{x_2''\}$. Note that $\Phi_t$ indeed provides a foliation of the portion of the domain between heights $x_2'$ and $x_2''$; see \Cref{fig:secNash_NashStratCurves2}.

\tikzset{every picture/.style={line width=0.75pt}}
\begin{figure}[!ht]%
\centering
\begin{tikzpicture}[x=0.75pt,y=0.75pt,yscale=-1,xscale=1]
\draw   (156.09,26.71) .. controls (192.09,25.71) and (247.09,53.71) .. (227.09,73.71) .. controls (207.09,93.71) and (234.09,105.71) .. (247.09,142.71) .. controls (260.09,179.71) and (185.09,209.71) .. (165.09,209.71) .. controls (145.09,209.71) and (92.09,198.21) .. (83.09,182.71) .. controls (74.09,167.21) and (80.09,142.71) .. (84.09,125.71) .. controls (88.09,108.71) and (76.09,110.71) .. (72.09,83.71) .. controls (68.09,56.71) and (120.09,27.71) .. (156.09,26.71) -- cycle ;
\draw [color={rgb, 255:red, 126; green, 211; blue, 33 }  ,draw opacity=1 ][line width=1.5]  [dash pattern={on 1.69pt off 2.76pt}]  (65.09,121.17) -- (252.09,121.17) ;
\draw [color={rgb, 255:red, 126; green, 211; blue, 33 }  ,draw opacity=1 ][line width=1.5]  [dash pattern={on 1.69pt off 2.76pt}]  (65.09,69.17) -- (252.09,69.17) ;
\draw [color={rgb, 255:red, 126; green, 211; blue, 33 }  ,draw opacity=1 ][line width=1.5]  [dash pattern={on 1.69pt off 2.76pt}]  (65.09,189.17) -- (252.09,189.17) ;
\draw [color={rgb, 255:red, 65; green, 117; blue, 5 }  ,draw opacity=1 ][line width=1.5]    (80.09,175.94) -- (237.09,175.94) ;
\draw [color={rgb, 255:red, 65; green, 117; blue, 5 }  ,draw opacity=1 ][line width=1.5]    (80.09,145.03) -- (247.09,145.71) ;
\draw [color={rgb, 255:red, 65; green, 117; blue, 5 }  ,draw opacity=1 ][line width=1.5]    (81.09,108.03) -- (228.09,107.94) ;
\draw [color={rgb, 255:red, 65; green, 117; blue, 5 }  ,draw opacity=1 ][line width=1.5]    (72.09,79.03) -- (222.09,78.94) ;
\draw [color={rgb, 255:red, 189; green, 16; blue, 224 }  ,draw opacity=0.5 ]   (91,173) -- (85.59,152.45) ;
\draw [shift={(85.09,150.51)}, rotate = 75.26] [color={rgb, 255:red, 189; green, 16; blue, 224 }  ,draw opacity=0.5 ][line width=0.75]    (10.93,-3.29) .. controls (6.95,-1.4) and (3.31,-0.3) .. (0,0) .. controls (3.31,0.3) and (6.95,1.4) .. (10.93,3.29)   ;
\draw [color={rgb, 255:red, 189; green, 16; blue, 224 }  ,draw opacity=0.5 ]   (225,172) -- (241.76,153.01) ;
\draw [shift={(243.09,151.51)}, rotate = 131.44] [color={rgb, 255:red, 189; green, 16; blue, 224 }  ,draw opacity=0.5 ][line width=0.75]    (10.93,-3.29) .. controls (6.95,-1.4) and (3.31,-0.3) .. (0,0) .. controls (3.31,0.3) and (6.95,1.4) .. (10.93,3.29)   ;
\draw [color={rgb, 255:red, 189; green, 16; blue, 224 }  ,draw opacity=0.5 ]   (127,172) -- (116.96,151.31) ;
\draw [shift={(116.09,149.51)}, rotate = 64.11] [color={rgb, 255:red, 189; green, 16; blue, 224 }  ,draw opacity=0.5 ][line width=0.75]    (10.93,-3.29) .. controls (6.95,-1.4) and (3.31,-0.3) .. (0,0) .. controls (3.31,0.3) and (6.95,1.4) .. (10.93,3.29)   ;
\draw [color={rgb, 255:red, 189; green, 16; blue, 224 }  ,draw opacity=0.5 ]   (159.09,171.8) -- (158.63,150.37) ;
\draw [shift={(158.59,148.37)}, rotate = 88.78] [color={rgb, 255:red, 189; green, 16; blue, 224 }  ,draw opacity=0.5 ][line width=0.75]    (10.93,-3.29) .. controls (6.95,-1.4) and (3.31,-0.3) .. (0,0) .. controls (3.31,0.3) and (6.95,1.4) .. (10.93,3.29)   ;
\draw [color={rgb, 255:red, 189; green, 16; blue, 224 }  ,draw opacity=0.5 ]   (190.09,171.8) -- (202.08,151.24) ;
\draw [shift={(203.09,149.51)}, rotate = 120.26] [color={rgb, 255:red, 189; green, 16; blue, 224 }  ,draw opacity=0.5 ][line width=0.75]    (10.93,-3.29) .. controls (6.95,-1.4) and (3.31,-0.3) .. (0,0) .. controls (3.31,0.3) and (6.95,1.4) .. (10.93,3.29)   ;
\draw [color={rgb, 255:red, 189; green, 16; blue, 224 }  ,draw opacity=0.5 ]   (88.09,138.51) -- (93.6,116.45) ;
\draw [shift={(94.09,114.51)}, rotate = 104.04] [color={rgb, 255:red, 189; green, 16; blue, 224 }  ,draw opacity=0.5 ][line width=0.75]    (10.93,-3.29) .. controls (6.95,-1.4) and (3.31,-0.3) .. (0,0) .. controls (3.31,0.3) and (6.95,1.4) .. (10.93,3.29)   ;
\draw [color={rgb, 255:red, 189; green, 16; blue, 224 }  ,draw opacity=0.5 ]   (115.09,139.8) -- (121.53,117.44) ;
\draw [shift={(122.09,115.51)}, rotate = 106.08] [color={rgb, 255:red, 189; green, 16; blue, 224 }  ,draw opacity=0.5 ][line width=0.75]    (10.93,-3.29) .. controls (6.95,-1.4) and (3.31,-0.3) .. (0,0) .. controls (3.31,0.3) and (6.95,1.4) .. (10.93,3.29)   ;
\draw [color={rgb, 255:red, 189; green, 16; blue, 224 }  ,draw opacity=0.5 ]   (156.09,139.8) -- (150.57,117.46) ;
\draw [shift={(150.09,115.51)}, rotate = 76.12] [color={rgb, 255:red, 189; green, 16; blue, 224 }  ,draw opacity=0.5 ][line width=0.75]    (10.93,-3.29) .. controls (6.95,-1.4) and (3.31,-0.3) .. (0,0) .. controls (3.31,0.3) and (6.95,1.4) .. (10.93,3.29)   ;
\draw [color={rgb, 255:red, 189; green, 16; blue, 224 }  ,draw opacity=0.5 ]   (200.09,140.8) -- (182.32,118.09) ;
\draw [shift={(181.09,116.51)}, rotate = 51.96] [color={rgb, 255:red, 189; green, 16; blue, 224 }  ,draw opacity=0.5 ][line width=0.75]    (10.93,-3.29) .. controls (6.95,-1.4) and (3.31,-0.3) .. (0,0) .. controls (3.31,0.3) and (6.95,1.4) .. (10.93,3.29)   ;
\draw [color={rgb, 255:red, 189; green, 16; blue, 224 }  ,draw opacity=0.5 ]   (236,139) -- (220.25,117.14) ;
\draw [shift={(219.09,115.51)}, rotate = 54.24] [color={rgb, 255:red, 189; green, 16; blue, 224 }  ,draw opacity=0.5 ][line width=0.75]    (10.93,-3.29) .. controls (6.95,-1.4) and (3.31,-0.3) .. (0,0) .. controls (3.31,0.3) and (6.95,1.4) .. (10.93,3.29)   ;
\draw [color={rgb, 255:red, 189; green, 16; blue, 224 }  ,draw opacity=0.5 ]   (92,106) -- (81.96,85.31) ;
\draw [shift={(81.09,83.51)}, rotate = 64.11] [color={rgb, 255:red, 189; green, 16; blue, 224 }  ,draw opacity=0.5 ][line width=0.75]    (10.93,-3.29) .. controls (6.95,-1.4) and (3.31,-0.3) .. (0,0) .. controls (3.31,0.3) and (6.95,1.4) .. (10.93,3.29)   ;
\draw [color={rgb, 255:red, 189; green, 16; blue, 224 }  ,draw opacity=0.5 ]   (119,105) -- (115.43,84.48) ;
\draw [shift={(115.09,82.51)}, rotate = 80.12] [color={rgb, 255:red, 189; green, 16; blue, 224 }  ,draw opacity=0.5 ][line width=0.75]    (10.93,-3.29) .. controls (6.95,-1.4) and (3.31,-0.3) .. (0,0) .. controls (3.31,0.3) and (6.95,1.4) .. (10.93,3.29)   ;
\draw [color={rgb, 255:red, 189; green, 16; blue, 224 }  ,draw opacity=0.5 ]   (147.09,105.8) -- (146.63,84.37) ;
\draw [shift={(146.59,82.37)}, rotate = 88.78] [color={rgb, 255:red, 189; green, 16; blue, 224 }  ,draw opacity=0.5 ][line width=0.75]    (10.93,-3.29) .. controls (6.95,-1.4) and (3.31,-0.3) .. (0,0) .. controls (3.31,0.3) and (6.95,1.4) .. (10.93,3.29)   ;
\draw [color={rgb, 255:red, 189; green, 16; blue, 224 }  ,draw opacity=0.5 ]   (216.09,105.8) -- (211.49,83.47) ;
\draw [shift={(211.09,81.51)}, rotate = 78.37] [color={rgb, 255:red, 189; green, 16; blue, 224 }  ,draw opacity=0.5 ][line width=0.75]    (10.93,-3.29) .. controls (6.95,-1.4) and (3.31,-0.3) .. (0,0) .. controls (3.31,0.3) and (6.95,1.4) .. (10.93,3.29)   ;
\draw [color={rgb, 255:red, 189; green, 16; blue, 224 }  ,draw opacity=0.5 ]   (179,106) -- (174.46,82.48) ;
\draw [shift={(174.09,80.51)}, rotate = 79.09] [color={rgb, 255:red, 189; green, 16; blue, 224 }  ,draw opacity=0.5 ][line width=0.75]    (10.93,-3.29) .. controls (6.95,-1.4) and (3.31,-0.3) .. (0,0) .. controls (3.31,0.3) and (6.95,1.4) .. (10.93,3.29)   ;
\draw  [color={rgb, 255:red, 65; green, 117; blue, 5 }  ,draw opacity=1 ][line width=1.5]  (257.09,147.23) .. controls (261.76,147.19) and (264.07,144.84) .. (264.02,140.17) -- (263.87,123.04) .. controls (263.81,116.37) and (266.11,113.02) .. (270.78,112.97) .. controls (266.11,113.02) and (263.75,109.71) .. (263.69,103.04)(263.72,106.04) -- (263.52,83.79) .. controls (263.48,79.12) and (261.13,76.81) .. (256.46,76.86) ;
\draw (32,113) node [anchor=north west][inner sep=0.75pt]  [color={rgb, 255:red, 126; green, 211; blue, 33 }  ,opacity=1 ]  {$h/2$};
\draw (25,60.5) node [anchor=north west][inner sep=0.75pt]  [color={rgb, 255:red, 126; green, 211; blue, 33 }  ,opacity=1 ]  {$3h/4$};
\draw (45,184) node [anchor=north west][inner sep=0.75pt]  [color={rgb, 255:red, 126; green, 211; blue, 33 }  ,opacity=1 ]  {$\varepsilon _{*}$};
\draw (244,167.4) node [anchor=north west][inner sep=0.75pt]  [color={rgb, 255:red, 65; green, 117; blue, 5 }  ,opacity=1 ]  {$I_{x_{2} '} \ \times \ \{x_{2} '\}$};
\draw (274,105) node [anchor=north west][inner sep=0.75pt]  [color={rgb, 255:red, 65; green, 117; blue, 5 }  ,opacity=1 ] [align=left] {Images of $I_{x_{2} '} \ \times \ \{x_{2} '\}$};
\end{tikzpicture}
\caption{Cartoon depiction of images of $I_{x_2'} \times \{x_2'\}$ under $\Phi_t$, with the arrows representing $\partial_t \Phi_t$. This construction of $\Phi_t$ stretches and shrinks the length of $I_{x_2'} \times \{x_2'\}$.}
\label{fig:secNash_NashStratCurves2}
\end{figure}
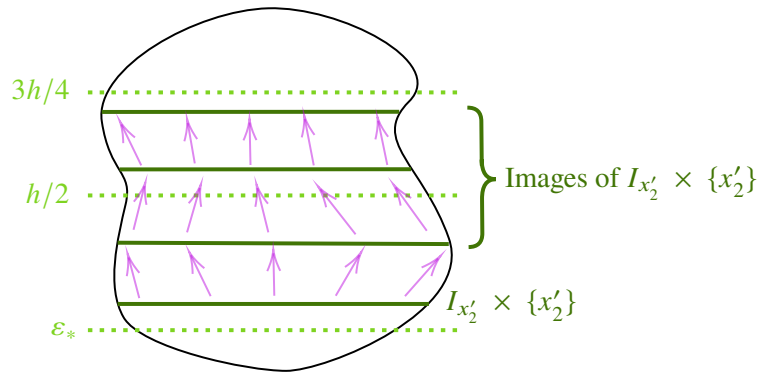

We again denote the image of the curves $t \mapsto \Phi_t(x_1)$ by $\phi_{x_1} := \{\Phi_t(x_1)\ |\ 0 \le t \le x_2'' - x_2'\}$, which are indexed by $x_1 \in I_{x_2'}$. Just as we did in \Cref{lem:secNash_Sit1StratBulk}, we can use these curves $\phi_{x_1}$ to integrate over the domain as follows:
\begin{align*}
&\int_{I_{x_2'}} \int_{\phi_{x_1}} |\grad f(z)|\, dz\, dx_1 \\
&\quad \ge \int_{I_{x_2'}} \abs*{f(\Phi_{x_2''-x_2'}(x_1))-f(\Phi_0(x_1)) } dx_1 \\
&\quad \ge \int_{I_{x_2'}} \paren*{ \abs*{f(x_1,x_2') - f_M} - \abs*{f\paren*{F_\ell(x_2'') + \paren*{x_1-F_\ell(x_2')}\frac{|I_{x_2''}|}{|I_{x_2'}|}, x_2''}-f_M} } dx_1 \\
&\quad = \int_{I_{x_2'}} \abs*{f(x_1,x_2') - f_M} dx_1  - \frac{|I_{x_2'}|}{|I_{x_2''}|}\int_{I_{x_2''}} \abs*{f(x_1,x_2'')-f_M} dx_1.
\end{align*}
By the hypothesis on $x_2'$ and by the choice of $x_2''$, we have
\[
\int_{I_{x_2'}} \int_{\phi_{x_1}} |\grad f(z)|\, dz\, dx_1
\ge
\frac{\kappa\Gamma|I_{x_2'}|}{2\nm{f-f_M}_1}
- \frac{|I_{x_2'}|}{|I_{x_2''}|}\frac{\kappa\Gamma|I_{x_2'}|}{4R_*\nm{f-f_M}_1}
\ge
\frac{\kappa\Gamma|I_{x_2'}|}{4\nm{f-f_M}_1},
\]
where the final inequality relies on the bound $|I_{x_2'}| / |I_{x_2''}| \le R_*$ due to (\ref{aux0307a}) and as $\eps_* \le x_2' < x_2'' \le 3h/4 < h-\eps_*$. It follows by the Cauchy--Schwarz inequality that
\begin{align*}
\int_\Omega |\grad f(x)|^2\, dx &\ge \int_{I_{x_2'}} \int_{\phi(x_1)} |\grad f(z)|^2\, dz\, dx_1 \\
&\ge \frac{1}{\delta R_*|I_{x_2'}|} \paren*{\int_{I_{x_2'}} \int_{\phi(x_1)} |\grad f(z)|\, dz\, dx_1}^2 \\
&\ge \frac{\kappa^3\Gamma^3|I_{x_2'}|^2}{128 R_*^2\nm{f-f_M}_1^4},
\end{align*}
where we use the smoothness of the boundary charts $F_r, F_\ell$ to compute:
\[
\int_{I_{x_2'}} \int_{\phi_{x_1}} 1\, dz\, dx_1 =
\abs*{\{\Phi_t(x_1)\ |\ x_1 \in I_{x_2'},\ 0 \le t \le x_2''-x_2' \}} \le \delta R_*|I_{x_2'}|.
\]
Then, (\ref{eq:secSett_DecayOfIx2}) completes the proof.
\end{proof}

We now combine \Cref{lem:secNash_Sit1StratBulk,lem:secNash_Sit2StratBulk} in order to prove \Cref{prop:secNash_StratBulk}, which we restate below for the reader's convenience.

\StratBulk*

\begin{proof}
We observe by H\"{o}lder's inequality that
\[
\kappa\Gamma/2
\le \nm{\ov{f}-f_M}_{L^2(\Omega_\eps)}^2
\le \nm{\ov{f}-f_M}_{L^1(\Omega_\eps)}\nm{\ov{f}-f_M}_{L^\infty(\Omega_\eps)}
\le \nm{f-f_M}_1\nm{\ov{f}-f_M}_{L^\infty(\Omega_\eps)}.
\]
By the continuity of $\ov{f}-f_M$, we may in turn choose $x_2' \in [\eps, h-\eps]$ such that
\[
\abs{\ov{f}(x_2')-f_M} = \nm{\ov{f}-f_M}_{L^\infty(\Omega_\eps)} \ge \frac{\kappa\Gamma}{2\nm{f-f_M}_1}.
\]
Then, as $\wt{f}$ is mean-zero in $x_1$, we have
\[
\int_{I_{x_2'}} \abs{f(x_1,x_2')-f_M}\, dx_1
\ge \abs*{\int_{I_{x_2'}} \left( f(x_1,x_2')-f_M \right)\, dx_1}
= |I_{x_2'}|\abs{\ov{f}(x_2')-f_M}
\ge \frac{\kappa|I_{x_2'}|\Gamma}{2\nm{f-f_M}_1}.
\]
If $x_2' \in [\eps, \eps_*]$, we may apply \Cref{lem:secNash_Sit1StratBulk}. If $x_2' \in [\eps_*, h/2]$, we may apply \Cref{lem:secNash_Sit2StratBulk}. In the situations $x_2' \in [h-\eps_*, h-\eps]$ or $x_2' \in [h/2, h-\eps_*]$, we can perform completely analogous arguments. Thus, collecting the conclusions of all possible situations and using that $R_* \ge 1$, the proof is complete with the universal constant given by $C_5 = 2^{7/2}R_*/K_*$.
\end{proof}

\subsection{Nash inequality on the unstratified component}
\label{subsec:UnstratNash}

We next establish the following inequality:
\begin{equation}\label{aux129g}
\nm{\wt{f}}_2^{5+\frac{22}{52}}
\lesssim_\Omega
\nm{f-f_M}_2^{3+\frac{21}{52}} \nm{\partial_1 f}_{\dH\pre_0}^{\frac{1}{52}} \nm{f-f_M}_1 \nm{\grad f}_2,
\end{equation}
which is precisely stated below in \Cref{thm:secNash_NashOnUnstratified}.

Recall the definition of the maximal ``sectional $L^2$ norm" $\Lambda$ from \eqref{aux129a}:
\[
\Lambda = \Gamma^{-3/2} \sup_{x_2 \in [0,h]}\nm{\wt{f}(\cdot, x_2)}_{L^2(I_{x_2})}^2,
\]
where $\Gamma = \nm{f -f_M}_2^2$. The specific choice of $\Gamma$-scaling in the definition of $\Lambda$ turns out to be exactly relevant for proving \eqref{aux129g}. It is also matches the scaling of the self-similar blow-up solution of the 2D PKS equation (\ref{eq:secIntro_PKSGeneral},\ref{eq:secIntro_PKSeqC}); see \cites{RS14,CTMN22a}.

\subsubsection{Proof of theorem and remarks}
\label{subsubsec:UnstratNashProof}

We consider three cases regarding the magnitude of $\Lambda$. When $\Lambda$ is small, the following inequality holds true.

\begin{restatable}{proposition}{Belowlambda}
\label{prop:secNash_Belowlambda}
Under \Cref{hyp:secNash}, suppose that
\begin{equation}\label{aux130a}
\nm{\wt{f}}_1 = 1,
\quad
\int_S \nm{\wt{f}(\cdot,x_2)}_{L^2(I_{x_2})}^2\, dx_2 \ge \kappa\Gamma,
\quad
\Gamma^{-3/2} \sup_{x_2 \in S} \nm{\wt{f}(\cdot,x_2)}_{L^2(I_{x_2})}^2 < \lambda,
\end{equation}
for some fixed (measurable) subset of heights $S \subseteq [0,h]$ and for some fixed $0 < \kappa \le 1$ and $0 < \lambda < 1$. Then, we have
\[
\kappa^{5/2}\Gamma\lambda\pre \le C_7 \paren*{\int_S \nm{\partial_1 f(\cdot,x_2)}_{L^2(I_{x_2})}^2\, dx_2}^{1/2},
\]
for some universal constant $C_7 = C_7(\Omega) > 0$.
\end{restatable}

Note that the case of $\Lambda < \lambda$ is covered by taking $S = [0, h]$ in the above proposition. The reason we generalize to an arbitrary subset $S$ of heights is that we will use \Cref{prop:secNash_Belowlambda} to help prove a more nuanced result later; see \Cref{lem:secNash_BetweenDeltaAndDeltaInvStep1}. The case of $\Lambda$ large leads to a similar inequality holding true.

\begin{restatable}{proposition}{AbovelambdaInv}
\label{prop:secNash_AbovelambdaInv}
Under \Cref{hyp:secNash}, suppose that
\[
\lambda \Gamma^{-1/2} \le \frac{\eps_*}{4\max\{2, R_*\}},
\quad
\Lambda > \lambda\pre,
\]
for some fixed $0 < \lambda < 1$. Then, it holds
\[
\Gamma\lambda\pre \le C_8 \nm{\grad f}_2,
\]
for some universal constant $C_8 > 0$ (independent even from $\Omega$).
\end{restatable}

We reserve the proofs of \Cref{prop:secNash_Belowlambda} and \Cref{prop:secNash_AbovelambdaInv} to \Cref{subsubsec:NashUnstratEasyCases}. Only the intermediate case remains, where $\lambda < \Lambda < \lambda\pre$ for some small parameter $0 < \lambda < 1$. It is in this scenario that we see the appearance of the $\dH_0\pre$ mixing norm.

\begin{restatable}{proposition}{BetweenlambdaAndlambdaInv}
\label{prop:secNash_BetweenlambdaAndlambdaInv}
Under \Cref{hyp:secNash}, suppose that
\[
\lambda^2 < 2^{-9}M_*\pre, \quad \lambda^3 \Gamma^{-1/2} \le \eps_*/4, \quad \nm{\wt{f}}_1 = 1, \quad \nm{\wt{f}}_2^2 \ge \kappa\Gamma, \quad \lambda \le \Lambda \le \lambda\pre,
\]
for some fixed $0 < \kappa \le 1$ and $0 < \lambda < 1$. Then, at least one of the following holds true:
\begin{equation}\label{aux130b}
\kappa^{5/2}\Gamma\lambda\pre  \le C_9\nm{\grad f}_2
\quad \text{or} \quad
\kappa^{21/2}\Gamma^{1/2}\lambda^{52} \le C_{10}\nm{\partial_1 f}_{\dH_0^{-1}},
\end{equation}
for some universal constants $C_9 = C_9(\Omega) > 0$ and $C_{10} > 0$ (independent even from $\Omega$).
\end{restatable}

We prove \Cref{prop:secNash_BetweenlambdaAndlambdaInv} across \Cref{subsubsec:NashUnstratHardCasesPart1,subsubsec:NashUnstratHardCasesPart2} as it is much more subtle than the other auxiliary propositions. For now, we leverage these three propositions by choosing the parameter $\lambda > 0$ so that the inequality $\kappa^{21/2}\Gamma^{1/2}\lambda^{52} \lesssim \nm{\partial_1 f}_{\dH_0^{-1}}$ cannot be satisfied. As a consequence, all the inequalities given by \Cref{prop:secNash_Belowlambda}, \Cref{prop:secNash_AbovelambdaInv}, and \Cref{prop:secNash_BetweenlambdaAndlambdaInv} will match to prove the following desired result.

\begin{theorem}
\label{thm:secNash_NashOnUnstratified}
Assume \Cref{hyp:secNash}. Then we have
\[
\nm{\wt{f}}_2^{5+\frac{22}{52}}
\le
C_{11} \nm{f-f_M}_2^{3+\frac{21}{52}} \nm{\partial_1 f}_{\dH_0\pre}^{\frac{1}{52}} \nm{f-f_M}_1 \nm{\grad f}_2,
\]
where $C_{11} = C_{11}(\Omega) > 0$ is a universal constant.
\end{theorem}

\begin{proof}
Let us first consider when $\nm{\wt{f}}_1 = 1$ and in turn define
\begin{equation}\label{aux130c}
\kappa := \nm{\wt{f}}_2^2 \Gamma\pre = \nm{\wt{f}}_2^2\nm{f-f_M}_2^{-2},
\quad
\lambda :=  \paren*{2C_{10}\nm{\partial_1 f}_{\dH_0^{-1}} \nm{\wt{f}}_2\pre \kappa^{-21/2}}^{1/52} > 0,
\end{equation}
where $C_{10}$ is the same as in \eqref{aux130b}. Then $0 < \kappa \le 1$ and trivially $\nm{\wt{f}}_2^2 \ge \kappa \Gamma$. Moreover, at least one of the following holds true: either
\begin{equation}\label{eq:secNash_UnstratNashCase3}
\lambda < 2^{-9/2} M_*^{-1/2} < 1
\quad \text{and} \quad
\lambda\Gamma^{-1/2} < \frac{\eps_*}{4\max\{2,R_*\}},
\end{equation}
or
\begin{equation}\label{eq:secNash_UnstratNashCase1}
\lambda \ge 2^{-9/2} M_*^{-1/2},
\end{equation}
or
\begin{equation}\label{eq:secNash_UnstratNashCase2}
\lambda\Gamma^{-1/2} \ge \frac{\eps_*}{4\max\{2,R_*\}}.
\end{equation}
If (\ref{eq:secNash_UnstratNashCase3}) holds, we will be able to consider how $\Lambda$ compares to $\lambda$ and invoke either \Cref{prop:secNash_Belowlambda}, \Cref{prop:secNash_AbovelambdaInv}, or \Cref{prop:secNash_BetweenlambdaAndlambdaInv}. If instead (\ref{eq:secNash_UnstratNashCase1}) or (\ref{eq:secNash_UnstratNashCase2}) holds, we will simply invoke the classical two-dimensional Nash inequality (\ref{eq:secIntro_Nash}).

Suppose (\ref{eq:secNash_UnstratNashCase3}) is true. In the case of $\Lambda < \lambda$, invoking \Cref{prop:secNash_Belowlambda} with $S := [0, h]$ and substituting the definitions of $\kappa$, $\Gamma$, and $\lambda$ gives
\begin{align*}
\nm{\wt{f}}_2^{5+\frac{22}{52}}
&\le
C_7(2C_{10})^{1/52}\nm{f-f_M}_2^{3+\frac{21}{52}}\nm{\partial_1 f}_{\dH_0\pre}^{1/52}\nm{\partial_1 f}_2 \\
&\le 2C_7(2C_{10})^{1/52}\nm{f-f_M}_2^{3+\frac{21}{52}}\nm{\partial_1 f}_{\dH_0\pre}^{1/52}\nm{f-f_M}_1\nm{\grad f}_2,
\end{align*}
where the second line holds by $1 = \nm{\wt{f}}_1 \le 2\nm{f-f_M}_1$; see \Cref{subsec:FuncDecomp}.

If instead $\Lambda > \lambda\pre$, then condition (\ref{eq:secNash_UnstratNashCase3}) allows us to invoke \Cref{prop:secNash_AbovelambdaInv}, giving us
\[
\nm{\wt{f}}_2^{\frac{22}{52}}
\le
C_8(2C_{10})^{1/52}\nm{f-f_M}_2^{-2+\frac{21}{52}}\nm{\partial_1 f}_{\dH_0\pre}^{1/52}\nm{\partial_1 f}_2,
\]
Multiplying both sides by $\nm{\wt{f}}_2^5 \le \nm{f-f_M}_2^5$ and again using $1 = \nm{\wt{f}}_1 \le 2\nm{f-f_M}_1$, we obtain
\[
\nm{\wt{f}}_2^{5+\frac{22}{52}}
\le 2C_8(2C_{10})^{1/52}\nm{f-f_M}_2^{3+\frac{21}{52}}\nm{\partial_1 f}_{\dH_0\pre}^{1/52}\nm{f-f_M}_1\nm{\grad f}_2.
\]
In the intermediate case of $\lambda \le \Lambda \le \lambda\pre$, note that (\ref{eq:secNash_UnstratNashCase3}) also ensures $\lambda^3\Gamma^{-1/2}\le \eps_*/4$. In turn,  \Cref{prop:secNash_BetweenlambdaAndlambdaInv} implies that at least one of the following holds true:
\[
\kappa^{5/2} \Gamma \lambda\pre \le C_9\nm{\grad f}_2
\quad \text{ or } \quad
\kappa^{21/2}\Gamma^{1/2}\lambda^{52} \le C_{10}\nm{\partial_1 f}_{\dH_0^{-1}}.
\]
The choice of $\lambda$ does not allow the right-hand inequality to be true, meaning the left-hand inequality must hold. We obtain the same desired inequality:
\[
\nm{\wt{f}}_2^{5+\frac{22}{52}}
\le 2C_9(2C_{10})^{1/52}\nm{f-f_M}_2^{3+\frac{21}{52}}\nm{\partial_1 f}_{\dH_0\pre}^{1/52}\nm{f-f_M}_1\nm{\grad f}_2.
\]
The proof is thus complete the in the case that (\ref{eq:secNash_UnstratNashCase3}) is true.

Let us consider the remaining situation where $\nm{\wt{f}}_1 = 1$ holds and (\ref{eq:secNash_UnstratNashCase3}) fails. Here, we recall the classical Nash inequality (\ref{eq:secIntro_Nash}) with dimension $d = 2$, which implies
\begin{equation}\label{eq:secNash_UnstratClassical}
\nm{\wt{f}}_2^5 \le \nm{f-f_M}_2^5 \le C_{\mathrm{N},2}\nm{f-f_M}_2^3\nm{f-f_M}_1\nm{\grad f}_2.
\end{equation}
If (\ref{eq:secNash_UnstratNashCase1}) is true, then we can multiply both sides of (\ref{eq:secNash_UnstratClassical}) by
\[
\nm{\wt{f}}_2^{22/52}
\le
2^{9/2} M_*^{1/2} (2C_{10})^{1/52} \nm{\partial_1 f}_{\dH_0^{-1}}^{1/52}\nm{f-f_M}_2^{21/52},
\]
giving us that
\[
\nm{\wt{f}}_2^{5+\frac{22}{52}}
\le
2^{9/2} M_*^{1/2} (2C_{10})^{1/52} C_{\mathrm{N},2} \nm{f-f_M}_2^{3+\frac{21}{52}} \nm{\partial_1 f}_{\dH_0^{-1}}^{1/52}\nm{f-f_M}_1\nm{\grad f}_2.
\]
The final case to consider is when (\ref{eq:secNash_UnstratNashCase2}) holds, which implies that
\begin{align*}
(2C_{10})^{1/52}\nm{\partial_1 f}_{\dH_0^{-1}}^{1/52}\nm{f-f_M}_2^{21/52}\nm{\wt{f}}_2^{-22/52}
&> \frac{\eps_*}{4\max\{2,R_*\}}\Gamma^{1/2} \\
&\ge \frac{\eps_*}{4\max\{2,R_*\}}\nm{\wt{f}}_2 \\
&\ge |\Omega|^{-1/2}\frac{\eps_*}{4\max\{2,R_*\}},
\end{align*}
where we have used Cauchy--Schwarz and the assumption that $\nm{\wt{f}}_1=1$. It follows from  (\ref{eq:secNash_UnstratClassical}) that
\[
\nm{\wt{f}}_2^{5+\frac{22}{52}}
\le
4\max\{2, R_*\} |\Omega|^{1/2}\eps_*\pre (2C_{10})^{1/52} C_{\mathrm{N},2} \nm{\partial_1 f}_{\dH_0^{-1}}^{1/52}\nm{f-f_M}_2^{3+\frac{21}{52}}\nm{f-f_M}_1\nm{\grad f}_2.
\]

To conclude the proof, we take
\[
C_{11} = C_{11}(\Omega) := 2(2C_{10})^{\frac{1}{52}}
\max\curly*{
C_7,
C_8,
C_9,
2^{7/2}M_*^{1/2} C_{\mathrm{N},2},
2\max\{2, R_*\} |\Omega|^{1/2}\eps_*\pre C_{\mathrm{N},2}
}.
\]
In summary, as long as $\nm{\wt{f}}_1 = 1$, it holds in any case that
\[
\nm{\wt{f}}_2^{5+\frac{22}{52}}
\le
C_{11}\nm{\partial_1 f}_{\dH_0^{-1}}^{1/52}\nm{f-f_M}_2^{3+\frac{21}{52}}\nm{f-f_M}_1\nm{\grad f}_2.
\]
As this inequality is scaling homogeneous, it continues to hold when the assumption $\nm{\wt{f}}_1 = 1$ is dropped.
\end{proof}

\begin{remark}
Suppose one does not like the appearance of the $L^2$ norm on the right hand side of \Cref{thm:secNash_NashOnUnstratified}, for whatever reason. One can simply invoke the classical Nash inequality (\ref{eq:secIntro_Nash}) to obtain
\[
\nm{\wt{f}}_2^{5+\frac{22}{52}}
\le
C_{11} C_{\mathrm{N},2}^{\frac{3}{2}+\frac{21}{104}} \nm{\partial_1 f}_{\dH_0\pre}^{\frac{1}{52}} \nm{f-f_M}_1^{\frac{5}{2}+\frac{21}{104}} \nm{\grad f}_2^{\frac{5}{2}+\frac{21}{104}}.
\]
\end{remark}

\begin{remark}
\Cref{thm:secNash_NashOnUnstratified} implies that
\[
\nm{f-f_M}_2^{2+\frac{1}{52}} \le C_{11} \paren*{\frac{\nm{f-f_M}_2^2}{\nm{\wt{f}}_2^2}}^{\frac{5}{2}+\frac{22}{104}} \nm{\partial_1 f}_{\dH_0\pre}^{\frac{1}{52}} \nm{f-f_M}_1 \nm{\grad f}_2.
\]
This improves upon the classical Nash inequality (\ref{eq:secIntro_Nash}) by offering extra control via the $\dH_0\pre$ norm of $\partial_1 f$, at the cost of a function dependent coefficient: $C_{11}(\nm{f-f_M}_2^2/\nm{\wt{f}}_2^2)^{\frac{5}{2}+\frac{22}{104}} \ge C_{11}$. This is indeed an improvement since it is easily seen by the Calder\'{o}n--Zygmund estimate that
\[
\nm{\partial_1 f}_{\dH_0\pre}^2
= \nm{\partial_1 \wt{f}}_{\dH_0\pre}^2
= -\gen{\wt{f}, \partial_1 (-\lap_D)\pre \partial_1 \wt{f}}
\le \nm{\wt{f}}_2 \nm{\partial_1 (-\lap_D)\pre \partial_1 \wt{f}}_2
\lesssim_\Omega \nm{\wt{f}}_2^2
\le \nm{f-f_M}_2^2.
\]
See \Cref{subsec:Notation} for the precise definition of the negative Sobolev norm $\dH_0\pre$. The improvement worsens for highly stratified functions $f$, as reflected by the function-dependent coefficient, with the inequality degenerating whenever $\wt{f}$ is constantly zero.
\end{remark}

\subsubsection{Proofs of auxiliary propositions for the cases of small and large \texorpdfstring{$\Lambda$}{Lambda}}
\label{subsubsec:NashUnstratEasyCases}

Once again recall that the definition of $\Lambda$ in \eqref{aux129a} reflects exactly the scaling required to prove \Cref{thm:secNash_NashOnUnstratified}. Indeed, if there is some parameter $0 < \lambda < 1$ such that $\Lambda < \lambda$, we prove in \Cref{prop:secNash_Belowlambda} below that an improved Nash inequality holds which incorporates the value of $\lambda$. If instead $\Lambda > \lambda\pre$, the proof of \Cref{prop:secNash_AbovelambdaInv} to follow also reveals an improved Nash inequality depending on $\lambda$. We loosely refer to these cases as $\Lambda$ being ``small'' and ``large'', respectively, and we emphasize that the estimates do not require a specific value of $\lambda$ to hold true.

The case of small $\Lambda$ is handled by careful usage of the classical one-dimensional Nash inequality (\ref{eq:secIntro_Nash}) along with a precise H\"{o}lder inequality. We restate \Cref{prop:secNash_Belowlambda} for the reader's convenience, which does consider a slightly more general case than $\Lambda < \lambda$ that will be useful later for proving \Cref{prop:secNash_BetweenlambdaAndlambdaInv}; see \Cref{lem:secNash_BetweenDeltaAndDeltaInvStep1}.

\Belowlambda*

\begin{proof}
We note here that $\partial_1 \wt{f} = \partial_1 f$ since by definition $\wt{f}(x_1,x_2) = f(x_1,x_2) - \ov{f}(x_2)$. We choose to work with $\wt{f}$ below.

For any integer $i \ge 1$, we define
\[
S_i := \curly*{
x_2 \in S\
\bigg|\
2^{-i} \le \Gamma^{-3/2}\nm{\wt{f}(\cdot, x_2)}_{L^2(I_{x_2})}^2 \le 2^{-i+1}
},
\quad
a_i := |S_i|.
\]
Letting $i_{-} := \ceil{\log(\lambda\pre)/\log(2)}$, we note that $S_i$ must be empty for each $1 \le i < i_{-}$ by the hypothesis of this proposition (if $i_{-} = 1$, then this statement is vacuous). Indeed, if there were some $x_2 \in S_i \subseteq S$ for such a value of $i$, then we would obtain the contradiction
\[
\Gamma^{-3/2}\nm{\wt{f}(\cdot,x_2)}_{L^2(I_{x_2})}^2 \ge 2^{-(i_{-}-1)} \ge \lambda .
\]
We also define
\[
i_+ := \ceil{\log\paren*{2 h \kappa\pre \Gamma^{1/2}}/\log(2)},
\quad
\mathcal{I} := \{i_{-} \le i \le i_+\ |\ a_i \neq 0\},
\quad
S' := \bigcup_{i \in \mathcal{I}} S_i,
\]
so that each $x_2 \in S \setminus S'$ either lies in some $S_i$ that is measure-zero or is such that
\[
\nm{\wt{f}(\cdot,x_2)}_{L^2(I_{x_2})}^2 < 2^{-i_+}\Gamma^{3/2} \le \frac{\kappa}{2h} \Gamma.
\]
It follows that
\[
\int_{S\setminus S'} \nm{\wt{f}(\cdot,x_2)}_{L^2(I_{x_2})}^2\, dx_2 \le \frac{1}{2}\kappa\Gamma.
\]
Note this implies that, given our hypotheses, $i_{+}$ indeed must be greater than or equal to $1$ and it must hold that $i_{+} \ge i_{-}$. Moreover, we obtain
\begin{equation}\label{eq:secNash_SumOverI}
\frac{1}{4}\kappa \Gamma \le \frac{1}{2} \int_{S'} \nm{\wt{f}(\cdot,x_2)}_{L^2(I_{x_2})}^2\, dx_2 \le \sum_{i \in \mathcal{I}} a_i2^{-i}\Gamma^{3/2}.
\end{equation}
We will use (\ref{eq:secNash_SumOverI}) in the final computation of this proof. We will also need the following instance of H\"{o}lder's inequality
\[
\nm{1}_{L^{8/5}(\R,\mu)} = \nm{ |g|^{1/2} |1/g|^{1/2} }_{L^{8/5}(\R,\mu)}
\le \nm{ |g|^{1/2} }_{L^2(\R,\mu)} \nm{ |1/g|^{1/2} }_{L^8(\R,\mu)} = \nm{g}_{L^1(\R,\mu)}^{1/2}\nm{1/g}_{L^4(\R,\mu)}^{1/2},
\]
as it implies
\begin{equation}\label{eq:secNash_InterpolationHolder}
\mu(\R)^5
= \nm{1}_{L^{8/5}(\R,\mu)}^8
\le \nm{g}_{L^1(\R,\mu)}^4\nm{1/g}_{L^4(\R,\mu)}^4
\end{equation}
for any measure $\mu$ on $\R$ and any function $g \in L^1(\R,\mu) \cap L^4(\R,\mu)$ with $g > 0$ satisfied $\mu$-almost everywhere.

To complete the proof, we perform the following sequence of lower bounds. We first observe that
\[
\int_S \nm{\partial_1 \wt{f}(\cdot,x_2)}_{L^2(I_{x_2})}^2\, dx_2
\ge
\int_{S'} \nm{\partial_1 \wt{f}(\cdot,x_2) }_{L^2(I_{x_2})}^2\, dx_2
\ge
C_{\mathrm{N},1}^{-2}\int_{S'} \frac{\nm{\wt{f}(\cdot,x_2) }_{L^2(I_{x_2})}^6}{\nm{\wt{f}(\cdot,x_2)}_{L^1(I_{x_2})}^4}\, dx_2
\]
by the classical Nash inequality (\ref{eq:secIntro_Nash}) with dimension $d = 1$. By the construction of $S'$, it follows that
\[
\int_S \nm{\partial_1 \wt{f}(\cdot,x_2)}_{L^2(I_{x_2})}^2\, dx_2
\ge
C_{\mathrm{N},1}^{-2}\sum_{i \in \mathcal{I}} 2^{-3i}\Gamma^{9/2} \int_{S_i} \frac{1}{\nm{\wt{f}(\cdot,x_2)}_{L^1(I_{x_2})}^4}\, dx_2. \]
We may now apply the H\"{o}lder inequality (\ref{eq:secNash_InterpolationHolder}) with $\mu$ being the Lebesgue measure restricted to $S_i$ and $g(x_2) := \nm{\wt{f}(\cdot,x_2)}_{L^1(I_{x_2})} > 0$, giving us
\[
\int_S \nm{\partial_1 \wt{f}(\cdot,x_2)}_{L^2(I_{x_2})}^2\, dx_2
\ge
C_{\mathrm{N},1}^{-2}\sum_{i \in \mathcal{I}} 2^{-3i}\Gamma^{9/2} a_i^5 \paren*{\int_{S_i}\nm{\wt{f}(\cdot, x_2)}_{L^1(I_{x_2})}\, dx_2 }^{-4}.
\]
Grouping the terms in the sum, we get
\begin{align*}
\int_S \nm{\partial_1 \wt{f}(\cdot,x_2)}_{L^2(I_{x_2})}^2\, dx_2
&\ge
C_{\mathrm{N},1}^{-2}\Gamma^2\sum_{i \in \mathcal{I}} 2^{2i} \paren*{a_i2^{-i}\Gamma^{1/2}} \paren*{\frac{\int_{S_i}\nm{\wt{f}(\cdot, x_2)}_{L^1(I_{x_2})}\, dx_2}{a_i2^{-i}\Gamma^{1/2}} }^{-4} \\
&\ge
C_{\mathrm{N},1}^{-2}\lambda^{-2}\Gamma^2\sum_{i \in \mathcal{I}} \paren*{a_i2^{-i}\Gamma^{1/2}} \paren*{\frac{\int_{S_i}\nm{\wt{f}(\cdot, x_2)}_{L^1(I_{x_2})}\, dx_2}{a_i2^{-i}\Gamma^{1/2}} }^{-4},
\end{align*}
since $2^{2i} \ge \lambda^{-2}$ for any $i \in \mathcal{I}$ by construction. We again employ (\ref{eq:secNash_InterpolationHolder}), where this time we take $\mu$ and $g$ to be the following discrete measure and function on $\mathcal{I} \subseteq \R$, respectively:
\[
\mu = \sum_{i \in \mathcal{I}} a_i2^{-i}\Gamma^{1/2} \delta_i, \quad g(i) = \frac{\int_{S_i}\nm{\wt{f}(\cdot, x_2)}_{L^1(I_{x_2})}\, dx_2}{a_i 2^{-i} \Gamma^{1/2}} > 0,
\]
with $\delta_i$ denoting the Dirac delta measure at $i \in \mathcal{I} \subseteq \R$. Together with (\ref{eq:secNash_SumOverI}), this yields
\begin{align*}
\int_S \nm{\partial_1 \wt{f}(\cdot,x_2)}_{L^2(I_{x_2})}^2\, dx_2
&\ge
C_{\mathrm{N},1}^{-2}\lambda^{-2}\Gamma^2 \paren*{\sum_{i \in \mathcal{I}} a_i2^{-i}\Gamma^{1/2}}^5 \paren*{\sum_{i \in \mathcal{I}}\int_{S_i}\nm{\wt{f}(\cdot, x_2)}_{L^1(I_{x_2})}\, dx_2 }^{-4} \\
&\ge 2^{10}C_{\mathrm{N},1}^{-2}\lambda^{-2}\Gamma^2 \kappa^5 \paren*{\int_{S'} \nm{\wt{f}(\cdot, x_2)}_{L^1(I_{x_2})}\, dx_2 }^{-4}.
\end{align*}
Since $\nm{\wt{f}}_1 = 1$, it follows that
\[
\int_S \nm{\partial_1 \wt{f}(\cdot,x_2)}_{L^2(I_{x_2})}^2\, dx_2
\ge
2^{10}C_{\mathrm{N},1}^{-2}\lambda^{-2}\Gamma^2 \kappa^5.
\]
The proof is complete with the universal constant given by $C_7 = 2^{-5} C_{\mathrm{N},1}$.
\end{proof}

We now work to prove the case of large $\Lambda$, which we will handle very similarly to how we did it in the proofs of \Cref{prop:secNash_StratCaps,prop:secNash_StratBulk}, the auxiliary results needed for the Nash stratification inequality \Cref{thm:secNash_NashOnStratified}. Indeed, we again attempt to lower bound the gradient $\grad f$ in $L^2$ by finding a drop in an appropriate cross-sectional norm of $f$. We also must consider two cases: if the drop occurs in the caps or in the bulk of the domain. The following lemma treats the first case concerning the caps.

\begin{lemma}\label{lem:secNash_Sit1AbovelambdaInv}
Under \Cref{hyp:secNash}, suppose that
\[
\lambda \Gamma^{-1/2} \le \frac{\eps_*}{4\max\{2, R_*\}}, \quad
\nm{\wt{f}(\cdot,x_2')}_{L^2(I_{x_2'})}^2 > \lambda\pre \Gamma^{3/2},
\]
for some fixed constant $0 < \lambda < 1$ and some fixed height $x_2' \in (0, \eps_*/4]$. Then we have
\[
\Gamma\lambda\pre \le 4 \nm{\partial_2 f}_2.
\]
\end{lemma}

\begin{proof}
We first note that we can find $x_2''$ such that
\[
x_2' < x_2'' \le x_2' + 4\lambda\Gamma^{-1/2} < \eps_*, \quad \int_{I_{x_2''}} |f(x_1, x_2'')-f_M|^2\, dx_1 \le \frac{1}{2}\lambda\pre \Gamma^{3/2},
\]
as otherwise we obtain the contradiction
\[
\Gamma = \nm{f-f_M}_2^2 \ge \int_{x_2'}^{x_2'+4\lambda\Gamma^{-1/2}} \int_{I_{x_2}} |f(x_1,x_2)-f_M|^2\, dx_1\, dx_2 > 2\Gamma.
\]
It is important to note that $x_2' < x_2'' < \eps_*$ implies $I_{x_2'} \subseteq I_{x_2} \subseteq I_{x_2''}$ for all $x_2' \le x_2 \le x_2''$, by \Cref{as:secSett_3} and \Cref{norm:secSett_eps}. With this choice of $x_2''$, we can now compute
\begin{align*}
2\int_{I_{x_2'}} \int_{x_2'}^{x_2''} |\partial_2 f||f-f_M|\, dx
&\ge \int_{I_{x_2'}}\abs*{\int_{x_2'}^{x_2''} \partial_2 (|f-f_M|^2)\, dx_2} dx_1 \\
&= \int_{I_{x_2'}}\left| |f(x_1,x_2'')-f_M|^2-|f(x_1,x_2')-f_M|^2 \right| dx_1 \\
&\ge \int_{I_{x_2'}} \abs{f(x_1,x_2')-f_M}^2 dx_1 - \int_{I_{x_2'}}\abs{f(x_1,x_2'')-f_M}^2 dx_1,
\end{align*}
from which it follows that
\begin{align*}
2\int_{I_{x_2'}} \int_{x_2'}^{x_2''} |\partial_2 f||f-f_M|\, dx
&\ge
\int_{I_{x_2'}} \abs{f(x_1,x_2')-f_M}^2 dx_1 - \int_{I_{x_2''}}\abs{f(x_1,x_2'')-f_M}^2 dx_1 \\
&\ge
\frac{1}{2}\lambda\pre\Gamma^{3/2},
\end{align*}
where we have again used that $I_{x_2'} \subseteq I_{x_2''}$. The proof is concluded by the Cauchy--Schwarz inequality:
\begin{align*}
\int_\Omega |\partial_2 f|^2\, dx &\ge \int_{I_{x_2'}}\int_{x_2'}^{x_2''} |\partial_2 f|^2\, dx \\
&\ge \paren*{\int_{I_{x_2'}} \int_{x_2'}^{x_2''} |\partial_2 f||f-f_M|\, dx}^2\paren*{\int_{I_{x_2'}} \int_{x_2'}^{x_2''} |f-f_M|^2\, dx}\pre \\
&\ge \paren*{\frac{1}{4}\lambda\pre\Gamma^{3/2}}^2 \Gamma\pre\\
&= \frac{1}{16}\lambda^{-2}\Gamma^2.
\end{align*}
\end{proof}

The second case needed for \Cref{prop:secNash_AbovelambdaInv} concerning the bulk of the domain is handled by the following lemma.

\begin{lemma}\label{lem:secNash_Sit2AbovelambdaInv}
Under \Cref{hyp:secNash}, suppose that
\[
\lambda \Gamma^{-1/2} \le \frac{\eps_*}{4\max\{2, R_*\}}, \quad
\nm{\wt{f}(\cdot,x_2')}_{L^2(I_{x_2'})}^2 > \lambda\pre \Gamma^{3/2},
\]
for some fixed constant $0 < \lambda < 1$ and some fixed height $x_2' \in (\eps_*/4, h/2]$. Then, we have that
\[
\Gamma\lambda\pre \le 4 \nm{\grad f}_2^2.
\]
\end{lemma}

\begin{proof}
By the same reasoning as in \Cref{lem:secNash_Sit1AbovelambdaInv}, we can find $x_2''$ such that
\[
x_2' < x_2'' \le x_2' + 4R_*\lambda\Gamma^{-1/2} \le 3h/4, \quad \int_{I_{x_2''}} |f(x_1, x_2'')-f_M|^2\, dx_1 \le \frac{1}{2R_*}\lambda\pre \Gamma^{3/2},
\]
as otherwise we obtain the contradiction
\[
\Gamma
= \nm{f-f_M}_2^2
\ge \int_{x_2'}^{x_2'+4R_*\lambda\Gamma^{-1/2}} \int_{I_{x_2}} |f-f_M|^2\, dx_1\, dx_2
> 2\Gamma.
\]
Note that  \Cref{norm:secSett_eps} ensures $x_2' < x_2'' \le 3h/4 \le h-\eps_*/4$ since $\eps_* \leq h/4$. Next, just as we did in the proof of \Cref{lem:secNash_Sit2StratBulk} for \Cref{prop:secNash_StratBulk}, we can explicitly construct the following foliation of the domain:
\[
\Phi_t(x_1) := \paren*{
F_\ell(x_2'+t) + \paren*{x_1-F_\ell(x_2')}\frac{|I_{x_2'+t}|}{|I_{x_2'}|},\ x_2' + t
}
\]
where $x_1 \in I_{x_2'}$ and $0 \le t \le x_2'' - x_2'$, meaning the curves $t \mapsto \Phi_t(x_1)$ start at $I_{x_2'} \times \{x_2'\}$ and end at $I_{x_2''} \times \{x_2''\}$; see \Cref{fig:secNash_NashStratCurves2}. We may now integrate over the family of curves $\phi_{x_1} := \{\Phi_t(x_1)\ |\ 0 \le t \le x_2'' - x_2'\}$ indexed by $x_1 \in I_{x_2'}$ as follows (compare to \eqref{aux129f}):
\begin{align*}
&2\int_{I_{x_2'}} \int_{\phi_{x_1}} |\grad f(z)||f(z)-f_M|\, dz\, dx_1\\
&\quad\ge \int_{I_{x_2'}} \abs*{ \abs{f\circ\Phi_{x_2''-x_2'}(x_1) - f_M}^2 - \abs{f\circ\Phi_{0}(x_1) - f_M}^2}\, dx_1 \\
&\quad\ge \int_{I_{x_2'}} \abs{f(x_1,x_2')-f_M}^2\, dx_1 - \int_{I_{x_2'}} \abs{f\circ\Phi_{x_2''-x_2'}(x_1)-f_M}^2\, dx_1 \\
&\quad= \int_{I_{x_2'}} \abs{f(x_1,x_2')-f_M}^2\, dx_1  - \frac{|I_{x_2'}|}{|I_{x_2''}|}\int_{I_{x_2''}} \abs*{f(x_1,x_2'')-f_M}^2\, dx_1 \\
&\quad\ge \lambda\pre\Gamma^{3/2} - \frac{|I_{x_2'}|}{2R_*|I_{x_2''}|}\lambda\pre\Gamma^{3/2} \\
&\quad\ge \frac{1}{2}\lambda\pre\Gamma^{3/2},
\end{align*}
where the final inequality holds by (\ref{aux0307a}) from \Cref{def:secSett_K*R*} and the fact $\eps_*/4 < x_2' < x_2'' \le h-\eps_*/4$. The Cauchy--Schwarz inequality completes the proof here, just as was done in \Cref{lem:secNash_Sit1AbovelambdaInv}, giving us
\[
\int_\Omega |\grad f|^2\, dx
\ge
\frac{1}{16}\lambda^{-2}\Gamma^2.
\]
\end{proof}

We combine \Cref{lem:secNash_Sit1AbovelambdaInv,lem:secNash_Sit2AbovelambdaInv}  to prove \Cref{prop:secNash_AbovelambdaInv}, which we restate below.

\AbovelambdaInv*

\begin{proof}
By the definition of $\Lambda$, we may pick $x_2' \in (0, h)$ such that
\[
\nm{\wt{f}(\cdot,x_2')}_{L^2(I_{x_2'})}^2 > \lambda\pre \Gamma^{3/2}.
\]
If $x_2' \in (0, \eps_*/4]$, we may apply \Cref{lem:secNash_Sit1AbovelambdaInv}. If $x_2' \in (\eps_*/4, h/2]$, we may apply \Cref{lem:secNash_Sit2AbovelambdaInv}. In the situations $x_2' \in [h-\eps_*/4, h)$ or $x_2' \in [h/2, h-\eps_*/4)$, we can perform completely analogous arguments. Thus, collecting the conclusions of all possible situations, the proof is complete with the universal constant given by $C_8 = 4$.
\end{proof}

\subsubsection{Test function construction for the case of intermediate \texorpdfstring{$\Lambda$}{Lambda}}
\label{subsubsec:NashUnstratHardCasesPart1}

As we continue to work under \Cref{hyp:secNash}, it remains to prove \Cref{prop:secNash_BetweenlambdaAndlambdaInv}. The proof of this proposition is significantly longer and more subtle that the proofs of \Cref{prop:secNash_Belowlambda,prop:secNash_AbovelambdaInv}, with it spanning \Cref{subsubsec:NashUnstratHardCasesPart1} and \Cref{subsubsec:NashUnstratHardCasesPart2}. Later in \Cref{subsubsec:NashUnstratHardCasesPart2}, we show that either an inequality of the form $\kappa^{5/2} \Gamma \lambda\pre \lesssim \nm{\partial_1 f}_2$ must hold true or the function $f$ satisfies the following further hypotheses.

\begin{hypotheses}\label{hyp:secNash_Lambda}
Under \Cref{hyp:secNash} and for fixed $0 < \kappa \le 1$ and $0 < \lambda < 1$, suppose that
\[
\nm{\wt{f}}_1 = 1, \quad \nm{\wt{f}}_2^2 \ge \kappa \Gamma, \quad \lambda \le \Lambda \le \lambda\pre.
\]
Furthermore, suppose there exists a fixed height $x_2' \in (0, h/2]$ at which $f$ satisfies the following.
\begin{enumerate}[label={(\alph*)},ref={Hypotheses \thehypotheses\alph*}]
\item We have
\[
\nm{\wt{f}(\cdot, x_2')}_{L^1(I_{x_2'})} \le 16\kappa\pre \Lambda\Gamma^{1/2}.
\]
\label{hyp:secNash_Lambda_PartA}
\item Let
\begin{equation}\label{aux131a}
\alpha := \frac{1}{16}\kappa \lambda\Lambda\pre.
\end{equation}
There exist disjoint closed connected subintervals $J^-, J^+ \subseteq I_{x_2'}$ such that
\[
\frac{9}{10}\alpha \Gamma \le |\wt{f}(x_1,x_2')| \le \frac{11}{10}\alpha \Gamma
\]
for all $x_1 \in J^+$ and
\[
|\wt{f}(x_1,x_2')| \le \frac{1}{2}\alpha \Gamma
\]
for all $x_1 \in J^-$.
\label{hyp:secNash_Lambda_PartB}
\item Let
\begin{equation}\label{aux131b}
\beta := \lambda^2 \alpha^2 \Lambda\pre = 2^{-8}\kappa^2 \lambda^4 \Lambda^{-3}.
\end{equation}
Then $J^+$ and $J^-$ each have length exactly $\beta \Gamma^{-1/2}$ and each is at least a distance of $\beta \Gamma^{-1/2}$ away from the boundary $\partial I_{x_2'}$. We denote the center of $J^+$ by $x_1'$ and the center of $J^-$ by $x_1''$.
\label{hyp:secNash_Lambda_PartC}
\item It holds that $|\wt{f}(x_1,x_2')| > \frac{1}{10}\alpha\Gamma$ for all $x_1$ strictly between $x_1'$ and $x_1''$.
\label{hyp:secNash_Lambda_PartD}
\end{enumerate}
\end{hypotheses}

Reducing the proof of \Cref{prop:secNash_BetweenlambdaAndlambdaInv} to the situation given by \Cref{hyp:secNash_Lambda} only requires very classical tools: mainly the Cauchy--Schwarz inequality and the fundamental theorem of calculus. Nonetheless, the argument is lengthy and requires notation, which is why we reserve it for \Cref{subsubsec:NashUnstratHardCasesPart2}. We note here that the specific values of the constants $\alpha$ and $\beta$ naturally arise from the arguments in \Cref{subsubsec:NashUnstratHardCasesPart2}.

In this section, we assume \Cref{hyp:secNash_Lambda} and present the heart of the proof of \Cref{prop:secNash_BetweenlambdaAndlambdaInv}. The idea is motivated by the variational formulation of $\dH_0\pre$ norm. Indeed, a simple integration by parts of (\ref{eq:secSett_HMixVariational}) from \Cref{subsec:Notation} gives us that
\begin{equation}\label{eq:secNash_BoundByPsi}
\nm{\partial_1 f}_{\dH_0^{-1}}
= \nm{\partial_1 \wt{f}}_{\dH_0^{-1}}
= \sup_{\psi \in \dH_0^1\setminus\{0\}} \frac{|\gen{\partial_1 \wt{f},\psi}|}{\nm{\psi}_{\dH_0^1}}
= \sup_{\psi \in \dH_0^1\setminus\{0\}} \frac{|\gen{\wt{f},\partial_1 \psi}|}{\nm{\psi}_{\dH_0^1}}
\ge \frac{|\gen{\wt{f},\partial_1 \Psi}|}{\nm{\Psi}_{\dH_0^1}},
\end{equation}
where $\Psi$ is any particular test function in $\dH_0^1(\Omega)$. Thus, we aim to construct a test function $\Psi$ that satisfies an appropriate lower bound on $|\gen{\wt{f},\partial_1 \Psi}|$ and an appropriate upper bound on $\nm{\Psi}_{\dH_0^1}$ that, when combined together with (\ref{eq:secNash_BoundByPsi}), imply $\kappa^{21/2} \Gamma^{1/2} \lambda^{52} \lesssim \nm{\partial_1 f}_{\dH_0\pre}$ holds true.

We begin with the following first step towards constructing the desired test function $\Psi$. Note that the first idea is to just use $\tilde f(\cdot, x_2')$ as the basis for test function construction. This approach fails as the $\dH_0^1$ norm of the corresponding test function could be too large. Instead, we will use the fact that the $L^2$ mass of $\tilde f(\cdot, x_2')$ has to be concentrated in $x_1$ in a certain way to reduce the $x_1$ region that we will use for the construction. The flavor of this argument may be similar to some ``atom'' or ``molecule'' test function constructions in Fourier analysis; see \cite{S93} for an overview and \cite{KN09} for another application of such techniques.

\tikzset{every picture/.style={line width=0.75pt}}
\begin{figure}[!ht]%
\centering
\begin{tikzpicture}[x=0.75pt,y=0.75pt,yscale=-1.1,xscale=1]
\draw    (20.09,167.55) -- (572.53,167.55) ;
\draw    (323.61,48.79) -- (323.73,285.6) ;
\draw [color={rgb, 255:red, 189; green, 16; blue, 224 }  ,draw opacity=0.5 ][line width=2.25]    (20.09,167.55) -- (166.04,167.8) ;
\draw [color={rgb, 255:red, 189; green, 16; blue, 224 }  ,draw opacity=0.5 ][line width=2.25]    (263.53,167.55) -- (345.28,167.8) ;
\draw [color={rgb, 255:red, 189; green, 16; blue, 224 }  ,draw opacity=0.5 ][line width=2.25]    (439.95,167.31) -- (572.53,167.55) ;
\draw  [draw opacity=0][line width=2.25]  (345.28,90.4) .. controls (345.72,67.53) and (367.46,49.16) .. (394.18,49.23) .. controls (421.16,49.31) and (442.99,68.15) .. (442.95,91.33) .. controls (442.95,91.33) and (442.95,91.34) .. (442.95,91.35) -- (394.11,91.19) -- cycle ; \draw  [color={rgb, 255:red, 189; green, 16; blue, 224 }  ,draw opacity=0.5 ][line width=2.25]  (345.28,90.4) .. controls (345.72,67.53) and (367.46,49.16) .. (394.18,49.23) .. controls (421.16,49.31) and (442.99,68.15) .. (442.95,91.33) .. controls (442.95,91.33) and (442.95,91.34) .. (442.95,91.35) ;
\draw  [draw opacity=0][line width=2.25]  (263.54,243.27) .. controls (262.95,266.14) and (241.1,284.3) .. (214.38,283.96) .. controls (187.41,283.62) and (165.69,264.56) .. (165.88,241.39) .. controls (165.88,241.38) and (165.88,241.37) .. (165.88,241.36) -- (214.71,242) -- cycle ; \draw  [color={rgb, 255:red, 189; green, 16; blue, 224 }  ,draw opacity=0.5 ][line width=2.25]  (263.54,243.27) .. controls (262.95,266.14) and (241.1,284.3) .. (214.38,283.96) .. controls (187.41,283.62) and (165.69,264.56) .. (165.88,241.39) .. controls (165.88,241.38) and (165.88,241.37) .. (165.88,241.36) ;
\draw    (216.6,158.98) -- (216.72,176.28) ;
\draw    (393.17,158.98) -- (393.28,176.28) ;
\draw [color={rgb, 255:red, 189; green, 16; blue, 224 }  ,draw opacity=1 ]   (317.04,88.87) -- (330.41,88.87) ;
\draw [color={rgb, 255:red, 189; green, 16; blue, 224 }  ,draw opacity=1 ]   (317.04,48.4) -- (330.41,48.4) ;
\draw [color={rgb, 255:red, 189; green, 16; blue, 224 }  ,draw opacity=1 ]   (317.04,243.44) -- (330.41,243.44) ;
\draw [color={rgb, 255:red, 189; green, 16; blue, 224 }  ,draw opacity=1 ]   (317.04,285.6) -- (330.41,285.6) ;
\draw   (169,177) .. controls (168.99,181.67) and (171.31,184.01) .. (175.98,184.02) -- (207.03,184.09) .. controls (213.7,184.11) and (217.02,186.45) .. (217.01,191.12) .. controls (217.02,186.45) and (220.36,184.13) .. (227.03,184.14)(224.03,184.13) -- (258.07,184.22) .. controls (262.74,184.23) and (265.08,181.9) .. (265.09,177.23) ;
\draw   (346,179) .. controls (345.99,183.67) and (348.31,186.01) .. (352.98,186.02) -- (384.03,186.09) .. controls (390.7,186.11) and (394.02,188.45) .. (394.01,193.12) .. controls (394.02,188.45) and (397.36,186.13) .. (404.03,186.14)(401.03,186.13) -- (435.07,186.22) .. controls (439.74,186.23) and (442.08,183.9) .. (442.09,179.23) ;
\draw   (571.09,131.86) .. controls (571.1,127.19) and (568.77,124.85) .. (564.1,124.84) -- (525.09,124.75) .. controls (518.42,124.74) and (515.09,122.4) .. (515.1,117.73) .. controls (515.09,122.4) and (511.76,124.72) .. (505.09,124.71)(508.09,124.71) -- (26.1,123.62) .. controls (21.43,123.61) and (19.1,125.93) .. (19.09,130.6) ;
\draw (221.16,136.89) node [anchor=north west][inner sep=0.75pt]    {$x_{1} ''$};
\draw (402.23,138.57) node [anchor=north west][inner sep=0.75pt]    {$x_{1} '$};
\draw (429.54,37.9) node [anchor=north west][inner sep=0.75pt]   [align=left] {\textcolor[rgb]{0.74,0.06,0.88}{Values of $\displaystyle \partial _{1} \Phi $}};
\draw (275.5,37) node [anchor=north west][inner sep=0.75pt]  [color={rgb, 255:red, 189; green, 16; blue, 224 }  ,opacity=1 ]  {$\frac{11}{10} \alpha \Gamma $};
\draw (275.5,77) node [anchor=north west][inner sep=0.75pt]  [color={rgb, 255:red, 189; green, 16; blue, 224 }  ,opacity=1 ]  {$\frac{9}{10} \alpha \Gamma $};
\draw (336.5,272) node [anchor=north west][inner sep=0.75pt]  [color={rgb, 255:red, 189; green, 16; blue, 224 }  ,opacity=1 ]  {$-\frac{11}{10} \alpha \Gamma $};
\draw (336.5,229) node [anchor=north west][inner sep=0.75pt]  [color={rgb, 255:red, 189; green, 16; blue, 224 }  ,opacity=1 ]  {$-\frac{9}{10} \alpha \Gamma $};
\draw (218,189.23) node [anchor=north west][inner sep=0.75pt]    {$J^{-}$};
\draw (396,190.23) node [anchor=north west][inner sep=0.75pt]    {$J^{+}$};
\draw (520,98.23) node [anchor=north west][inner sep=0.75pt]    {$I_{x_{2} '}$};
\end{tikzpicture}
\caption{Cartoon depiction of $\partial_1 \Phi$.}
\label{fig:secNash_PhiConstruction}
\end{figure}
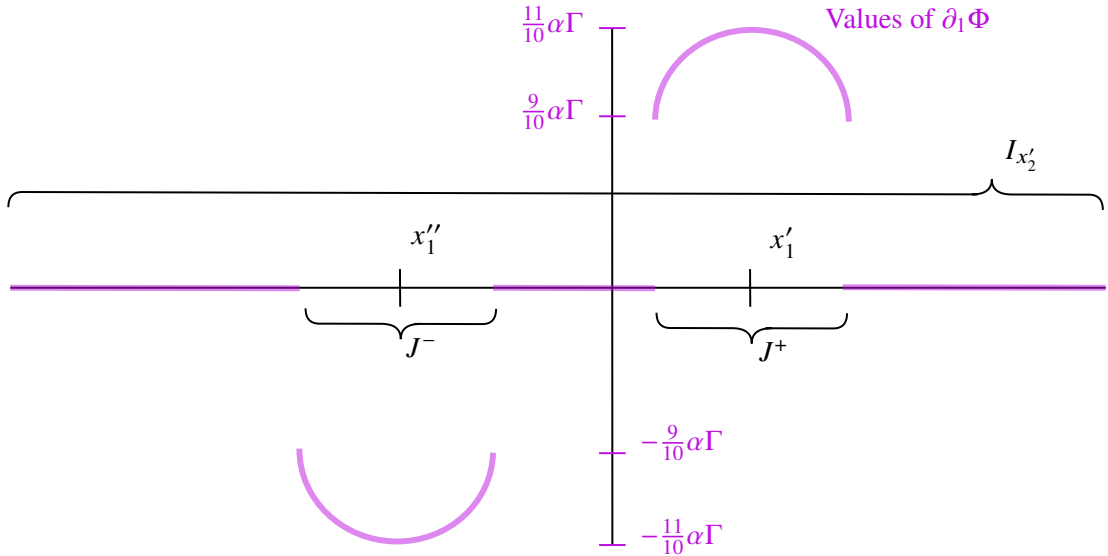

\begin{definition}
Under \Cref{hyp:secNash,hyp:secNash_Lambda}, we use integration
\[
\Phi(x_1) = \int_{F_\ell(x_2')}^{x_1} \partial_1 \Phi(z)\, dz
\]
in in order to construct a function $\Phi \in H^1_0(I_{x_2'})$ with the following derivative:
\[
\partial_1 \Phi(x_1) =
\begin{cases}
\wt{f}(x_1,x_2') & \text{for } x_1 \in J^+ \\
-\wt{f}(x_1-x_1''+x_1',x_2') & \text{for } x_1 \in J^- \\
0 & \text{for } x_1 \in I_{x_2'}\setminus (J^- \cup J^+).
\end{cases}
\]
See \Cref{fig:secNash_PhiConstruction} for an example of how $\partial_1 \Phi$ may look like.
\end{definition}

\begin{lemma}
Under \Cref{hyp:secNash,hyp:secNash_Lambda}, the function $\Phi$ satisfies the following useful properties.
\begin{enumerate}[label={(\alph*)},ref={Lemma \thelemma\alph*}]
\item $\partial_1 \Phi$ is mean zero over $I_{x_2'}$.
\label{lem:secNash_Phi_PartA}
\item The support $\supp \Phi$ is the smallest connected subinterval of $I_{x_2'}$ containing $J^+$ and $J^-$. Hence, $\Phi$ lies in $H_0^1(I_{x_2'})$ and is compactly supported away from the boundary $\partial I_{x_2'}$.
\label{lem:secNash_Phi_PartB}
\item We have $|\supp \Phi| \le 2\sigma\Gamma^{-1/2}$, where
\begin{equation}\label{aux131c} \sigma := 100 \alpha^{-2} \Lambda = 100 \cdot 2^8 \kappa^{-2} \lambda^{-2} \Lambda^3. \end{equation}
\label{lem:secNash_Phi_PartC}
\item Furthermore, $\Phi$ satisfies
\[
\nm{\partial_1 \Phi}_{L^2(I_{x_2'})}^2 \le 2 \Lambda \Gamma^{3/2},
\quad
\nm{\Phi}_{L^\infty(I_{x_2'})} \le 32\kappa\pre\Lambda\Gamma^{1/2}.
\]
\label{lem:secNash_Phi_PartD}
\item Finally,
\[
\int_{I_{x_2'}} \paren*{f(x_1,x_2')-f_M}\partial_1\Phi(x_1)\, dx_1
\ge \frac{9}{25}\alpha^2\Gamma^2 |J^+|.
\]
\label{lem:secNash_Phi_PartE}
\end{enumerate}
\end{lemma}

\begin{proof}
The construction of $\Phi$ ensures (a) and (b) since $|J^+| = |J^-| = \beta \Gamma^{-1/2}$ by \ref{hyp:secNash_Lambda_PartC}. For the part (c), we have the estimate
\[
|\supp \Phi| \le \frac{1}{2}|J^-| + |x_1''-x_1'| + \frac{1}{2}|J^+| = \beta\Gamma^{-1/2} + |x_1''-x_1'|
\]
by \ref{hyp:secNash_Lambda_PartC}. Now, it must be true that $|x_1''-x_1'| \le \sigma\Gamma^{-1/2}$ as otherwise we obtain the following contradiction
\[
\nm{\wt{f}(\cdot,x_2')}_{L^2(I_{x_2'})}^2 \ge \frac{1}{100}\alpha^2\Gamma^2|x_1''-x_1'| > \frac{1}{100}\alpha^2\sigma\Gamma^{3/2} = \Lambda\Gamma^{3/2}
\]
against the definition of $\Lambda$, where we have used \ref{hyp:secNash_Lambda_PartD} for the first inequality. Since $\lambda \le \Lambda \le \lambda\pre$ and $0 < \kappa \le 1$ easily imply the bound
\[
\sigma
> \frac{\sigma}{100 \cdot 2^{16}}
= 2^{-8}\kappa^{-2} \lambda^{-2} \Lambda^3
\ge 2^{-8} \kappa^{-2} \lambda
\ge 2^{-8} \kappa^2 \lambda
\ge 2^{-8} \kappa^2 \lambda^4\Lambda^{-3}
= \beta,
\]
we altogether have that
\[
|\supp \Phi| \le (\beta+\sigma)\Gamma^{-1/2} \le 2\sigma \Gamma^{-1/2},
\]
completing the proof of (c).

To prove (d), we observe that
\[
\nm{\partial_1 \Phi}_{L^2(I_{x_2'})}^2 = 2 \int_{J^+} |\wt{f}(x_1,x_2')|^2\, dx_1 \le 2\Lambda \Gamma^{3/2},
\]
by the definition of $\Lambda$, and also that
\[
\nm{\Phi}_{L^\infty(I_{x_2'})} \le 2\nm{\wt{f}(\cdot,x_2')}_{L^1(I_{x_2'})} \le 32\kappa\pre\Lambda\Gamma^{1/2},
\]
by \ref{hyp:secNash_Lambda_PartA}.

Lastly, to prove part (e), we first have by \ref{hyp:secNash_Lambda_PartB} that
\begin{align*}
\int_{I_{x_2'}} \wt{f}(x_1,x_2') \partial_1\Phi(x_1)\, dx_1
&= \brac*{\int_{J^+} + \int_{J^-}} \paren*{\wt{f}(x_1,x_2') \partial_1\Phi(x_1)}\, dx_1 \\
&\ge \brac*{\frac{9}{10}\alpha\Gamma-\frac{1}{2}\alpha \Gamma}\int_{J^+} \partial_1 \Phi(x_1)\, dx_1 \\
&\ge
\brac*{\frac{9}{10}\alpha\Gamma-\frac{1}{2}\alpha \Gamma}\frac{9}{10}\alpha \Gamma |J^+| \\
&=
\frac{9}{25}\alpha^2\Gamma^2 |J^+|.
\end{align*}
The desired claim follows since
\[
\int_{I_{x_2'}} \paren*{f(x_1,x_2')-f_M}\partial_1\Phi(x_1)\, dx_1
=
\int_{I_{x_2'}} \wt{f}(x_1,x_2') \partial_1\Phi(x_1)\, dx_1,
\]
because $\partial_1 \Phi$ is mean-zero over $I_{x_2'}$ and $f(x_1,x_2')-f_M = \ov{f}(x_2')-f_M + \wt{f}(x_1,x_2')$.
\end{proof}

So far, $\Phi$ is defined only on the horizontal cross-section $I_{x_2'}$, where $x_2' \in (0, h/2]$ by \Cref{hyp:secNash_Lambda}. To construct $\Psi$, we simply translate $\Phi$ to other cross sections $I_{x_2}$ for an appropriate range of heights $x_2 \ge x_2'$.

\begin{definition}
Under \Cref{hyp:secNash,hyp:secNash_Lambda}, we further fix a smooth cut-off function $\chi \in C_c^\infty(\R; [0,1])$ with $\chi(x_2)$ equal to $1$ for $x_2 \in [-1/2, 1/2]$ and equal to $0$ outside of $(-1,1)$. Also choose $\chi(x_2)$ so that $|\partial_2\chi| \le 4$ on all of $\R$. Let
\begin{equation}\label{aux131d}
\eta := \kappa^3 \lambda^{14}\Lambda^{-11}
\end{equation}
and $x_2'' := x_2' + \eta \Gamma^{-1/2}.$ Define
\[
\Psi(x_1,x_2) := \Phi(x_1)\chi\paren*{\frac{x_2-x_2''}{\eta \Gamma^{-1/2}}}.
\]
\end{definition}

\begin{lemma}\label{lem:secNash_eta}
Under \Cref{hyp:secNash,hyp:secNash_Lambda}, suppose further that
\begin{equation}\label{aux131e}
\lambda^2 < 2^{-9}M_*\pre, \quad \lambda^3 \Gamma^{-1/2} \le \eps_*/4.
\end{equation}
Then the following hold true.
\begin{enumerate}[label={(\alph*)},ref={Lemma \thelemma\alph*}]
\item The interval of heights
\[
[x_2', x_2' + 2\eta\Gamma^{-1/2}] = [x_2'' - \eta\Gamma^{-1/2}, x_2'' + \eta\Gamma^{-1/2}]
\]
is contained in $(0, h)$.
\label{lem:secNash_eta_PartA}
\item The support $\supp \Phi$, defined as a subset of $I_{x_2'}$, is also a subset of $I_{x_2}$ for any $x_2' \le x_2 \le x_2' + 2\eta\Gamma^{-1/2}$.
\label{lem:secNash_eta_PartB}
\item The support $\supp \Psi$ is contained in $\Omega$ and $\Psi \in H_0^1(\Omega)$.
\label{lem:secNash_eta_PartC}
\item We have $|\supp \Psi| \le 4\sigma\eta\Gamma\pre$.
\label{lem:secNash_eta_PartD}
\end{enumerate}
\end{lemma}
Note that the assumptions \eqref{aux131e} are exactly the ones appearing in the Proposition \ref{prop:secNash_BetweenlambdaAndlambdaInv}.

\begin{proof}
We note that $\eta := \kappa^3\lambda^{14}\Lambda^{-11} \le \lambda^3$ since $0 < \kappa \le 1$ and $\lambda \le \Lambda$. We also recall that $x_2' \in (0, h/2]$ by \Cref{hyp:secNash_Lambda}.

Consider first the case where $x_2' \in (0, \eps_*/4]$. Then, the hypotheses $\lambda^3 \Gamma^{-1/2} \le \eps_*/4$ implies
\[
x_2' + 2\eta\Gamma^{-1/2} \le x_2' + \eps_*/2 < \eps_*,
\]
proving part (a). Moreover, by \Cref{as:secSett_3}, it follows that $I_{x_2'} \subseteq I_{x_2}$ for all $x_2' \le x_2 \le x_2'+2\eta\Gamma^{-1/2}$. Thus, parts (b) and (c) both follow immediately from \ref{lem:secNash_Phi_PartB} and part (a) of this lemma, in the case where $x_2' \in (0, \eps_*/4]$.

We next consider the case where $x_2' \in (\eps_*/4, h/2]$. We similarly have that
\[
x_2' + 2\eta\Gamma^{-1/2} \le x_2' + \eps_*/2 \le h/2+\eps_*/2 \le h-\frac{1}{4}\eps_*,
\]
by \Cref{norm:secSett_eps}, which proves part (a) in this case. Moreover, fixing $x_2$ such that $x_2' \le x_2 \le x_2'+2\eta\Gamma^{-1/2}$, we have by the definition \eqref{aux128a} of $M_*$ that
\[
|F_\ell(x_2')-F_\ell(x_2)|\Gamma^{1/2} \le M_*|x_2'-x_2|\Gamma^{1/2} \le 2M_*\eta = 2M_*\kappa^3\lambda^{14}\Lambda^{-11}.
\]
Since $0 < \kappa \le 1$ and $\lambda^2 < 2^{-9}M_*\pre$, it follows that
\[
|F_\ell(x_2')-F_\ell(x_2)|\Gamma^{1/2} \le 2M_*\kappa^2\lambda^{14}\Lambda^{-11} < 2^{-8}\kappa^2\lambda^{12}\Lambda^{-11}.
\]
Finally, as $\lambda \le \Lambda$, we obtain
\[
|F_\ell(x_2')-F_\ell(x_2)|\Gamma^{1/2} < 2^{-8}\kappa^2\lambda^{4}\Lambda^{-3} = \beta.
\]
Completely analogous reasoning gives us that
\[
|F_r(x_2')-F_r(x_2)|\Gamma^{1/2} < \beta,
\]
for all $x_2' \le x_2 \le x_2'+2\eta\Gamma^{-1/2}$. By \ref{hyp:secNash_Lambda_PartC}, part (a) of this lemma, and \ref{lem:secNash_Phi_PartB} all together complete the proof of parts (b) and (c).

Lastly,part (d) holds by \ref{lem:secNash_Phi_PartC} and the construction of $\Psi$ since
\[
|\supp \Psi| \le |\supp \Phi|(2\eta\Gamma^{-1/2}) =  4\sigma\eta\Gamma\pre.
\]
\end{proof}

We note here that the exact choice for the constant $\eta$ is not only motivated by the previous lemma, but also by the proof of \Cref{prop:secNash_BetweenlambdaAndlambdaInvStep8} which concludes \Cref{subsubsec:NashUnstratHardCasesPart1} below. Now that we have constructed the test function $\Psi$, we next establish the desired upper bound on $\nm{\Psi}_{\dH_0^1}$.

\begin{lemma}\label{lem:secNash_BetweenDeltaAndDeltaInvStep7}
Under \Cref{hyp:secNash,hyp:secNash_Lambda}, suppose further that
\[
\lambda^2 < 2^{-9}M_*\pre, \quad \lambda^3 \Gamma^{-1/2} \le \eps_*/4.
\]
Then we have
\[
\nm{\Psi}_{\dH_0^1} \le 10 \cdot 2^{13} \kappa^{-7/2}\lambda^{-16}\Gamma^{1/2}.
\]
\end{lemma}

\begin{proof}
First, we observe that by construction $\partial_1 \Psi(x_1,x_2) = \partial_1 \Phi(x_1) \chi\paren*{\frac{x_2-x_2''}{\eta \Gamma^{-1/2}}}$. It immediately follows by \Cref{lem:secNash_eta} that
\begin{align*}
\nm{\partial_1 \Psi}_2^2
&=
\nm{\partial_1 \Phi}_{L^2(I_{x_2'})}^2 \int_{x_2'' - \eta\Gamma^{-1/2}}^{x_2'' + \eta\Gamma^{-1/2}} \abs*{\chi\paren*{\frac{x_2-x_2''}{\eta \Gamma^{-1/2}}}}^2\, dx_2 \\
&=
\eta \Gamma^{-1/2} \nm{\partial_1 \Phi}_{L^2(I_{x_2'})}^2 \int_{-1}^1 \abs{\chi(x_2)}^2\, dx_2.
\end{align*}
Using that $\chi$ is cut-off function bounded above by $1$ along with \ref{lem:secNash_Phi_PartD} and \eqref{aux131d}, we have that
\[
\nm{\partial_1 \Psi}_2^2 \le 2\eta \Gamma^{-1/2} \nm{\partial_1 \Phi}_{L^2(I_{x_2'})}^2 \le 4\eta \Lambda \Gamma = 4\kappa^3\lambda^{14}\Lambda^{-10}\Gamma \le 4\kappa^3\lambda^4\Gamma,
\]
where we've also used $\lambda \le \Lambda$.

For the vertical partial derivative, we have $\partial_2 \Psi(x_1,x_2) = \eta\pre \Gamma^{1/2}\Phi(x_1) \partial_2\chi\paren*{\frac{x_2-x_2''}{\eta \Gamma^{-1/2}}}$. Using the bound $|\partial_2\chi(x_2)| \le 4$ gives us that
\begin{align*}
\nm{\partial_2 \Psi}_2^2
&=
\eta^{-2}\Gamma \int_\Omega \abs*{\Phi(x_1)\partial_2\chi\paren*{\frac{x_2-x_2''}{\eta \Gamma^{-1/2}}}}^2\, dx \\
&\le
\eta^{-2}\Gamma \nm{\Phi}_{L^\infty(I_{x_2'})}^2 \int_{\supp \Psi}\abs*{\partial_2\chi\paren*{\frac{x_2-x_2''}{\eta \Gamma^{-1/2}}}}^2\, dx \\
&\le
2^4 \eta^{-2}\Gamma \nm{\Phi}_{L^\infty(I_{x_2'})}^2|\supp\Psi|.
\end{align*}
Applying \ref{lem:secNash_Phi_PartD} and \ref{lem:secNash_eta_PartD} and recalling definitions \eqref{aux131d} and \eqref{aux131c} of $\eta$ and $\sigma$ in turn gives
\[
\nm{\partial_2 \Psi}_2^2
\le
2^{16} \sigma \eta^{-1} \kappa^{-2}\Lambda^2\Gamma
=
100\cdot 2^{24} \kappa^{-7} \lambda^{-16}\Lambda^{16} \Gamma
\le
100 \cdot 2^{24} \kappa^{-7} \lambda^{-32}\Gamma,
\]
where we've also used $\lambda \le \Lambda^{-1}$.

Altogether, we estimate
\[
\nm{\Psi}_{\dH_0^1} = \paren{\nm{\partial_1 \Psi}_2^2+\nm{\partial_2 \Psi}_2^2}^{1/2} \le \paren{4\kappa^3\lambda^4 \Gamma+100 \cdot 2^{24} \kappa^{-7} \lambda^{-32}\Gamma}^{1/2} \le 10 \cdot 2^{13} \kappa^{-7/2}\lambda^{-16}\Gamma^{1/2},
\]
where we've used $0 < \lambda, \kappa \le 1$.
\end{proof}

To conclude the proof of Proposition \ref{prop:secNash_BetweenlambdaAndlambdaInv} under \Cref{hyp:secNash_Lambda}, we show that either a useful lower bound on $|\gen{\wt{f},\partial_1 \Psi}|$ holds or an inequality of the form $\kappa^{5/2} \lambda^{-1} \Gamma \lesssim \nm{\partial_2 f}_2$ must hold true.

\begin{proposition}\label{prop:secNash_BetweenlambdaAndlambdaInvStep8}
Under \Cref{hyp:secNash,hyp:secNash_Lambda}, suppose further that
\[
\lambda^2 < 2^{-9}M_*\pre, \quad \lambda^3 \Gamma^{-1/2} \le \eps_*/4.
\]
Then, at least one of the following holds true:
\[
\kappa^{21/2}\lambda^{52}\Gamma^{1/2} \le C_{10}\nm{\partial_1 f}_{\dH_0^{-1}}
\quad \text{or} \quad
\kappa^{5/2} \Gamma\lambda\pre \le c_0\nm{\partial_2 f}_2,
\]
where $C_{10}, c_0 > 0$ are universal constants (both independent even from $\Omega$).
\end{proposition}

\begin{proof}
As explained at the beginning of \Cref{subsubsec:NashUnstratHardCasesPart1}, our goal is to show $\kappa^{21/2} \Gamma^{1/2} \lambda^{52} \lesssim \nm{\partial_1 f}_{\dH_0\pre}$ by leveraging the inequality (\ref{eq:secNash_BoundByPsi}) that comes from the variational formulation of the $\dH_0\pre$ norm. To do so, we need to provide an upper bound for $\nm{\Psi}_{\dH_0^1}$ and a lower bound for $|\gen{\wt{f},\partial_1 \Psi}|$. We have already bounded $\nm{\Psi}_{\dH_0^1}$ in \Cref{lem:secNash_BetweenDeltaAndDeltaInvStep7}. So, using \Cref{lem:secNash_eta}, let us consider
\begin{align*}
|\gen{\wt{f},\partial_1 \Psi}|
&\ge
\int_\Omega \wt{f}(x_1,x_2) \partial_1\Psi(x_1,x_2)\, dx \\
&=
\int_{x_2''-\eta\Gamma^{-1/2}}^{x_2''+\eta\Gamma^{-1/2}} \chi\paren*{\frac{x_2-x_2''}{\eta \Gamma^{-1/2}}}\, \paren*{ \int_{\supp \Phi} \wt{f}(x_1,x_2)\partial_1 \Phi(x_1)\, dx_1 }\, dx_2.
\end{align*}
To continue the lower bound, we make the further assumption that
\begin{equation}\label{eq:secNash_Step8}
\int_{\supp \Phi} \wt{f}(x_1,x_2)\partial_1 \Phi(x_1)\, dx_1 \ge \frac{3}{10}\alpha \Gamma \int_{J^+} \partial_1 \Phi(x_1)\, dx_1
\end{equation}
for all $x_2''-\eta\Gamma^{-1/2} \le x_2 \le x_2''+\eta\Gamma^{-1/2}$. We will later show that if the assumption (\ref{eq:secNash_Step8}) does not hold for all $x_2''-\eta\Gamma^{-1/2} \le x_2 \le x_2''+\eta\Gamma^{-1/2}$, then we must instead have the bound $\kappa^{5/2}\Gamma\lambda\pre \lesssim \nm{\partial_2 f}_2$. For now, we assume (\ref{eq:secNash_Step8}) and obtain
\begin{align*}
|\gen{\wt{f},\partial_1 \Psi}|
&\ge
\frac{3}{10}\alpha \Gamma \int_{J^+} \partial_1\Phi(x_1)\, dx_1 \int_{x_2''-\eta\Gamma^{-1/2}}^{x_2''+\eta\Gamma^{-1/2}} \chi\paren*{\frac{x_2-x_2''}{\eta \Gamma^{-1/2}}}\, dx_2 \\
&=
\frac{3}{10}\alpha \eta \Gamma^{1/2} \int_{J^+} \partial_1\Phi(x_1)\, dx_1 \int_{-1}^{1} \chi(x_2)\, dx_2.
\end{align*}
As $\chi$ is a nonnegative function with $\chi(x_2) = 1$ for $-1/2 \le x_2 \le 1/2$, we obtain
\[
|\gen{\wt{f},\partial_1 \Psi}|
\ge
\frac{3}{10}\alpha \eta \Gamma^{1/2} \int_{J^+} \partial_1\Phi(x_1)\, dx_1.
\]
Next, using the construction of $\Phi$ along with \ref{hyp:secNash_Lambda_PartB} and \ref{hyp:secNash_Lambda_PartC}, as well as definitions \eqref{aux131a}, \eqref{aux131b}, \eqref{aux131d} of $\alpha,$ $\beta$, and $\eta$, it follows that
\[
|\gen{\wt{f},\partial_1 \Psi}|
\ge
\frac{3}{10}\alpha \eta \Gamma^{1/2} \paren*{\frac{9}{10}\alpha\Gamma}|J^+|
=
\frac{3^3}{10^2}\alpha^2 \beta \eta \Gamma
=
\frac{3^3 \kappa^7 \lambda^{20}\Gamma}{10^2\cdot 2^{16}  \Lambda^{16}}.
\]
This is a viable lower bound on $|\gen{\wt{f},\partial_1 \Psi}|$ which we can combine with (\ref{eq:secNash_BoundByPsi}) and \Cref{lem:secNash_BetweenDeltaAndDeltaInvStep7} to obtain
\[
\nm{\partial_1 \wt{f}}_{\dH_0^{-1}}
\ge
\frac{3^3 \kappa^{21/2} \lambda^{36}\Gamma^{1/2}}{10^3\cdot 2^{29}  \Lambda^{16}}
\ge
\frac{3^3}{10^3\cdot 2^{29}} \kappa^{21/2} \lambda^{52}\Gamma^{1/2},
\]
where we have also used that $\Lambda \le \lambda\pre$. Thus, we may take $C_{10} := 2^{29} 3^{-3} 10^3$, completing the proof in the case that (\ref{eq:secNash_Step8}) holds for all $x_2''-\eta\Gamma^{-1/2} \le x_2 \le x_2''+\eta\Gamma^{-1/2}$.

Now, instead suppose there exists some $z \in [0,h]$ such that $x_2' \le z \le x_2'+2\eta\Gamma^{-1/2}$ and
\[
\int_{\supp \Phi} \paren*{f(x_1,z)-f_M}\partial_1 \Phi(x_1)\, dx_1 = \int_{\supp \Phi} \wt{f}(x_1,z)\partial_1 \Phi(x_1)\, dx_1 < \frac{3}{10}\alpha \Gamma \int_{J^+} \partial_1 \Phi(x_1)\, dx_1,
\]
where the first equality holds since $f(x_1,z) - f_M = \ov{f}(z)-f_M+\wt{f}(x_1,z)$ and as $\partial_1 \Phi(x_1)$ is mean-zero over $\supp \Phi$, see  \ref{lem:secNash_Phi_PartA}. Then, \Cref{lem:secNash_eta} allows us to integrate over $\Omega$ as follows:
\begin{align*}
&\int_{\supp \Phi} \int_{x_2'}^{z} \abs*{\partial_2 [f(x_1, x_2)-f_M] \partial_1 \Phi(x_1) }\, dx_2\, dx_1 \\
&\quad\ge
\int_{\supp \Phi} \paren{f(x_1, x_2')-f_M}\partial_1 \Phi(x_1)\, dx_1 -  \int_{\supp \Phi} \paren{f(x_1, z)-f_M}\partial_1 \Phi(x_1)\, dx_1 \\
&\quad>
\int_{\supp \Phi} \paren{f(x_1, x_2')-f_M}\partial_1 \Phi(x_1)\, dx_1 -  \frac{3}{10}\alpha \Gamma \int_{J^+} \partial_1 \Phi(x_1)\, dx_1 \\
&\quad\ge \int_{\supp \Phi} \paren{f(x_1, x_2')-f_M}\partial_1 \Phi(x_1)\, dx_1 - \frac{33}{100}\alpha^2 \Gamma^2|J^+|,
\end{align*}
where the final line holds by \ref{hyp:secNash_Lambda_PartB}. It follows by \ref{lem:secNash_Phi_PartE} and \ref{hyp:secNash_Lambda_PartC} that
\[
\int_{\supp \Phi} \int_{x_2'}^{z} \abs*{\partial_2[f(x_1, x_2)-f_M] \partial_1 \Phi(x_1) }\, dx_2\, dx_1
>
\frac{3}{100}\alpha^2\Gamma^2 |J^+|
=
\frac{3}{100}\alpha^2 \beta\Gamma^{3/2}.
\]
Since $z-x_2' \le 2\eta\Gamma^{-1/2}$, we also compute
\[
\int_{\supp \Phi} \int_{x_2'}^{z} |\partial_1 \Phi(x_1)|^2\, dx_2\, dx_1
\le
2\eta \Gamma^{-1/2} \int_{I_{x_2'}} |\partial_1 \Phi(x_1)|^2\, dx_1
\le
4\eta \Lambda \Gamma
\]
by \ref{lem:secNash_Phi_PartD}. Collecting our work, we obtain
\begin{align*}
\nm{\partial_2 f}_2^2 &\ge \int_{\supp \Phi} \int_{x_2'}^{z} |\partial_2 f(x)|^2\, dx_2\, dx_1 \\
&\ge \paren*{\int_{\supp \Phi} \int_{x_2'}^{z} \abs*{\partial_2 [f(x)-f_M] \partial_1 \Phi(x_1) }\, dx_2\, dx_1}^2 \paren*{\int_{\supp \Phi} \int_{x_2'}^{z} |\partial_1 \Phi(x_1)|^2\, dx_2\, dx_1}\pre \\
&\ge \paren*{\frac{3}{100}\alpha^2\beta\Gamma^{3/2}}^2 \paren*{4\eta \Lambda \Gamma}\pre \\
&= \frac{3^2}{2^{34} \cdot 10^4} \kappa^5 \lambda^{-2}\Gamma^2.
\end{align*}
Setting $c_0 := 2^{17}\cdot 10^2 \cdot 3\pre$ gives the desired inequality $\kappa^{5/2}\Gamma\lambda\pre \le c_0\nm{\partial_2 f}_2$.
\end{proof}

We note that $\eta$ was specifically chosen to ensure that the inequality $\kappa^{5/2}\Gamma\lambda\pre \lesssim \nm{\partial_2 f}_2$ is true whenever (\ref{eq:secNash_Step8}) does not hold for all $x_2''-\eta\Gamma^{-1/2} \le x_2 \le x_2''+\eta\Gamma^{-1/2}$.

\subsubsection{Completing the proof for the case of intermediate \texorpdfstring{$\Lambda$}{Lambda}}
\label{subsubsec:NashUnstratHardCasesPart2}

Continuing to work under \Cref{hyp:secNash}, it remains to show that one can reduce the proof of \Cref{prop:secNash_BetweenlambdaAndlambdaInv} to the case of \Cref{hyp:secNash_Lambda} holding true. We do so in this section by only assuming \Cref{hyp:secNash} and proceeding in several steps. In each step, we consider a particular assumption of \Cref{hyp:secNash_Lambda} and show that if there is not a positive measure subset of heights satisfying that assumption, then a Nash-type inequality of the form $\kappa^{5/2} \Gamma \lambda\pre \lesssim \nm{\partial_1 f}_2$ must hold true. In the end, we collect these lemmas together along with the previously proven \Cref{prop:secNash_BetweenlambdaAndlambdaInvStep8} in order to present the full proof of \Cref{prop:secNash_BetweenlambdaAndlambdaInv}. We note that, between the lemmas to follow below, we discuss definitions/constructions of many constants and sets that will be useful across \Cref{subsubsec:NashUnstratHardCasesPart2}.

Our first goal is to show a quantitative lower bound on the measure of heights $x_2$ satisfying \ref{hyp:secNash_Lambda_PartA}, lest the inequality $\kappa^{5/2} \Gamma \lambda\pre \lesssim \nm{\partial_1 f}_2$ hold true. To do so, we define two subsets of heights, $S_1$ and $S_2$, which give us control on the cross-sectional $L^1$ and $L^2$ norms of the unstratified component $\wt{f}$.

\begin{definition}
Under \Cref{hyp:secNash} and for fixed $0 < \kappa \le 1$ and $0 < \lambda < 1$, we define the corresponding sets
\begin{align*}
&S_1 := \curly*{
x_2 \in [0,h]\ \big|\
\nm{\wt{f}(\cdot,x_2)}_{L^2(I_{x_2})}^2 \ge \lambda \Gamma^{3/2}
}, \\
&S_2 := \curly*{
x_2 \in S_1\ \big|\
\nm{\wt{f}(\cdot, x_2)}_{L^1(I_{x_2})} \le 16 \kappa\pre \Lambda \Gamma^{1/2}
}.
\end{align*}
We note here that, for any $x_2 \in S_2 \subseteq S_1$, we have
\begin{equation}\label{eq:secNash_fwtLinftyLowerBound}
\nm{\wt{f}(\cdot,x_2)}_{L^\infty(I_{x_2})}
\ge
\nm{\wt{f}(\cdot,x_2)}_{L^2(I_{x_2})}^2\nm{\wt{f}(\cdot,x_2)}_{L^1(I_{x_2})}\pre
\ge
\alpha \Gamma,
\end{equation}
where we recall from \eqref{aux131a} that $\alpha := \frac{1}{16}\kappa\lambda\Lambda\pre > 0$.
\end{definition}

\begin{remark}
In the special case of
\[
\kappa = \nm{\wt{f}}_2^2\Gamma\pre = \nm{\wt{f}}_2^2\nm{f-f_M}_2^{-2},
\]
we then have
\[
\kappa\pre \Lambda \Gamma^{1/2} = \nm{\wt{f}}_2^{-2}\sup_{x_2 \in [0, h]} \nm{\wt{f}(\cdot,x_2)}_{L^2(I_{x_2})}^2.
\]
This choice of $\kappa$ was made in the proof of \Cref{thm:secNash_NashOnUnstratified}. Nonetheless, in what follows, we work with the more general condition that $\nm{\wt{f}}_2^2 \ge \kappa \Gamma$ as that will be handy for the argument.
\end{remark}

In the following lemma, we will first show that the measure $|S_1|$ being small implies the Nash-type inequality that we desire. This leads us then to focus on the alternate case where $|S_1|$ is large enough, which we show implies that $|S_2|$ must also satisfy a quantitative lower bound.

\begin{lemma}\label{lem:secNash_BetweenDeltaAndDeltaInvStep1}
Under \Cref{hyp:secNash} and for fixed $0 < \kappa \le 1$ and $0 < \lambda < 1$, suppose that
\[
\nm{\wt{f}}_1 = 1, \quad \nm{\wt{f}}_2^2 \ge \kappa \Gamma.
\]
If $|S_1| \le \frac{1}{2}\kappa\Lambda\pre\Gamma^{-1/2}$, then it must hold that
\[
\kappa^{5/2} \Gamma \lambda\pre \le c_1\nm{\partial_1 \wt{f}}_2 = c_1\nm{\partial_1 f}_2
\]
for some universal constant $c_1 := c_1(\Omega) > 0$. If instead $|S_1| > \frac{1}{2}\kappa\Lambda\pre\Gamma^{-1/2}$, then also $|S_2| > \frac{1}{4}\kappa\Lambda\pre\Gamma^{-1/2}$.
\end{lemma}

\begin{proof}
Suppose that $|S_1| \le \frac{1}{2}\kappa\Lambda\pre\Gamma^{-1/2}$. Then, by the definition of $\Lambda$ in (\ref{aux129a}), we have
\[
\int_{S_1} \nm{\wt{f}(\cdot,x_2)}_{L^2(I_{x_2})}^2\, dx_2 \le \Lambda\Gamma^{3/2}|S_1| \le \frac{1}{2}\kappa\Gamma.
\]
By the definition of $S_1$, we may in turn apply \Cref{prop:secNash_Belowlambda} by taking $S := [0,h] \setminus S_1$, giving us that
\[
\kappa^{5/2} \Gamma \lambda\pre \le 2^{5/2} C_7 \nm{\partial_1 \wt{f}}_2.
\]
This establishes the desired claim with $c_1 = c_1(\Omega) := 2^{5/2} C_7 > 0$.

If instead $|S_1| > \frac{1}{2}\kappa\Lambda\pre\Gamma^{-1/2}$ while $|S_2| \le \frac{1}{4}\kappa\Lambda\pre\Gamma^{-1/2}$, then we would obtain the following contradiction:
\[
1 = \nm{\wt{f}}_1 \ge \int_{S_1\setminus S_2} \int_{I_{x_2}} |\wt{f}(x_1,x_2)|\, dx_1\, dx_2 > 16\kappa\pre \Lambda \Gamma^{1/2}|S_1\setminus S_2| \ge 4.
\]
This completes the proof.
\end{proof}

Thus, \Cref{lem:secNash_BetweenDeltaAndDeltaInvStep1} shows us that \ref{hyp:secNash_Lambda_PartA} holds for a positive measure subset of heights $x_2$, namely the set $S_2$, lest the inequality $\kappa^{5/2} \Gamma \lambda\pre \lesssim \nm{\partial_1 f}_2$ hold true. So, we now on focus on the case where $|S_2|$ is large. The next few lemmas show that there must exist a positive subset of heights $x_2$ with at least one pair of corresponding subintervals $J^+,J^- \subseteq I_{x_2}$ satisfying the assumptions in \ref{hyp:secNash_Lambda_PartB} and \ref{hyp:secNash_Lambda_PartC}, lest the Nash-type inequality also hold true. To that end, we make the following definition.

\begin{definition}\label{def:secNash_J}
Under \Cref{hyp:secNash} and for fixed $0 < \kappa \le 1$ and $0 < \lambda < 1$, for each $x_2 \in S_2$ we fix a choice of $x_1' \in I_{x_2}$ which satisfies
\[
|\wt{f}(x_1',x_2)| = \alpha \Gamma.
\]
This is possible by (\ref{eq:secNash_fwtLinftyLowerBound}). In turn, let $J \subseteq I_{x_2}$ denote the largest (necessarily closed) connected subinterval of $I_{x_2}$ that contains (the associated fixed choice of) $x_1'$ and such that
\[
|\wt{f}(x_1,x_2)| \ge \frac{9}{10}\alpha \Gamma
\]
for all $x_1 \in J$. This guarantees $\wt{f}(\cdot,x_2)$ is signed, as a function, on $J$. Moreover,
\[
\abs{\wt{f}(x_1^\ell, x_2)} = \abs{\wt{f}(x_1^r, x_2)} = \frac{9}{10}\alpha \Gamma
\]
where $x_1^\ell$ and $x_1^r$ denote the left-hand and right-hand endpoints of $J$, respectively. We note that, for each height $x_2 \in S_2$, the choice of $x_1'$ need not be unique, but this will not affect the arguments to follow. Also, we do not explicitly express the dependence of $x_1'$ and $J$ on $x_2$ in order to keep notation simple.
\end{definition}

We wish to work with a height $x_2 \in S_2$ for which the corresponding interval $J$ is well behaved, meaning that the interval $J$ has the properties required by \ref{hyp:secNash_Lambda_PartB} and \ref{hyp:secNash_Lambda_PartC}. The first step in this direction is the construction of the following subset $S_3$ of $S_2$.

\begin{definition}
Under \Cref{hyp:secNash} and for fixed $0 < \kappa \le 1$ and $0 < \lambda < 1$, we define
\[
S_3 := \curly*{
x_2 \in S_2\ \bigg|\ |x_1'-x_1^r| > \frac{3}{2}\beta \Gamma^{-1/2} \text{ and } |x_1'-x_1^\ell| > \frac{3}{2}\beta \Gamma^{-1/2}
},
\]
where as before (see \eqref{aux131b})
\[
\beta = \lambda^2\alpha^2 \Lambda\pre = 2^{-8}\kappa^2\lambda^4\Lambda^{-3} > 0.
\]
For any choice of $x_2 \in S_3$ with corresponding $x_1'$ and $J$, we can define the closed subinterval
\begin{equation}\label{aux131f}
J^+ := \brac*{x_1'-\frac{1}{2}\beta\Gamma^{-1/2},\ x_1'+\frac{1}{2}\beta\Gamma^{-1/2}} \subseteq J
\end{equation}
centered at $x_1'$ with length exactly $\beta\Gamma^{-1/2}$. Given the definition of $S_3$, we see that $J^+$ must be at least a distance of $\beta\Gamma^{-1/2}$ away from $\partial I_{x_2}$.
\end{definition}

We now show that if the set $S_3$ does not have large enough positive measure, then the desired Nash-type inequality must hold. The proof to follow also reveals the reasoning behind the definition of the constant $\beta$.

\begin{lemma}\label{lem:secNash_BetweenDeltaAndDeltaInvStep2}
Under \Cref{hyp:secNash} and for fixed $0 < \kappa \le 1$ and $0 < \lambda < 1$, suppose that
\[
\nm{\wt{f}}_1 = 1,
\quad
\nm{\wt{f}}_2^2 \ge \kappa \Gamma,
\quad
|S_2| > \frac{1}{4}\kappa\Lambda\pre\Gamma^{-1/2}.
\]
If $|S_3| \le \frac{1}{8}\kappa\Lambda\pre\Gamma^{-1/2}$, then it must be true that
\[
\kappa^{5/2} \Gamma \lambda\pre \le \kappa^{1/2} \Gamma \lambda\pre \le c_2\nm{\partial_1 \wt{f}}_2 = c_2\nm{\partial_1 f}_2
\]
for some universal constant $c_2 > 0$ (independent even from $\Omega$).
\end{lemma}

\begin{proof}
For now, fix $x_2 \in S_2$ and the corresponding $x_1'$ with interval $J$ that has endpoints $x_1^\ell$, $x_1^r$. Then
\[
\int_{x_1^\ell}^{x_1'} \abs{\partial_1 \wt{f}(x_1,x_2)}\, dx_1
\ge
\abs*{\int_{x_1^\ell}^{x_1'} \partial_1 \wt{f}(x_1,x_2)\, dx_1 }
=
\abs*{\wt{f}(x_1', x_2) - \wt{f}(x_1^\ell, x_2)}
=
\frac{1}{10}\alpha \Gamma.
\]
In turn, we observe that
\begin{align*}
\int_{I_{x_2}} \abs{\partial_1 \wt{f}(x_1,x_2)}^2\, dx_1
&\ge \int_{x_1^\ell}^{x_1'} \abs{\partial_1 \wt{f}(x_1,x_2)}^2\, dx_1 \\
&\ge \frac{1}{x_1' - x_1^\ell} \paren*{\int_{x_1^\ell}^{x_1'} \abs{\partial_1 \wt{f}(x_1,x_2)} dx_1}^2 \\
&\ge \frac{\alpha^2\Gamma^2}{100(x_1' - x_1^\ell)}.
\end{align*}
It analogously also holds that
\[
\int_{I_{x_2}} \abs{\partial_1 \wt{f}(x_1,x_2)}^2\, dx_1 \ge \frac{\alpha^2\Gamma^2}{100(x_1^r - x_1')}.
\]
The above discussion applies to any $x_2 \in S_2$. If we further suppose $x_2 \in S_2$ is such that $x_2 \notin S_3$, then the corresponding interval $J \subseteq I_{x_2}$ is such that $x_1' - x_1^\ell \le \frac{3}{2}\beta \Gamma^{-1/2}$ or $x_1^r - x_1' \le \frac{3}{2}\beta \Gamma^{-1/2}$. Supposing that $x_1' - x_1^\ell \le \frac{3}{2}\beta \Gamma^{-1/2}$, it would hold
\[
\int_{I_{x_2}} \abs{\partial_1 \wt{f}(x_1,x_2)}^2\, dx_1
\ge
\frac{\alpha^2\Gamma^2}{100(x_1' - x_1^\ell)}
\ge
\frac{\Lambda \Gamma^{5/2}}{150 \lambda^2}.
\]
The same inequality holds if instead $x_1^r -x_1' \le \frac{3}{2}\beta \Gamma^{-1/2}$. This lower bound is independent of the choice of $x_2 \in S_2 \setminus S_3$. Thus, supposing $|S_3| \le \frac{1}{8}\kappa\Lambda\pre\Gamma^{-1/2}$, we compute
\[
\int_\Omega |\partial_1 \wt{f}(x)|^2\, dx
\ge
\int_{S_2\setminus S_3} \int_{I_{x_2}} |\partial_1 \wt{f}(x)|^2\, dx_1\, dx_2
\ge
\frac{\Lambda \Gamma^{5/2}}{150 \lambda^2}|S_2\setminus S_3|
\ge
\frac{\kappa}{8 \cdot 150 \lambda^2} \Gamma^2.
\]
This establishes the claim with $c_2 = 20\sqrt{3}$.
\end{proof}

Therefore, in the case of a set $S_2$ of large measure, \Cref{lem:secNash_BetweenDeltaAndDeltaInvStep2} implies that the desired Nash-type inequality holds or there is a significant measure of heights $x_2$ (namely those lying in $S_3$) with the associated interval $J$ having a well-behaved subinterval $J^+$. We continue to refine the set of heights in our construction by working only with those where $\wt{f}$ is bounded by $\frac{11}{10}\alpha\Gamma$ on each $J^+$, as this is needed for \ref{hyp:secNash_Lambda_PartB}.

\begin{definition}
Under \Cref{hyp:secNash} and for fixed $0 < \kappa \le 1$ and $0 < \lambda < 1$, we define the set
\[
S_4 := \curly*{
x_2 \in S_3 \ \bigg |\
|\wt{f}(x_1,x_2)| \le \frac{11}{10}\alpha \Gamma \text{ for all } x_1 \in J^+
}.
\]
\end{definition}

\begin{lemma}\label{lem:secNash_BetweenDeltaAndDeltaInvStep3}
Under \Cref{hyp:secNash} and for fixed $0 < \kappa \le 1$ and $0 < \lambda < 1$,  suppose that
\[
\nm{\wt{f}}_1 = 1,
\quad
\nm{\wt{f}}_2^2 \ge \kappa \Gamma,
\quad
|S_3| > \frac{1}{8}\kappa\Lambda\pre\Gamma^{-1/2}.
\]
If $|S_4| \le \frac{1}{16}\kappa\Lambda\pre\Gamma^{-1/2}$, then we have
\[
\kappa^{5/2} \Gamma \lambda\pre \le \kappa^{1/2} \Gamma \lambda\pre \le c_3\nm{\partial_1 \wt{f}}_2 = c_3\nm{\partial_1 f}_2
\]
for some universal constant $c_3 > 0$ (independent even from $\Omega$).
\end{lemma}

\begin{proof}
For each $x_2 \in S_3 \setminus S_4$, there is some $z \in J^+$ such that $|\wt{f}(z,x_2)| > \frac{11}{10}\alpha \Gamma$. Supposing for now that $z > x_1'$, we in turn have the lower bound
\[
\int_{x_1'}^z |\partial_1 \wt{f}(x_1,x_2)|\, dx_1
\ge
\abs*{\int_{x_1'}^z \partial_1 \wt{f}(x_1,x_2)\, dx_1}
=
\abs*{\wt{f}(z, x_2) - \wt{f}(x_1', x_2)}
>
\frac{1}{10}\alpha\Gamma,
\]
from which it follows
\[
\int_{I_{x_2}} \abs{\partial_1 \wt{f}(x_1,x_2)}^2\, dx_1
\ge
\int_{J^+} \abs{\partial_1 \wt{f}(x_1,x_2)}^2\, dx_1 \\
\ge
\frac{1}{|J^+|} \paren*{\int_{x_1'}^z \abs{\partial_1 \wt{f}(x_1,x_2)} dx_1}^2 \\
\ge
\frac{\Lambda\Gamma^{5/2}}{100\lambda^2},
\]
where we have used $|J^+| = \beta\Gamma^{-1/2}$. The same inequality holds if instead $z < x_1'$. Also, this inequality is independent of the choice $x_2 \in S_3 \setminus S_4$. Thus, if we suppose that $|S_4| \le \frac{1}{16}\kappa\Lambda\pre\Gamma^{-1/2}$, we compute
\[
\int_\Omega |\partial_1 \wt{f}|^2\, dx
\ge
\int_{|S_3\setminus S_4|} \int_{I_{x_2}} |\partial_1 \wt{f}|^2\, dx_1\, dx_2
\ge
\frac{\Lambda\Gamma^{5/2}}{100\lambda^2} |S_3\setminus S_4|
\ge
\frac{\kappa}{1600\lambda^2} \Gamma^2.
\]
This establishes the claim with $c_3 = 40$.
\end{proof}

With \Cref{lem:secNash_BetweenDeltaAndDeltaInvStep3}, we now know that a significant positive measure of heights $x_2$ have an associated subinterval $J^+$ satisfying the assumptions of \ref{hyp:secNash_Lambda_PartB} and \ref{hyp:secNash_Lambda_PartC}, lest the inequality $\kappa^5 \Gamma \lambda\pre \lesssim \nm{\partial_1 f}_2$ hold true. We next work to establish the analogous lemmas for subintervals of the form $J^-$ as described in \ref{hyp:secNash_Lambda_PartB}, \ref{hyp:secNash_Lambda_PartC}, and \ref{hyp:secNash_Lambda_PartD}. To do so, we consider how the value of $\wt{f}(\cdot,x_2)$ drops along the horizontal cross section $I_{x_2}$ when $x_2 \in S_4$.

\begin{definition}\label{def:secNash_B}
Under \Cref{hyp:secNash} and for fixed $0 < \kappa \le 1$ and $0 < \lambda < 1$, let us consider a fixed $x_2 \in S_4$ and the corresponding interval $J^+$ with center $x_1'$. Recall that $|\wt{f}(x_1, x_2)| \ge \frac{9}{10}\alpha \Gamma$ for $x_1 \in J^+$ and, in particular, $\wt{f}(\cdot,x_2)$ is a signed function on $J^+$. It follows that
\[
\abs*{\int_{J^+} \wt{f}(x_1,x_2)\, dx_1} \ge \frac{9}{10}\alpha\Gamma|J^+| = \frac{9}{10}\alpha\beta \Gamma^{1/2}.
\]
Since $\wt{f}$ is mean-zero over $I_{x_2}$, the pigeon-hole principle gives us a subinterval $B \subseteq I_{x_2}$ such that
\begin{equation}\label{eq:secNash_ToHelpChoosex1''}
\sgn\paren*{\wt{f}(x_1',x_2)} \int_B \wt{f}(x_1,x_2)\, dx_1 \le -\frac{9}{20}\alpha\beta\Gamma^{1/2},
\end{equation}
where either
\[
B = \brac*{F_\ell(x_2),\ x_1'-\frac{1}{2}\beta\Gamma^{-1/2}}
\quad \text{or} \quad
B = \brac*{x_1'+\frac{1}{2}\beta\Gamma^{-1/2},\ F_r(x_2)}.
\]
Thus, there must exist some $x_1'' \in B$, of which we fix a choice for each height $x_2$, such that $|\wt{f}(x_1'', x_2)| = \frac{1}{10}\alpha \Gamma$ and $|\wt{f}(x_1, x_2)| > \frac{1}{10}\alpha \Gamma$ for each $x_1$ between $x_1'$ and $x_1''$. This construction ensures that \ref{hyp:secNash_Lambda_PartD} is satisfied and, in particular, $\wt{f}(\cdot,x_2)$ is signed on the interval between $x_1'$ and $x_1''$. Also note that $x_1'' \notin J$ by construction with either $x_1'' < x_1^\ell$ or $x_1'' > x_1^r$ holding true. See \Cref{fig:secNash_x1''construction} for a cartoon example. We do not express the dependence of $x_1''$ and $B$ on $x_2$ in order to keep notation simple.
\end{definition}

It will be nice to work with heights $x_2$ such that the corresponding intervals $J$ are sufficiently separated from the points $x_1''$, since we need to ensure that $J^\pm$ are disjoint and satisfy $|J^\pm| = \beta \Gamma^{-1/2}$. More specifically, we work with the following subset $S_5 \subseteq S_4$.

\tikzset{every picture/.style={line width=0.75pt}}
\begin{figure}[!ht]%
\centering
\begin{tikzpicture}[x=0.75pt,y=0.75pt,yscale=-1,xscale=1]
\draw    (20.11,169.13) -- (621.09,168.92) ;
\draw    (352.48,56.33) -- (352.48,281.51) ;
\draw    (228.61,163.97) -- (228.61,176.34) ;
\draw    (378.24,163.05) -- (378.24,175.42) ;
\draw    (430.98,163.05) -- (430.98,175.42) ;
\draw    (489.85,164.29) -- (489.85,176.66) ;
\draw    (20.11,162.95) -- (20.11,175.32) ;
\draw   (20,181.04) .. controls (19.99,185.71) and (22.31,188.05) .. (26.98,188.07) -- (259.59,188.88) .. controls (266.26,188.91) and (269.58,191.25) .. (269.57,195.92) .. controls (269.58,191.25) and (272.92,188.93) .. (279.59,188.95)(276.59,188.94) -- (371.22,189.27) .. controls (375.89,189.28) and (378.23,186.96) .. (378.24,182.29) ;
\draw   (380.59,182.28) .. controls (380.59,186.95) and (382.92,189.28) .. (387.59,189.28) -- (425.27,189.28) .. controls (431.94,189.28) and (435.27,191.61) .. (435.27,196.28) .. controls (435.27,191.61) and (438.6,189.28) .. (445.27,189.28)(442.27,189.28) -- (482.96,189.28) .. controls (487.63,189.28) and (489.96,186.95) .. (489.96,182.28) ;
\draw [color={rgb, 255:red, 189; green, 16; blue, 224 }  ,draw opacity=0.5 ][line width=2.25]    (20.11,296.57) .. controls (27.08,341.49) and (70.01,355.1) .. (94.54,353.87) .. controls (119.07,352.63) and (168.13,310.56) .. (186.52,266.02) .. controls (204.92,221.48) and (206.15,176.94) .. (228.22,154.67) .. controls (250.3,132.4) and (354.09,135.31) .. (376.09,96.31) .. controls (398.09,57.31) and (446.54,32.19) .. (496.82,91.58) .. controls (547.11,150.96) and (567.96,143.54) .. (612.11,146.01) ;
\draw [color={rgb, 255:red, 189; green, 16; blue, 224 }  ,draw opacity=1 ]   (344.61,97.63) -- (360.89,97.63) ;
\draw [color={rgb, 255:red, 189; green, 16; blue, 224 }  ,draw opacity=1 ]   (344.61,55.71) -- (360.89,55.71) ;
\draw [color={rgb, 255:red, 189; green, 16; blue, 224 }  ,draw opacity=1 ]   (344.61,146.63) -- (360.89,146.63) ;
\draw   (20,241.04) .. controls (20.01,245.71) and (22.34,248.04) .. (27.01,248.04) -- (433.84,247.68) .. controls (440.51,247.67) and (443.84,250) .. (443.85,254.67) .. controls (443.84,250) and (447.17,247.67) .. (453.84,247.67)(450.84,247.67) -- (614.09,247.53) .. controls (618.76,247.52) and (621.09,245.19) .. (621.09,240.52) ;
\draw (232.89,148.83) node [anchor=north west][inner sep=0.75pt]    {$x_{1} ''$};
\draw (381.28,146.91) node [anchor=north west][inner sep=0.75pt]    {$x_{1}^{\ell }$};
\draw (433.47,149.91) node [anchor=north west][inner sep=0.75pt]    {$x_{1} '$};
\draw (492.11,146.91) node [anchor=north west][inner sep=0.75pt]    {$x_{1}^{r}$};
\draw (23.31,146.59) node [anchor=north west][inner sep=0.75pt]    {$F_{\ell }( x_{2})$};
\draw (271.34,193.21) node [anchor=north west][inner sep=0.75pt]    {$B$};
\draw (437.71,192.28) node [anchor=north west][inner sep=0.75pt]   [align=left] {$\displaystyle J^{+}$ of length $\displaystyle \beta \Gamma ^{-1/2}$};
\draw (305,45) node [anchor=north west][inner sep=0.75pt]  [color={rgb, 255:red, 189; green, 16; blue, 224 }  ,opacity=1 ]  {$\frac{11}{10} \alpha \Gamma $};
\draw (305,87) node [anchor=north west][inner sep=0.75pt]  [color={rgb, 255:red, 189; green, 16; blue, 224 }  ,opacity=1 ]  {$\frac{9}{10} \alpha \Gamma $};
\draw (305,137) node [anchor=north west][inner sep=0.75pt]  [color={rgb, 255:red, 189; green, 16; blue, 224 }  ,opacity=1 ]  {$\frac{1}{10} \alpha \Gamma $};
\draw (446,252.94) node [anchor=north west][inner sep=0.75pt]    {$I_{x_{2}}$};
\draw (431.54,27.9) node [anchor=north west][inner sep=0.75pt]   [align=left] {\textcolor[rgb]{0.74,0.06,0.88}{Values of }\textcolor[rgb]{0.74,0.06,0.88}{$\displaystyle \widetilde{f}( \cdot ,x_{2})$ where $\displaystyle x_{2} \in S_{4}$}};
\end{tikzpicture}
\caption{An example graph of $\wt{f}(\cdot, x_2)$ over $I_{x_2}$ for a fixed choice of $x_2 \in S_4$, highlighting the construction of $x_1''$ and the interval $B$.}
\label{fig:secNash_x1''construction}
\end{figure}
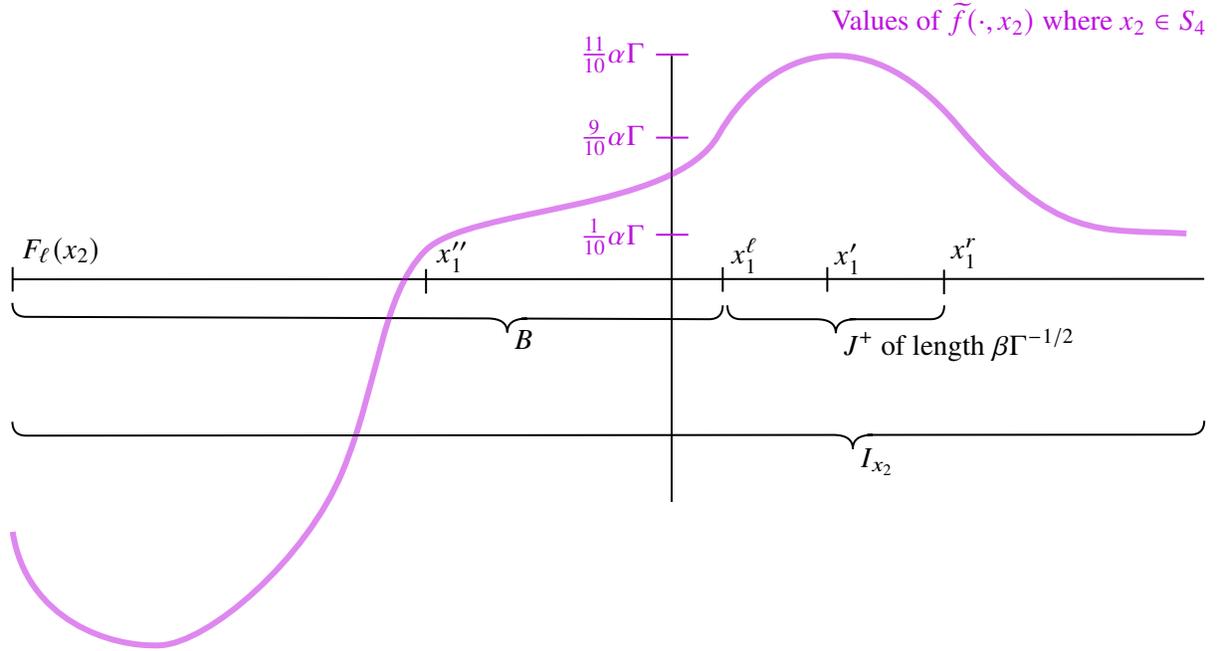

\begin{definition}
Under \Cref{hyp:secNash} and for fixed $0 < \kappa \le 1$ and $0 < \lambda < 1$, we define
\[
S_5 := \curly*{
x_2 \in S_4\ \bigg|\ |x_1'' - x_1^r| > \frac{1}{2}\beta\Gamma^{-1/2} \text{ and } |x_1'' - x_1^\ell| > \frac{1}{2}\beta\Gamma^{-1/2}
}.
\]
\end{definition}

Continuing the theme of this section, we show that either the measure $|S_5|$ is significant or we have a Nash-type inequality.

\begin{lemma}\label{lem:secNash_BetweenDeltaAndDeltaInvStep4}
Under \Cref{hyp:secNash} and for fixed $0 < \kappa \le 1$ and $0 < \lambda < 1$, suppose that
\[
\nm{\wt{f}}_1 = 1,
\quad
\nm{\wt{f}}_2^2 \ge \kappa \Gamma,
\quad
|S_4| > \frac{1}{16}\kappa\Lambda\pre\Gamma^{-1/2}.
\]
If $|S_5| \le \frac{1}{32}\kappa\Lambda\pre\Gamma^{-1/2}$, then it must hold that
\[
\kappa^{5/2} \Gamma \lambda\pre \le \kappa^{1/2} \Gamma \lambda\pre \le c_4\nm{\partial_1 \wt{f}}_2 = c_4\nm{\partial_1 f}_2
\]
for some universal constant $c_4 > 0$ (independent even from $\Omega$).
\end{lemma}

\begin{proof}
Fix any $x_2 \in S_4$ and suppose without loss of generality that $x_1'' < x_1^\ell$. We compute
\[
\int_{x_1''}^{x_1^\ell} |\partial_1 \wt{f}(x_1,x_2)|\, dx_1 \ge \abs*{\int_{x_1''}^{x_1^\ell} \partial_1 \wt{f}(x_1,x_2)\, dx_1} = \abs*{\wt{f}(x_1^\ell, x_2) - \wt{f}(x_1'', x_2)} = \frac{4}{5}\alpha \Gamma,
\]
and hence
\begin{align*}
\int_{I_{x_2}} \abs{\partial_1 \wt{f}(x_1,x_2)}^2\, dx_1
&\ge \int_{x_1''}^{x_1^\ell} \abs{\partial_1 \wt{f}(x_1,x_2)}^2\, dx_1 \\
&\ge \frac{1}{x_1^\ell - x_1''} \paren*{\int_{x_1''}^{x_1^\ell} \abs{\partial_1 \wt{f}(x_1,x_2)} dx_1}^2 \ge \frac{16\alpha^2\Gamma^2}{25(x_1^\ell - x_1'')}.
\end{align*}
If instead it holds $x_1'' > x_1^r$, then we analogously have
\[
\int_{I_{x_2}} \abs{\partial_1 \wt{f}(x_1,x_2)}^2\, dx_1 \ge \frac{16\alpha^2\Gamma^2}{25(x_1'' - x_1^r)}.
\]
We note that, if $x_1'' > x_1^\ell$ holds, then $x_1'' > x_1^r$ since $x_1'' \notin J$. Similarly, if $x_1'' < x_1^r$ holds, then $x_1'' < x_1^\ell$.

The above discussion holds for any $x_2 \in S_4$. If we further suppose $x_2 \in S_4$ is such that $x_2 \notin S_5$, then either $|x_1'' - x_1^r| \le \frac{1}{2}\beta\Gamma^{-1/2}$ or $|x_1'' - x_1^\ell| \le \frac{1}{2}\beta\Gamma^{-1/2}$. Regardless, using the definitions \eqref{aux131a} and \eqref{aux131b} of $\alpha$ and $\beta$, it would imply that
\[
\int_{I_{x_2}} \abs{\partial_1 \wt{f}(x_1,x_2)}^2\, dx_1
\ge
\frac{32 \Lambda \Gamma^{5/2}}{25 \lambda^2}.
\]
Importantly, this lower bound is independent of the choice $x_2 \in S_4 \setminus S_5$. Thus, supposing $|S_5| \le \frac{1}{32}\kappa\Lambda\pre\Gamma^{-1/2}$, we compute
\[
\int_\Omega |\partial_1 \wt{f}|^2\, dx
\ge
\int_{|S_4\setminus S_5|} \int_{I_{x_2}} |\partial_1 \wt{f}|^2\, dx_1\, dx_2
\ge
\frac{32 \Lambda \Gamma^{5/2}}{25 \lambda^2}|S_4\setminus S_5|
\ge
\frac{\kappa}{25 \lambda^2} \Gamma^2.
\]
This establishes the claim with $c_4 = 5$.
\end{proof}

Now that we are working with heights $x_2$ such that each $x_1''$ is at least a distance of $\frac{1}{2}\beta\Gamma^{-1/2}$ away from $J$, we also would like to ensure that each $x_1''$ is at least a distance of $\frac{3}{2}\beta\Gamma^{-1/2}$ away from $\partial I_{x_2}$,
since we would like to have space to fit in $J^-.$

\begin{definition}
Under \Cref{hyp:secNash} and for fixed $0 < \kappa \le 1$ and $0 < \lambda < 1$, we define the set
\[
S_6 := \curly*{
x_2 \in S_5\ \bigg|\ |x_1'' - F_\ell(x_2)| > \frac{3}{2}\beta\Gamma^{-1/2} \text{ and } |x_1'' - F_r(x_2)| > \frac{3}{2}\beta\Gamma^{-1/2}
}.
\]
Fixing any $x_2 \in S_6$, we observe by the constructions of $S_5$ and $S_6$ that we can find a closed connected subinterval $J^-$ centered at $x_1''$ with length exactly $\beta\Gamma^{-1/2}$. Then, $J$ and $J^-$ will be disjoint and, in particular, $J^+$ and $J^-$ will be disjoint (recall the definition \eqref{aux131f} of $J^+$). Moreover, $J^-$ will be at least a distance of $\beta\Gamma^{-1/2}$ away from $\partial I_{x_2}$.
\end{definition}

\begin{lemma}\label{lem:secNash_BetweenDeltaAndDeltaInvStep5}
Under \Cref{hyp:secNash} and for fixed $0 < \kappa \le 1$ and $0 < \lambda < 1$, suppose that
\[
\nm{\wt{f}}_1 = 1,
\quad
\nm{\wt{f}}_2^2 \ge \kappa \Gamma,
\quad
|S_5| > \frac{1}{32}\kappa\Lambda\pre\Gamma^{-1/2}.
\]
If $|S_6| \le \frac{1}{64}\kappa\Lambda\pre\Gamma^{-1/2}$, then it must be true that
\[
\kappa^{5/2} \Gamma \lambda\pre \le \kappa^{1/2} \Gamma \lambda\pre \le c_5\nm{\partial_1 \wt{f}}_2 = c_5\nm{\partial_1 f}_2
\]
for some universal constant $c_5 > 0$ (independent even from $\Omega$).
\end{lemma}

\begin{proof}
Fix $x_2 \in S_5\setminus S_6$ and suppose without loss of generality that $x_1'' \in B = [F_\ell(x_2), x_1'-\frac{1}{2}\beta\Gamma^{-1/2}]$; see \Cref{def:secNash_B}. Since $x_1'' \in B$, $x_1'' \notin J$, and $x_2 \in S_5 \subseteq S_3$, it follows that
\[
F_r(x_2) - x_1'' > F_r(x_2) - x_1^\ell > \frac{3}{2}\beta\Gamma^{-1/2}.
\]
As $x_2 \notin S_6$, we must have $x_1'' - F_\ell(x_2) \le \frac{3}{2}\beta\Gamma^{-1/2}$. In turn, we claim that there must exist some $z$ satisfying $F_\ell(x_2) < z < x_1''$ such that $\sgn\paren{\wt{f}(x_1',x_2)} \wt{f}(z,x_2) \le -3\alpha \Gamma/10$. Indeed, if no such point $z$ exists, we obtain that
\[
\sgn\paren*{\wt{f}(x_1',x_2)} \int_{F_{\ell}(x_2)}^{x_1''} \wt{f}(x_1,x_2)\, dx_1
>
-(x_1'' - F_\ell(x_2))\frac{3}{10}\alpha \Gamma
\ge
-\frac{9}{20}\alpha\beta\Gamma^{1/2},
\]
which contradicts (\ref{eq:secNash_ToHelpChoosex1''}) as we recall that $\wt{f}(x_1,x_2)$ is signed for $x_1'' \le x_1 \le x_1'$ (again see \Cref{def:secNash_B}). We compute
\[
\int_z^{x_1''} |\partial_1 \wt{f}(x_1,x_2)|\, dx_1 \ge \abs*{\int_z^{x_1''} \partial_1 \wt{f}(x_1,x_2)\, dx_1} = \abs*{\wt{f}(x_1'', x_2) - \wt{f}(z, x_2)} \ge \frac{2\alpha \Gamma}{5},
\]
since we recall that $\wt{f}(x_1'', x_2)$ and $\wt{f}(z, x_2)$ must be oppositely signed, with $|\wt{f}(x_1'', x_2)| = \alpha\Gamma/10$ and $\sgn\paren{\wt{f}(x_1',x_2)} \wt{f}(z,x_2) \le -3\alpha \Gamma/10$. It follows that
\begin{align*}
\int_{I_{x_2}} \abs{\partial_1 \wt{f}(x_1,x_2)}^2\, dx_1
&\ge \int_{F_\ell(x_2)}^{x_1''} \abs{\partial_1 \wt{f}(x_1,x_2)}^2\, dx_1 \\
&\ge \frac{1}{x_1'' - F_\ell(x_2)} \paren*{\int_{z}^{x_1''} \abs{\partial_1 \wt{f}(x_1,x_2)} dx_1}^2 \\
&\ge \frac{8\Lambda \Gamma^{5/2}}{75\lambda^2}.
\end{align*}
If we instead had assumed that $B = [x_1'+\frac{1}{2}\beta\Gamma^{-1/2}, F_r(x_2)]$, a completely analogous argument gives the same inequality.

Altogether, the above discussion holds for any $x_2 \in S_5\setminus S_6$. Thus, supposing $|S_6| \le \frac{1}{64}\kappa\Lambda\pre\Gamma^{-1/2}$, we compute
\[
\int_\Omega |\partial_1 \wt{f}|^2\, dx
\ge
\int_{|S_5\setminus S_6|} \int_{I_{x_2}} |\partial_1 \wt{f}|^2\, dx_1\, dx_2 \\
\ge
\frac{8\Lambda \Gamma^{5/2}}{75\lambda^2}|S_5\setminus S_6|
\ge
\frac{\kappa}{8\cdot 75 \lambda^2} \Gamma^2.
\]
This establishes the claim with $c_5 = 10 \sqrt{6}$.
\end{proof}

The final refinement needed for the subset of heights is that we work only with those heights where $\wt{f}$ is bounded from above by $\frac{1}{2}\alpha\Gamma$ on each $J^-$. This is so that we may satisfy \ref{hyp:secNash_Lambda_PartB}.

\begin{definition}
Under \Cref{hyp:secNash} and for fixed $0 < \kappa \le 1$ and $0 < \lambda < 1$, we define the set
\[
S_7 := \curly*{
x_2 \in S_6 \ \bigg |\
|\wt{f}(x_1,x_2)| \le \frac{1}{2}\alpha \Gamma \text{ for all } x_1 \in J^-
}.
\]
\end{definition}

\begin{lemma}\label{lem:secNash_BetweenDeltaAndDeltaInvStep6}
Under \Cref{hyp:secNash} and for fixed $0 < \kappa \le 1$ and $0 < \lambda < 1$, suppose that
\[
\nm{\wt{f}}_1 = 1,
\quad
\nm{\wt{f}}_2^2 \ge \kappa \Gamma,
\quad
|S_6| > \frac{1}{64}\kappa\Lambda\pre\Gamma^{-1/2}.
\]
If $|S_7| \le \frac{1}{128}\kappa\Lambda\pre\Gamma^{-1/2}$, then it must hold that
\[
\kappa^{5/2} \Gamma \lambda\pre \le \kappa^{1/2} \Gamma \lambda\pre \le c_6\nm{\partial_1 \wt{f}}_2 = c_6\nm{\partial_1 f}_2
\]
for some universal constant $c_6 > 0$ (independent even from $\Omega$).
\end{lemma}

\begin{proof}
For each $x_2 \in S_6 \setminus S_7$, we know there is some $z \in J^-$ such that $|\wt{f}(z,x_2)| > \frac{1}{2}\alpha \Gamma$. Supposing for now that $z > x_1''$, we have the lower bound
\[
\int_{x_1''}^z |\partial_1 \wt{f}(x_1,x_2)|\, dx_1
\ge
\abs*{\int_{x_1''}^z \partial_1 \wt{f}(x_1,x_2)\, dx_1}
=
\abs*{\wt{f}(z, x_2) - \wt{f}(x_1'', x_2)}
>
\frac{2}{5}\alpha \Gamma,
\]
and hence
\[
\int_{I_{x_2}} \abs{\partial_1 \wt{f}(x_1,x_2)}^2\, dx_1
\ge
\int_{J^-} \abs{\partial_1 \wt{f}(x_1,x_2)}^2\, dx_1
\ge
\frac{1}{|J^-|} \paren*{\int_{x_1''}^z \abs{\partial_1 \wt{f}(x_1,x_2)} dx_1}^2
\ge
\frac{4\Lambda\Gamma^{5/2}}{25\lambda^2}.
\]
The same inequality holds if instead $z < x_1''$. Moreover, this bound is independent of the choice of $x_2 \in S_6 \setminus S_7$. Thus, if we suppose that $|S_7| \le \frac{1}{128}\kappa\Lambda\pre\Gamma^{-1/2}$, we compute
\[
\int_\Omega |\partial_1 \wt{f}|^2\, dx
\ge
\int_{|S_6\setminus S_7|} \int_{I_{x_2}} |\partial_1 \wt{f}|^2\, dx_1\, dx_2
\ge
\frac{4\Lambda\Gamma^{5/2}}{25\lambda^2} |S_6\setminus S_7|
\ge
\frac{\kappa}{32 \cdot 25 \lambda^2} \Gamma^2.
\]
This establishes the claim with $c_6 = 20\sqrt{2}$.
\end{proof}

Collecting all of the above lemmas in \Cref{subsubsec:NashUnstratHardCasesPart2} along with \Cref{prop:secNash_BetweenlambdaAndlambdaInvStep8}, we are now ready to finally complete the proof of \Cref{prop:secNash_BetweenlambdaAndlambdaInv}. We repeat the statement here.

\BetweenlambdaAndlambdaInv*

\begin{proof}
Consider first the case where $|S_k| \le \frac{1}{2^k}\kappa\Lambda\pre \Gamma^{-1/2}$ for some $k = 1, \ldots, 7$. Then, it follows
\[
\kappa^{5/2} \Gamma \lambda\pre \le \max\{c_1, c_2, c_3, c_4, c_5, c_6\} \nm{\partial_1 f}_2^2,
\]
by the results \Cref{lem:secNash_BetweenDeltaAndDeltaInvStep1} through \Cref{lem:secNash_BetweenDeltaAndDeltaInvStep6}. So, instead consider the case where $|S_7| > \frac{1}{128}\kappa\Lambda\pre \Gamma^{-1/2}$ and fix any choice of $x_2' \in S_7$ with corresponding intervals $J^+$ and $J^-$ with centers $x_1'$ and $x_1''$. Assume for now that $x_2' \in (0, h/2]$. Recall, by \Cref{def:secNash_J} and \Cref{def:secNash_B}, that $\wt{f}(\cdot, x_2')$ is signed as a function on the interval $J$ and has the same sign for $x_1$ between $x_1'$ and $x_1''$. If the sign is positive, $f$ satisfies the additional \Cref{hyp:secNash_Lambda} and we may invoke \Cref{prop:secNash_BetweenlambdaAndlambdaInvStep8} to conclude that at least one of the following holds true:
\[
\kappa^{21/2}\Gamma^{1/2}\lambda^{52} \le C_{10}\nm{\partial_1 f}_{\dH_0^{-1}}
\quad \text{or} \quad
\kappa^{5/2} \Gamma \lambda\pre \le c_0\nm{\partial_2 f}_2.
\]
If instead the sign of $\wt{f}(\cdot, x_2')$ is negative, we may simply apply \Cref{prop:secNash_BetweenlambdaAndlambdaInvStep8} to $-f$ to obtain the same conclusion. The only remaining case is when $x_2' \in [h/2, h)$, in which case we can apply the above reasoning to the function $f(x_1, h-x_2)$ defined on the domain $\{(x_1,x_2) \in \R^2\ |\ (x_1, h-x_2) \in \Omega\}$.

Collecting all of these cases, setting
\[
C_9 = C_9(\Omega) := \max\{c_0, c_1, c_2, c_3, c_4, c_5, c_6\} >0
\]
completes the proof.
\end{proof}

\subsection{The full Nash stratification inequality}
\label{subsec:TotalNash}

Given \Cref{thm:secNash_NashOnStratified,thm:secNash_NashOnUnstratified}, we can now shortly prove \Cref{thm:secIntro_NashTotal} which will be used when studying the PKS-IPM system; see \Cref{subsec:GWPProof} and  \Cref{prop:secGWP_GeneralArgument}. Afterward, we conclude with a discussion on the generality and sharpness of the results in this section.

\subsubsection{Proof of \texorpdfstring{\Cref{thm:secIntro_NashTotal}}{Theorem 1.1}}
\label{subsubsec:TotalNashProof}

We restate our main theorem below for the reader's convenience.

\NashTotal*

\begin{proof}
Without loss of generality, we further impose \Cref{norm:secSett_h,norm:secSett_eps} on the domain. \Cref{thm:secNash_NashOnStratified} gives us that
\[
\nm{\ov{f}-f_M}_2^2 \le C_6^{\frac{3}{4}} \nm{f-f_M}_2^{\frac{1}{4}} \nm{f-f_M}_1 \nm{\grad f}_2^{\frac{3}{4}},
\]
while \Cref{thm:secNash_NashOnUnstratified} implies that
\[
\nm{\wt{f}}_2^2
\le
C_{11}^{\frac{52}{141}} \nm{\partial_1 f}_{\dH_0\pre}^{\frac{1}{141}} \nm{f-f_M}_1^{\frac{52}{141}} \nm{f-f_M}_2^{\frac{59}{47}} \nm{\grad f}_2^{\frac{52}{141}}.
\]
Upon adding these two inequalities and dividing by $\nm{f-f_M}_2^{\frac{1}{4}}$, we obtain
\[
\nm{f-f_M}_2^{\frac{7}{4}} \le C_6^{\frac{3}{4}} \nm{f-f_M}_1 \nm{\grad f}_2^{\frac{3}{4}} + C_{11}^{\frac{52}{141}} \nm{\partial_1 f}_{\dH_0\pre}^{\frac{1}{141}} \nm{f-f_M}_1^{\frac{52}{141}} \nm{f-f_M}_2^{\frac{189}{188}} \nm{\grad f}_2^{\frac{52}{141}}.
\]
Raising both sides to the power of $8/7$, we get
\[
\nm{f-f_M}_2^2 \le 2^{\frac{8}{7}}C_6^{\frac{6}{7}} \nm{f-f_M}_1^{\frac{8}{7}} \nm{\grad f}_2^{\frac{6}{7}} + 2^{\frac{8}{7}}C_{11}^{\frac{416}{987}}  \nm{\partial_1 f}_{\dH_0\pre}^{\frac{8}{987}} \nm{f-f_M}_1^{\frac{416}{987}} \nm{f-f_M}_2^{\frac{54}{47}} \nm{\grad f}_2^{\frac{416}{987}}.
\]
Letting
\[
C_0 := C_0(\Omega) := \max\curly*{2^{\frac{8}{7}}C_6^{\frac{6}{7}}, 2^{\frac{8}{7}}C_{11}^{\frac{416}{987}}}, \qquad C_1 := C_1(\Omega) := C_0C_{\mathrm{N},2}^{\frac{27}{47}},
\]
completes the proof.
\end{proof}

\subsubsection{Discussion on domain assumptions and inequality sharpness}
\label{subsubsec:NashDiscussion}

To conclude \Cref{sec:Nash}, we now discuss examples of domains that do not satisfy all of the assumptions in \Cref{subsec:Domain}. Our purpose is to highlight some technical assumptions that we require, to explain the reason why we require them, and to justify why it is not unreasonable to impose these restrictions.

Besides \Cref{as:secSett_1}, which simply ensures that $\Omega$ be a nonempty, open, bounded, and connected planar domain, the most important requirement that we further impose is that each horizontal cross-section $I_{x_2}$ is a connected interval. An example of a domain that does not satisfy this condition is given in \Cref{fig:secNash_AssumpBreak1}. We impose connectedness of the cross-sections for the following reasons. First, this condition ensures that $\ov{f}$ can be represented easily as an integral functional of $f \in C^\infty(\ov{\Omega})$. Indeed, on more general domains such as the one given in \Cref{fig:secNash_AssumpBreak1}, the natural analog of $\ov{f}$ will be a multi-valued function of $x_2,$ leading to various technicalities. Second, the condition of connected cross-sections makes clear how to integrate over subsets of the domain $\Omega$. Examples of proofs where we rely on this are: the proof of \Cref{prop:secNash_StratCaps} when we integrate near the lower $\eps$-cap $L_\eps$; the vector field constructions in \Cref{lem:secNash_Sit1StratBulk,lem:secNash_Sit2StratBulk} towards the proof of \Cref{prop:secNash_StratBulk}; similarly in the vector field constructions in \Cref{lem:secNash_Sit1AbovelambdaInv,lem:secNash_Sit2AbovelambdaInv} towards the proof of \Cref{prop:secNash_AbovelambdaInv}; and multiple times throughout the construction of the test function $\Psi$ in \Cref{subsubsec:NashUnstratHardCasesPart1} toward the proof of \Cref{prop:secNash_BetweenlambdaAndlambdaInv}.

Another technicality that is forbidden specifically by \Cref{as:secSett_2} is when the left are right sides of the boundary $\partial \Omega$ are ever flat (even for just one point) away from the caps. More rigorously, the boundary charts $F_r, F_\ell \in C^\infty((h_L, h_U)) \cap C([h_L,h_U])$ are not allowed to obtain an infinite derivative at any height $h_L < x_2 < h_U$. An example of a domain that violates this is provided by \Cref{fig:secNash_AssumpBreak2}. We impose this condition as otherwise the explicit vector field constructions in \Cref{lem:secNash_Sit1StratBulk,lem:secNash_Sit2StratBulk} along with \Cref{lem:secNash_Sit1AbovelambdaInv,lem:secNash_Sit2AbovelambdaInv} become much more technical to write down.

We believe that \Cref{thm:secNash_NashOnStratified} and \Cref{thm:secNash_NashOnUnstratified} still hold true for domains such as those exemplified by \Cref{fig:secNash_AssumpBreak1} and \Cref{fig:secNash_AssumpBreak2}. Moreover, the restriction that the domain boundary be smooth is also not strictly required by our arguments above, and can likely be lowered to $C^2$ regularity if not further. However, we do not pursue these results as they would complicate the arguments presented above and would distract from the main purpose of this article.

\tikzset{every picture/.style={line width=0.75pt}}
\begin{figure}[!ht]%
\centering
\begin{subfigure}{.32\textwidth}%
\captionsetup{width=0.9\textwidth}
\centering
\begin{tikzpicture}[x=1pt,y=1pt,yscale=-0.5,xscale=0.5]
\draw   (186,53) .. controls (221.76,54.56) and (231.76,5.56) .. (279.76,5.56) .. controls (327.76,5.56) and (336.76,43.56) .. (332.76,83.56) .. controls (328.76,123.56) and (288.76,184.56) .. (264.76,197.8) .. controls (240.76,211.04) and (152.76,214.8) .. (142.76,178.8) .. controls (132.76,142.8) and (160.76,88.8) .. (135.76,83.8) .. controls (110.76,78.8) and (112.76,100.56) .. (77.76,104.56) .. controls (42.76,108.56) and (46.76,64.56) .. (55.76,43.56) .. controls (64.76,22.56) and (83.76,10.56) .. (106.76,13.56) .. controls (129.76,16.56) and (150.24,51.44) .. (186,53) -- cycle ;
\draw   (193.76,136.28) .. controls (193.76,121.77) and (210.1,110) .. (230.26,110) .. controls (250.42,110) and (266.76,121.77) .. (266.76,136.28) .. controls (266.76,150.79) and (250.42,162.56) .. (230.26,162.56) .. controls (210.1,162.56) and (193.76,150.79) .. (193.76,136.28) -- cycle ;
\draw  [color={rgb, 255:red, 208; green, 2; blue, 27 }  ,draw opacity=0.6 ][line width=4.5]  (193.76,136.28) .. controls (193.76,121.77) and (210.1,110) .. (230.26,110) .. controls (250.42,110) and (266.76,121.77) .. (266.76,136.28) .. controls (266.76,150.79) and (250.42,162.56) .. (230.26,162.56) .. controls (210.1,162.56) and (193.76,150.79) .. (193.76,136.28) -- cycle ;
\draw [color={rgb, 255:red, 208; green, 2; blue, 27 }  ,draw opacity=0.6 ][line width=4.5]    (99.76,12.12) .. controls (120.26,14.12) and (121.76,22.12) .. (149.76,40.12) .. controls (177.76,58.12) and (199.76,56.16) .. (220.76,37.16) .. controls (241.76,18.16) and (256.76,6.16) .. (279.76,5.56) ;
\draw [color={rgb, 255:red, 208; green, 2; blue, 27 }  ,draw opacity=0.6 ][line width=4.5]    (77.76,104.56) .. controls (90.76,102.86) and (98.76,100.6) .. (109.76,92.16) .. controls (120.76,83.72) and (124.76,81.72) .. (135.76,83.8) .. controls (146.76,85.88) and (147.51,103.91) .. (146.76,110.16) ;
\end{tikzpicture}
\caption{Horizontal cross-sections are not connected (in three ways).}
\label{fig:secNash_AssumpBreak1}
\end{subfigure}%
\begin{subfigure}{.32\textwidth}%
\captionsetup{width=0.9\textwidth}
\centering
\begin{tikzpicture}[x=0.75pt,y=0.75pt,yscale=-0.85,xscale=0.85]
\draw   (139.09,23.99) .. controls (158.09,23.99) and (202.09,32.9) .. (221.09,40.9) .. controls (240.09,48.9) and (248.09,51.9) .. (248.09,71.9) .. controls (248.09,91.9) and (232.09,90.99) .. (220.09,90.99) .. controls (208.09,90.99) and (188.09,89.99) .. (180.09,89.99) .. controls (172.09,89.99) and (158.09,100.99) .. (164.09,112.99) .. controls (170.09,124.99) and (204.09,128.99) .. (204.09,157.99) .. controls (204.09,186.99) and (179.71,206.61) .. (161.09,206.99) .. controls (142.46,207.36) and (126.09,177.99) .. (125.09,164.99) .. controls (124.09,151.99) and (133.09,119.99) .. (100.09,119.99) .. controls (67.09,119.99) and (50.09,79.99) .. (63.09,56.99) .. controls (76.09,33.99) and (120.09,23.99) .. (139.09,23.99) -- cycle ;
\draw [color={rgb, 255:red, 208; green, 2; blue, 27 }  ,draw opacity=0.6 ][line width=4.5]    (180.09,89.99) -- (230.09,90.99) ;
\draw  [draw opacity=0][fill={rgb, 255:red, 208; green, 2; blue, 27 }  ,fill opacity=0.6 ] (95.7,119.99) .. controls (95.7,117.56) and (97.66,115.6) .. (100.09,115.6) .. controls (102.51,115.6) and (104.47,117.56) .. (104.47,119.99) .. controls (104.47,122.41) and (102.51,124.37) .. (100.09,124.37) .. controls (97.66,124.37) and (95.7,122.41) .. (95.7,119.99) -- cycle ;
\end{tikzpicture}
\caption{Sides of boundary have flat portions (at point and a portion).}
\label{fig:secNash_AssumpBreak2}
\end{subfigure}%
\begin{subfigure}{.32\textwidth}%
\captionsetup{width=0.9\textwidth}
\centering
\begin{tikzpicture}[x=0.75pt,y=0.75pt,yscale=-0.8,xscale=0.8]
\draw   (80.09,30.33) .. controls (117.59,30.33) and (111.09,30.33) .. (189.09,30.33) .. controls (267.09,30.33) and (185.09,164.9) .. (178.09,171.19) .. controls (171.09,177.47) and (153.09,210.67) .. (128.09,210.67) .. controls (103.09,210.67) and (76.09,198.67) .. (68.09,177.67) .. controls (60.09,156.67) and (75.09,131.19) .. (71.09,113.19) .. controls (67.09,95.19) and (31.09,89.19) .. (29.09,72.5) .. controls (27.09,55.81) and (42.59,30.33) .. (80.09,30.33) -- cycle ;
\draw [color={rgb, 255:red, 208; green, 2; blue, 27 }  ,draw opacity=0.6 ][line width=4.5]    (73.09,29.9) -- (189.09,30.33) ;
\end{tikzpicture}
\caption{Flat top (and/or bottom) of the boundary.}
\label{fig:secNash_AssumpBreak3}
\end{subfigure}%
\caption{Three examples of domains which do not satisfy all of the requirements of \Cref{as:secSett_1,as:secSett_2,as:secSett_3}.}
\label{fig:secNash_AssumpBreak}
\end{figure}
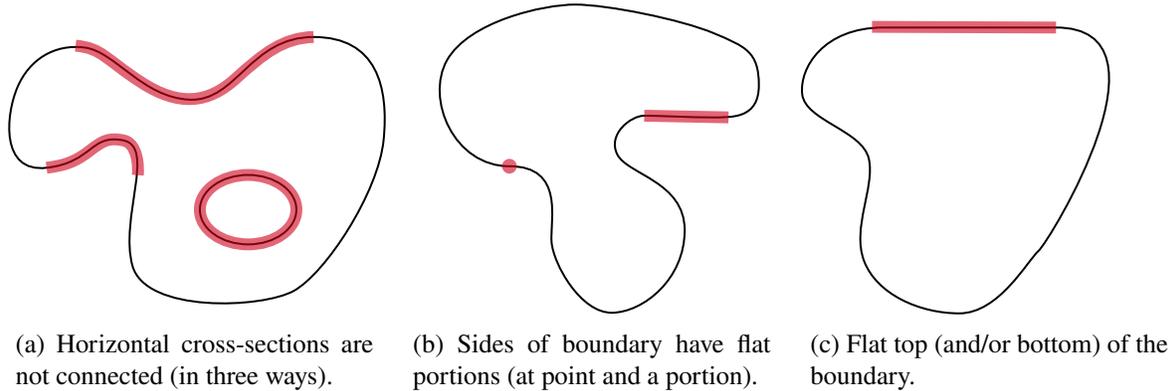

We also note that \Cref{as:secSett_3} disallows either of the caps of the domain $\Omega$ being flat, as exemplified in \Cref{fig:secNash_AssumpBreak3}. If exactly one of the caps is flat while the other still has cross-sections with length going to zero (in the fashion of \Cref{as:secSett_3}), then the arguments presented above can be repeated with very minor modifications to establish \Cref{thm:secNash_NashOnStratified} and \Cref{thm:secNash_NashOnUnstratified} for such domains. However, if both the bottom and top caps are flat, then the proof of \Cref{thm:secNash_NashOnStratified} can be significantly simplified, yielding an even better scaling! Indeed, for such domains, one no longer needs \Cref{prop:secNash_StratCaps} nor \Cref{lem:secNash_Sit1StratBulk}. Instead, one can very easily modify the argument of \Cref{lem:secNash_Sit2StratBulk} to obtain that
\[
\nm{\ov{f}-f_M}_2^3 \lesssim_\Omega \nm{f-f_M}_1^2 \nm{\grad f}_2
\]
for all smooth functions $f$ on a domain $\Omega$ satisfying \Cref{as:secSett_1,as:secSett_2} with the modification that the top and bottom of the domain be flat (meaning $F_\ell(h_L) < F_r(h_L)$, $F_\ell(h_U) < F_r(h_U)$ while still $F_\ell'(h_L)$, $F_\ell'(h_U)$, $F_r'(h_L)$, $F_r'(h_U)$ are all infinite). This is a refined two-dimensional Nash inequality corresponding to the $d = 1$ dimensional scaling of the classical inequality (\ref{eq:secIntro_Nash}). Over such domains with both top and bottom flat caps, the proof of \Cref{thm:secNash_NashOnUnstratified} can also be easily modified; one need only make a minor modification to the proof of \Cref{lem:secNash_eta}. This however does not lead to any improvement of \Cref{thm:secNash_NashOnUnstratified}.

Altogether, we have justified the choice of assumptions in \Cref{subsec:Domain} and have explained how to generalize the Nash stratification inequalities to domains with top and bottom flat boundary portions. The improvement in \Cref{thm:secNash_NashOnUnstratified} over such domains leads us to our final point of discussion. We make no claims regarding the optimality of the constants in \Cref{thm:secNash_NashOnStratified,thm:secNash_NashOnUnstratified,thm:secIntro_NashTotal}. However, we do keep careful track of the constants we obtain throughout \Cref{subsec:StratNash,subsec:UnstratNash} as it may be interesting to the reader how the constants depend on more information about the domain $\Omega$ besides its area $|\Omega|$. Specifically, the constants $\eps_*$, $K_*$, $R_*$, and $M_*$ --- which capture information about boundary curvature --- all appear in the derivations of the universal constants for \Cref{subsec:StratNash,subsec:UnstratNash}. Although they are the best that we could obtain, we also make no claims regarding the optimality of exponents in \Cref{thm:secNash_NashOnStratified,thm:secNash_NashOnUnstratified,thm:secIntro_NashTotal}. We do note though that, importantly, since the PKS-IPM system (\ref{eq:secIntro_PKSIPM}) is $L^\infty_t L^1_x$-critical, any improved scaling such as that provided by \Cref{thm:secNash_NashOnStratified} and any inclusion of the $\dH_0\pre$ seminorm of $\partial_1 f$, as is done by \Cref{thm:secNash_NashOnUnstratified}, would lead to a proof of global regularity for the PKS-IPM system; see \Cref{prop:secGWP_GeneralArgument}.

It is also worth pointing out that the question of optimality of constants and exponents may heavily depend on class of domains being considered. Indeed, we see from the discussion above that allowing domains with flat top and bottom boundaries leads to a Nash inequality for $\ov{f}$ with a formal one-dimensional scaling. Yet, if we do not allow for flat top and bottom boundaries, then \Cref{thm:secNash_NashOnStratified} only provides a $3/2$-dimensional formal scaling; see \Cref{rem:secIntro_NashStratInterp1}.

\section{Global well-posedness of PKS-IPM on general 2D domains}
\label{sec:GWP}

Now that we have established the anisotropic Nash stratification inequality \Cref{thm:secIntro_NashTotal}, the goal for \Cref{sec:GWP} is to provide an example of its application to the study of a critical nonlinear partial differential equation: the two-dimensional PKS-IPM system (\ref{eq:secIntro_PKSIPM}). We begin in \Cref{subsec:PKSIPMPreliminaries} by reviewing well-known properties of solutions to (\ref{eq:secIntro_PKSIPM}), many of which also hold for a much broader family of chemotaxis equations (as summarized in \Cref{subsec:BackChemo,subsec:BackChemoFluids}). Next, in \Cref{subsec:Variance}, we establish the important a-priori variance inequality and that the solution's variance is a blow-up control quantity. We then introduce the potential energy of a solution, in \Cref{subsec:Potential}, which relates the $\dH_0\pre$ norm of $\partial_1 \rho$ to the solution's $\dH^1$ norm and variance. We note that the results of \Cref{subsec:PKSIPMPreliminaries,subsec:Variance,subsec:Potential} only require the domain $\Omega$ to satisfy \Cref{as:secSett_1} and so we frequently work under the following hypotheses throughout \Cref{sec:GWP}.

\begin{hypotheses}\label{hyp:secGWP}
We fix a domain $\Omega \subseteq \R^2$ satisfying \Cref{as:secSett_1} and \Cref{norm:secSett_h} (but not necessarily \Cref{as:secSett_2,as:secSett_3} nor \Cref{norm:secSett_eps}). We also use the notation defined in \Cref{sec:Setting}. Lastly, suppose $\rho \in C^\infty( \ov{\Omega} \times [0, T_*))$ is a smooth solution up to some time $0 < T_* \le \infty$ for the two-dimensional PKS-IPM system (\ref{eq:secIntro_PKSIPM}) with respect to a fixed gravitational constant $g \neq 0$ and a nonnegative initial data $\rho_0(x) = \rho(x,0) \ge 0$ satisfying the Neumann boundary condition $\partial_n \rho|_{\partial \Omega} = 0$.
\end{hypotheses}

Finally, in \Cref{subsec:GWPProof}, we will further suppose that the domain satisfy all the assumptions of \Cref{subsec:Domain}. Then we may combine the a-priori variance inequality, the potential energy estimates, and the Nash stratification inequality to obtain a integro-differential system of inequalities. An ODE argument then gives the desired global regularity result of \Cref{thm:secIntro_GWP}.

\subsection{Review of basic solution properties}
\label{subsec:PKSIPMPreliminaries}

For the sake of completeness, we review here well-known properties that hold when a very general fluid advection term is added to the general PKS equation (\ref{eq:secIntro_PKSGeneral}). They have all been observed by previous works.

We begin by stating the preservation of the solution's nonnegativity and total mass. These observations do not make use of the specific equations for the chemical concentration $c(x,t)$ nor the fluid velocity $u(x,t)$ in (\ref{eq:secIntro_PKSIPM}). Rather, we will only need the Neumann boundary condition on $c(x,t)$ and that $u(x,t)$ be an incompressible vector field obeying the no-flux boundary condition.

\begin{lemma}
Given \Cref{hyp:secGWP}, we have that $\rho(x,t) \ge 0$ for all $0 \le t < T_*$.
\end{lemma}

This is a standard consequence of the parabolic comparison principle. 

\begin{lemma}
Given \Cref{hyp:secGWP}, the spatial average
\[
\rho_M(t) := \frac{1}{|\Omega|}\int_\Omega \rho(x,t)\, dx
\]
must be constant in time and is equal to
\[
\rho_M := \rho_M(0) = \frac{1}{|\Omega|}\int_\Omega \rho_0(x)\, dx.
\]
\end{lemma}

This is a direct consequence of incompressibility of $u$ and boundary conditions. 

We conclude with a crucial observation that Darcy's law for the fluid velocity $u(x,t)$ in the third line of (\ref{eq:secIntro_PKSIPM}) can be replaced by the following Biot--Savart law.

\begin{lemma}
Given \Cref{hyp:secGWP}, the velocity field $u(x,t)$ associated to $\rho(x,t)$ satisfies
\begin{equation}\label{eq:secGWP_BiotSavart}
u = g\grad^\perp(-\lap_D)\pre \partial_1 \rho,
\end{equation}
where $\grad^\perp := (-\partial_2, \partial_1)$ and $(-\lap_D)\pre$ denotes the inverse of the Dirichlet Laplacian on $\Omega$.
\end{lemma}

\begin{proof}
The incompressibility of $u$ ensures the existence of a scalar stream function $\psi$, meaning $u = \grad^\perp \psi$. This stream function is unique up to an additive constant. The no-flux boundary condition on $u$ in (\ref{eq:secIntro_PKSIPM}) allows us to fix $\psi$ satisfying Dirichlet boundary conditions on $\Omega$. Thus, the claim follows from taking the scalar curl of both sides of the third line of (\ref{eq:secIntro_PKSIPM}).
\end{proof}

As these lemmas are used so frequently throughout this work, we may not explicitly cite them.

\subsection{Variance as a blow-up control quantity}
\label{subsec:Variance}

We now consider the $L^2$-variance of the solution $\rho$ to (\ref{eq:secIntro_PKSIPM}), proving a key energy inequality and a conditional regularity statement. We first establish an important a-priori estimate. The techniques used are standard and this estimate or similar variations of it have been used in many other works on the PKS equation (\ref{eq:secIntro_PKSGeneral}).

\begin{proposition}\label{prop:secGWP_AprioriVarianceIneq}
Given \Cref{hyp:secGWP}, we have that
\[
\frac{d}{dt}\nm{\rho(\cdot,t) - \rho_M}_2^2 + \nm{\grad \rho(\cdot,t)}_2^2 \le C_{12} \nm{\rho(\cdot,t)-\rho_M}_2^4 + 2\rho_M\nm{\rho(\cdot,t) - \rho_M}_2^2,
\]
where $C_{12} = C_{12}(\Omega) > 0$ is a universal constant.
\end{proposition}

\begin{proof}
We test the first equation of (\ref{eq:secIntro_PKSIPM}) on both sides by $\rho - \rho_M$ to obtain
\[
\frac{1}{2}\frac{d}{dt} \nm{\rho(\cdot, t)-\rho_M}_2^2 + \nm{\grad \rho(\cdot, t)}_2^2 = - \int_\Omega (\rho-\rho_M) \div(\rho \grad(-\lap_N)\pre[\rho-\rho_M])\, dx.
\]
The contribution from the advection term in (\ref{eq:secIntro_PKSIPM}) is zero due to the incompressibility and boundary conditions on $u(x,t)$, while the diffusion term is integrated by parts to produce the $\dH^1(\Omega)$ norm on the left-hand side. The right-hand side arises from the chemotaxis term in (\ref{eq:secIntro_PKSIPM}), which, using the boundary conditions for $\rho,$ we can simplify as follows:
\begin{align*}
-\int_\Omega (\rho-\rho_M)\div(\rho \grad(-\lap_N)\pre[\rho-\rho_M])\, dx
&= \frac{1}{2}\int_\Omega \grad \rho^2 \cdot \grad(-\lap_N)\pre[\rho-\rho_M]\, dx \\
&= \frac{1}{2}\int_\Omega \rho^2 (\rho-\rho_M)\, dx\\
&= \frac{1}{2}\int_\Omega (\rho-\rho_M)^3\, dx + \rho_M \nm{\rho-\rho_M}_2^2.
\end{align*}
Applying the two-dimensional Gagliardo--Nirenberg--Sobolev inequality
\[
\nm{f-f_M}_3^3 \le C_{\mathrm{GNS}} \nm{f-f_M}_2^2 \nm{\grad f}_2,
\]
which holds for any function in $f \in H^1(\Omega)$ with $C_{\mathrm{GNS}} = C_{\mathrm{GNS}}(\Omega) > 0$ denoting the associated universal constant, we obtain
\begin{align*}
\frac{1}{2}\frac{d}{dt} \nm{\rho-\rho_M}_2^2 + \nm{\grad \rho}_2^2 &\le C_{\mathrm{GNS}}\nm{\rho-\rho_M}_2^2\nm{\grad \rho}_2 + \rho_M \nm{\rho-\rho_M}_2^2 \\
&\le \frac{1}{2}C_{\mathrm{GNS}}^2\nm{\rho-\rho_M}_2^4 + \frac{1}{2}\nm{\grad \rho}_2^2 + \rho_M \nm{\rho-\rho_M}_2^2.
\end{align*}
Rearranging concludes the proof with $C_{12} = C_{12}(\Omega) := C_{\mathrm{GNS}}^2$.
\end{proof}

\begin{corollary}\label{cor:secGWP_AprioriVarianceIneq}
Given \Cref{hyp:secGWP} and if $\nm{\rho(\cdot,t) - \rho_M}_2^2 \ge 1$ at some time $0 \le t < T_*$, then
\[
\frac{d}{dt}\nm{\rho(\cdot,t) - \rho_M}_2^2 + \nm{\grad \rho(\cdot,t)}_2^2 \le (C_{12} + 2\rho_M)\nm{\rho(\cdot,t)-\rho_M}_2^4
\]
at that time $t$.
\end{corollary}

\begin{remark}
\Cref{cor:secGWP_AprioriVarianceIneq} also holds for the two-dimensional parabolic-elliptic PKS equation (\ref{eq:secIntro_PKSGeneral}, \ref{eq:secIntro_PKSeqC}) without fluid advection. Indeed, the proof of \Cref{prop:secGWP_AprioriVarianceIneq} only leverages the incompressibility and boundary conditions on the fluid velocity $u(x,t)$, with no need for the Biot--Savart law (\ref{eq:secGWP_BiotSavart}).
\end{remark}

This corollary reveals the criticality of the general PKS equation (\ref{eq:secIntro_PKSGeneral}, \ref{eq:secIntro_PKSeqC}) and the PKS-IPM system (\ref{eq:secIntro_PKSIPM}) with respect to the scaling of the classical two-dimensional Nash inequality (\ref{eq:secIntro_Nash}), as we discussed in \Cref{subsec:BackChemo} and \Cref{subsec:BackChemoIPM}. This motivates the need for the anisotropic refinement given by  \Cref{thm:secIntro_NashTotal}, which we eventually employ in \Cref{prop:secGWP_GeneralArgument}.

We note here that the above a-priori variance estimate, when combined with the properties established in \Cref{subsec:PKSIPMPreliminaries}, leads to a proof of the local well-posedness of smooth nonnegative solutions to (\ref{eq:secIntro_PKSIPM}). Importantly, the variance of a solution serves as a control quantity for finite-time blow-up.

\begin{theorem}\label{thm:secGWP_LWP}
Fix a domain $\Omega \subseteq \R^2$ satisfying \Cref{as:secSett_1}, fix a choice of gravitational constant $g \neq 0$, and fix any nonnegative initial data $\rho_0 \in C^\infty(\ov{\Omega})$ satisfying the Neumann boundary condition $\partial_n \rho|_{\partial \Omega} = 0$. Then, there exists a unique corresponding solution $\rho(x,t)$ to the system (\ref{eq:secIntro_PKSIPM}) in the class $C^\infty(\ov{\Omega} \times [0,T_*))$, where $0 < T_* \le \infty$ denotes the maximal lifespan of $\rho(x,t)$. If $T_* < \infty$, then we further have that
\[
\lim_{t \to T_*} \nm{\rho(\cdot, t) - \rho_M}_2^2 = \infty.
\]
Note that we always work with the chemical concentration $c$ that satisfies $c = (-\lap_N)\pre[\rho-\rho_M]$; harmonic perturbations are not allowed.
\end{theorem}

\begin{proof}
We only sketch the argument as the details are completely standard and have been carried out in other works \cites{W12,KX16,HKY25}. First, we may rewrite the PKS-IPM system (\ref{eq:secIntro_PKSIPM}) as a perturbation of the heat equation:
\[
\begin{cases}
\partial_t \rho - \lap \rho = - \div\paren*{g \rho \grad^\perp(-\lap_D)\pre\partial_1 \rho + \rho \grad (-\lap_N)\pre[\rho-\rho_M]}, \\
\partial_n \rho|_{\partial \Omega} = 0, \\
\rho(\cdot, 0) = \rho_0 \ge 0,
\end{cases}
\]
where we have used the Biot--Savart law (\ref{eq:secGWP_BiotSavart}) for the fluid velocity. Duhamel's principle motivates defining the integral functional
\begin{align*}
\Phi[f](x,t)
&:= e^{t\lap_N}\brac*{\rho_0}(x,t) \\
&\quad -\int_0^t e^{(t-s)\lap_N} \brac*{\div\paren*{g f(\cdot,s) \grad^\perp(-\lap_D)\pre\partial_1 f(\cdot,s)}}(x,s)\, ds\\
&\quad -\int_0^t e^{(t-s)\lap_N} \brac*{\div\paren*{f(\cdot,s) \grad (-\lap_N)\pre[f(\cdot,s)-\rho_M]}}(x,s)\, ds,
\end{align*}
for $f \in L^\infty( [0, T); C^0(\Omega) )$. Here $C^0(\Omega)$ is the space of continuous functions on $\Omega$ while $e^{t\lap_N}$ denotes the standard Neumann heat semigroup. Note we consider $\rho_0$ to be a fixed function and $\rho_M > 0$ to be a fixed constant. Moreover, $T > 0$ will be a time determined by the size of the initial data $\rho_0$.

The next step is to establish $\Phi$ as a contraction on $L^\infty( [0, T); C^0(\Omega) )$, which can be done via typical Neumann heat semigroup estimates, elliptic estimates, and various appropriate Sobolev inequalities. Then, a classical fixed point argument (after some boot-strapping to improve regularity) gives local well-posedness of (\ref{eq:secIntro_PKSIPM}) in the class of functions $C^\infty(\ov{\Omega} \times [0, T_*))$ with non-negative initial data, with $T_*$ denoting the maximal life span of the solution. Moreover, a consequence of this argument is that
\[
\lim_{t \to T_*} \nm{\rho(\cdot, t) - \rho_M}_\infty^2 = \infty
\]
whenever $0< T_* < \infty$ is finite. This estimate was first carried out in \cite{W12}*{Lemma 2.1} for a different chemotaxis-fluid system than the PKS-IPM system. The bounds for (\ref{eq:secIntro_PKSIPM}) on the periodic strip $\T \times (0, \pi)$ have been derived in \cite{HKY25}*{Theorem 3.1}. The more general domain setting that we consider here requires no changes to the proof of \cite{HKY25}*{Theorem 3.1} due to the Neumann boundary conditions in (\ref{eq:secIntro_PKSIPM}).

The final step is to improve the control quantity from $L^\infty$ to $L^2$. This was first done for a passive PKS-fluid system in \cite{KX16}*{Proposition 9.1}. Their argument can be directly adapted to the general bounded domain case (similarly to \cite{HKY25}, where periodic strip was considered), again due to the Neumann boundary conditions in (\ref{eq:secIntro_PKSIPM}). Upon doing so, one obtains that
\[
\nm{\rho(\cdot, t) - \rho_M}_\infty \lesssim_\Omega \max\left\{1,\ \rho_M,\ \frac{1}{4}\sup_{t \in [0,T]} \nm{\rho(\cdot, t) - \rho_M}_2^2\right\}
\]
for every solution $\rho \in C^\infty(\ov{\Omega} \times [0,T])$ with nonnegative initial data $\rho_0 \ge 0$. This bound combined with the differential inequality in \Cref{prop:secGWP_AprioriVarianceIneq} completes the proof; see \cite{HKY25}*{Corollary 3.3}.
\end{proof}

\subsection{Potential energy and a stratification norm bound}
\label{subsec:Potential}

The gravitational coupling force given in the third line of  (\ref{eq:secIntro_PKSIPM}) plays a significant role in the dynamics of the PKS-IPM system, motivating us to define an associated potential energy. By the nonnegativity of the solution $\rho$, one immediately observes that the potential energy is bounded.

\begin{definition}
Given \Cref{hyp:secGWP} with $g > 0$, we define the potential energy of $\rho$ to be
\[
E(t) := \int_\Omega x_2\rho(x,t)\, dx
\]
for any $0 \le t < T_*$.
\end{definition}

\begin{lemma}\label{lem:secGWP_BoundPotentialEnergy}
Given \Cref{hyp:secGWP} with $g > 0$, we have that
\[
0 \le E(t) \le h\nm{\rho_0}_1,
\]
for any $0 \le t < T_*$. Recall $h$ is the height of $\Omega$; see \Cref{norm:secSett_h}.
\end{lemma}

Although \Cref{lem:secGWP_BoundPotentialEnergy} is elementary, the boundedness of the potential energy is a crucial observation. It implies that the time derivative $E'(t)$ cannot be too large in magnitude (negative or positive) for too long of time intervals. Thus, we derive an expression for $E'(t)$, known as the gravitational power in physics terminology, which crucially involves the $\dH_0\pre$ norm of the solution $\rho$.

\begin{lemma}\label{lem:secGWP_Power}
Given \Cref{hyp:secGWP} with $g > 0$, the gravitational power satisfies
\[
E'(t) = -g\nm{\partial_1 \rho(\cdot,t)}_{\dH_0^{-1}}^2 - \int_\Omega \partial_2 \rho\, dx + \int_\Omega \rho \partial_2 (-\lap_N)\pre [\rho-\rho_M]\, dx.
\]
The first term comes from Darcian fluid advection; the second, from diffusion; and the third, from chemotaxis.
\end{lemma}

This result was first observed in \cite{HKY25}*{Proposition 4.6} and the proof is exactly the same in our context.

Our goal is is to derive a system of inequalities for three quantities: the $L^2$ norm of $\rho$, the $\dH^1$ norm of $\rho$, and the $\dH_0\pre$ norm of $\partial_1 \rho$. To obtain it, we will use the gravitational power in \Cref{lem:secGWP_Power} above, the Nash stratification inequality in \Cref{thm:secIntro_NashTotal}, and the variance estimate in \Cref{cor:secGWP_AprioriVarianceIneq}. To complete this task, we need to provide estimates that control the diffusion and chemotaxis terms of the gravitational power in \Cref{lem:secGWP_Power}.

We begin with a lemma controlling the contribution from diffusion which does not require the function to solve the PKS-IPM system (\ref{eq:secIntro_PKSIPM}).

\begin{lemma}\label{lem:secGWP_DiffusionEst}
Fix a domain $\Omega \subseteq \R^2$ satisfying \Cref{as:secSett_1} and let $f \in C^\infty(\ov{\Omega})$. Then, we have that
\[
\abs*{\int_\Omega \partial_2 f\, dx} \le C_{13} f_M^{1/4}\nm{\grad f}_2^{3/4},
\]
where $C_{13} = C_{13}(\Omega) > 0$ is a universal constant.
\end{lemma}

\begin{proof}
Integrating by parts, we obtain
\[
\abs*{\int_\Omega \partial_2 f\, dx}
= \abs*{\int_\Omega \partial_2(f-f_M)\, dx}
= \abs*{\int_{\partial \Omega} (f-f_M) n_2\, dx},
\]
which we may naively estimate as
\[
\abs*{\int_\Omega \partial_2 f\, dx}
\le \nm{f-f_M}_{L^1(\partial\Omega)}
\lesssim_{r, \Omega} \nm{f-f_M}_{L^r(\partial\Omega)}
\]
for any $1 < r < \infty$. Invoking the standard Besov trace embedding (for example, see \cite{AF03}*{Theorem 7.43 and Remark 7.45}), it follows
\[
\abs*{\int_\Omega \partial_2 f\, dx}
\lesssim_{r, \Omega} \nm{f-f_M}_{L^r(\partial\Omega)}
\lesssim_{r, s, p, \Omega} \nm{f-f_M}_{W^{s,p}}
\]
for any $0 < s < 1$ and $1 \le p < \infty$ such that $1 < sp < 2$ and $r = p/(2-sp) > p$. We recall here that the (real-interpolation based) fractional Sobolev space $W^{s,p}(\Omega)$ norm is given by
\[
\nm{f}_{W^{s,p}(\Omega)} := \nm{f}_p + \paren*{\int_\Omega\int_\Omega \frac{|f(x)-f(y)|^p}{|x-y|^{2+sp}}\, dx\, dy}^{1/p},
\]
with this being equivalent to the standard Besov space $B^s_{p,p}(\Omega)$ norm. We next invoke the fractional Gagliardo--Nirenberg--Sobolev interpolation inequality proven by Brezis and Mironsecu \cite{BM18}*{Theorem 1}, giving us
\[
\abs*{\int_\Omega \partial_2 f\, dx}
\lesssim_{r, s, p, \Omega} \nm{f-f_M}_{W^{s,p}(\Omega)}
\lesssim_{r, s, p, \theta, \Omega} \nm{f-f_M}_1^\theta \nm{f-f_M}_{H^1}^{1-\theta},
\]
for $\theta \in (0,1)$ such that $1 - s = \theta = (2/p) - 1$. We may simplify the right-hand side via the Poincar\'{e}--Wirtinger inequality:
\[
\abs*{\int_\Omega \partial_2 f\, dx}
\lesssim_{r, s, p, \theta, \Omega} \nm{f-f_M}_1^\theta \nm{f-f_M}_{H^1}^{1-\theta}
\lesssim_{r, s, p, \theta, \Omega} \nm{f-f_M}_1^\theta \nm{\grad f}_2^{1-\theta}.
\]
Working through the numerology, we can obtain the bound
\[
\abs*{\int_\Omega \partial_2 f\, dx} \lesssim_{\delta, \Omega} \nm{f-f_M}_1^{1/3 - \delta}\nm{\grad f}_2^{2/3 + \delta} \lesssim_{\delta, \Omega} f_M^{1/3 - \delta}\nm{\grad f}_2^{2/3 + \delta}
\]
for any fixed $0 < \delta < 1/3$ upon setting
\[
0 < \eps := \frac{9\delta}{8-6\delta} < 1/2,
\quad
p := \frac{3}{2} + \eps,
\quad
r := \frac{p}{4-2p},
\quad
s := 2-\frac{2}{p},
\quad
\theta := 1-s.
\]
The specific choice of $\delta := 1/12$ returns the statement of \Cref{lem:secGWP_DiffusionEst}.
\end{proof}

\begin{remark}
From the proof, we see that the statement of \Cref{lem:secGWP_DiffusionEst} can be improved to
\[
\abs*{\int_\Omega \partial_2 f\, dx} \lesssim_{\delta, \Omega} f_M^{1/3 - \delta}\nm{\grad f}_2^{2/3 + \delta}
\]
for any fixed $0 < \delta < 1/3$. We cannot obtain the edge case corresponding to $\delta = 0$ due to the failure of the Besov trace embedding for $L^1(\partial \Omega)$. Regardless, we need only any statement that is stronger than
\[
\abs*{\int_\Omega \partial_2 f\, dx} \le \nm{\grad f}_2
\]
for the arguments to follow; see \Cref{prop:secGWP_GeneralArgument}. This explains why we take $\delta = 1/12$ in the proof of \Cref{lem:secGWP_DiffusionEst}, as it suffices for our purposes and leads to nice numerology.
\end{remark}

We next provide a lemma controlling the contribution from chemotaxis which also does not require the function to solve the PKS-IPM system (\ref{eq:secIntro_PKSIPM}).

\begin{lemma}\label{lem:secGWP_ChemoEst}
Fix a domain $\Omega \subseteq \R^2$ satisfying \Cref{as:secSett_1} and let $f \in C^\infty(\ov{\Omega})$ be a nonnegative function. Then, we have that
\[
\abs*{\int_\Omega f \partial_2 (-\lap_N)\pre [f-f_M]\, dx} \le C_{14} \paren*{f_M^{2/3} \nm{f-f_M}_2^{4/3} +f_M^{5/3}\nm{f-f_M}_2^{1/3}},
\]
where $C_{14} = C_{14}(\Omega) > 0$ is a universal constant.
\end{lemma}

\begin{proof}
This result was first observed in \cite{HKY25}*{Proposition 4.6} and the proof is exactly the same in our context. First, we have by H\"{o}lder's inequality that
\[
\abs*{\int_\Omega f \partial_2 (-\lap_N)\pre [f-f_M]\, dx} \le \nm{f}_2 \nm{\partial_2 (-\lap_N)\pre [f-f_M]}_2,
\]
which leads to
\begin{align*}
\abs*{\int_\Omega f \partial_2 (-\lap_N)\pre [f-f_M]\, dx}
&\lesssim_\Omega \nm{f}_2 \nm{\partial_2 (-\lap_N)\pre [f-f_M]}_{W^{1,6/5}} \\
&\le \nm{f}_2 \nm{(-\lap_N)\pre [f-f_M]}_{W^{2,6/5}},
\end{align*}
via the standard two-dimensional Sobolev embedding of $W^{1,6/5}(\Omega)$ into $L^2(\Omega)$. Then, the classical $L^p$ elliptic estimate combined with the log-convexity of the $L^p$ norm give
\begin{align*}
\abs*{\int_\Omega f \partial_2 (-\lap_N)\pre [f-f_M]\, dx}
&\lesssim_\Omega \nm{f}_2 \nm{f-f_M}_{6/5} \\
&\le \nm{f}_2 \nm{f-f_M}_1^{2/3}\nm{f-f_M}_2^{1/3},
\end{align*}
which shows why the exponent $6/5$ was chosen for the Sobolev embedding. To conclude the proof, one uses the trivial bounds
\[
\nm{f-f_M}_1 \le 2|\Omega|f_M, \quad \nm{f}_2 \le \nm{f-f_M}_2 + |\Omega|^{1/2}f_M,
\]
which hold by the nonnegativity of $f$.
\end{proof}

Collecting the lemmas of this section, we are able to shortly prove the following key bound on the $\dH_0\pre$ norm of $\partial_1 \rho$.

\begin{proposition}\label{prop:secGWP_StratBound}
Given \Cref{hyp:secGWP} with $g > 0$, we have
\begin{align*}
g \int_{T_1}^{T_2} \nm{\partial_1 \rho(\cdot,t)}_{\dH_0^{-1}}^2\, dt
&\le h|\Omega|\rho_M + C_{13} \rho_M^{1/4} \int_s^r \nm{\grad \rho(\cdot, t)}_2^{3/4}\, dt \\
&\quad + C_{14} \rho_M^{2/3} \int_{T_1}^{T_2} \nm{\rho(\cdot, t)-\rho_M}_2^{4/3}\, dt + C_{14}\rho_M^{5/3} \int_{T_1}^{T_2} \nm{\rho(\cdot, t)-\rho_M}_2^{1/3}\, dt
\end{align*}
for any time interval $[T_1, T_2] \subseteq [0, T_*)$.
\end{proposition}

\begin{proof}
Integrating the expression for gravitational power given by \Cref{lem:secGWP_Power}, we obtain
\begin{align*}\label{eq:secGWP_StratBound1}
g \int_{T_1}^{T_2} \nm{\partial_1 \rho(\cdot,t)}_{\dH_0^{-1}}^2\, dt
&= E(T_1)-E(T_2)
- \int_{T_1}^{T_2} \int_\Omega \partial_2 \rho\, dx\, dt \\
&\quad + \int_{T_1}^{T_2} \int_\Omega \rho \partial_2 (-\lap_N)\pre [\rho-\rho_M]\, dx\, dt. \numberthis
\end{align*}
We may bound the first term on the right-hand side of (\ref{eq:secGWP_StratBound1}) by \Cref{lem:secGWP_BoundPotentialEnergy}:
\[
|E(T_1)-E(T_2)| \le h\nm{\rho_0}_1 = h|\Omega|\rho_M.
\]
The second term on the right-hand side of (\ref{eq:secGWP_StratBound1}) is handled by \Cref{lem:secGWP_DiffusionEst}:
\[
\int_{T_1}^{T_2} \abs*{\int_\Omega \partial_2 \rho\, dx}\, dt
\le C_{13} \rho_M^{1/4} \int_{T_1}^{T_2} \nm{\grad \rho(\cdot, t)}_2^{3/4}\, dt.
\]
Lastly, the third term on the right-hand side of (\ref{eq:secGWP_StratBound1}) is handled by \Cref{lem:secGWP_ChemoEst}:
\begin{align*}
\int_{T_1}^{T_2} \abs*{\int_\Omega \rho \partial_2 (-\lap_N)\pre [\rho-\rho_M]\, dx}\, dt
&\le
C_{14} \rho_M^{2/3} \int_{T_1}^{T_2} \nm{\rho(\cdot, t)-\rho_M}_2^{4/3}\, dt \\
&\quad + C_{14}\rho_M^{5/3} \int_{T_1}^{T_2} \nm{\rho(\cdot, t)-\rho_M}_2^{1/3}\, dt.
\end{align*}
Collecting all of the estimates completes the proof.
\end{proof}

\begin{corollary}\label{cor:secGWP_StratBound}
Given \Cref{hyp:secGWP} with $g > 0$, we have
\begin{align*}
&\int_{T_1}^{T_2} \nm{\partial_1 \rho(\cdot,t)}_{\dH_0^{-1}}^2\, dt\\
&\quad\le C_{15} g^{-1} \paren*{\rho_M^{1/4} + \rho_M^{5/3}} \paren*{1 + \int_{T_1}^{T_2} \nm{\grad \rho(\cdot, t)}_2^{3/4}\, dt + \int_{T_1}^{T_2} \nm{\rho(\cdot, t)-\rho_M}_2^{4/3}\, dt}
\end{align*}
for any time interval $[T_1, T_2] \subseteq [0, T_*)$ such that $\nm{\rho(\cdot, t) - \rho_M}_2 \ge 1$ for all $T_1 \le t \le T_2$. Here, $C_{15} = C_{15}(\Omega) > 0$ is a universal constant.
\end{corollary}

\begin{remark}\label{rem:secGWP_gNegative}
Under \Cref{hyp:secGWP}, suppose that $g < 0$. Then we instead define the potential energy of $\rho$ to be
\[
E(t) := \int_\Omega (h - x_2)\rho(x,t)\, dx
\]
for any $0 \le t < T_*$. The analogues of \Cref{lem:secGWP_BoundPotentialEnergy} and \Cref{lem:secGWP_Power} still hold:
\[
0 \le E(t) \le h\nm{\rho_0}_1
\]
and
\[
E'(t) = g\nm{\partial_1 \rho(\cdot,t)}_{\dH_0^{-1}}^2 + \int_\Omega \partial_2 \rho\, dx - \int_\Omega \rho \partial_2 (-\lap_N)\pre [\rho-\rho_M]\, dx,
\]
respectively. It follows that \Cref{prop:secGWP_StratBound} and \Cref{cor:secGWP_StratBound} still hold in the case of $g < 0$, with the same universal constants and with $g$ in the original statements being replaced by $|g|$.
\end{remark}

\subsection{Proof of \texorpdfstring{\Cref{thm:secIntro_GWP}}{Theorem 1.2} and remarks}
\label{subsec:GWPProof}

We are now ready to prove the global well-posedness of the PKS-IPM system (\ref{eq:secIntro_PKSIPM}) on general domains. By \Cref{thm:secGWP_LWP}, it suffices to establish a uniform in time bound on the variance of the solution. We do so by observing a system of four integro-differenetial inequalities all involving the $L^2$ norm of the solution $\rho$; the $\dH^1$ norm of $\rho$; and the $\dH_0\pre$ norm of $\partial_1 \rho$. First, recall the classical two-dimensional Nash inequality (\ref{eq:secIntro_Nash}),
\begin{equation}\label{eq:secGWP_ineq1}
\nm{\rho-\rho_M}_2^2 \le 2C_{\mathrm{N},2}|\Omega| \rho_M\nm{\grad \rho}_2,
\end{equation}
and the Nash stratification inequality \Cref{thm:secIntro_NashTotal}, which implies
\begin{equation}\label{eq:secGWP_ineq2}
\nm{\rho-\rho_M}_2^2
\le
C_0(2|\Omega| \rho_M)^{\frac{8}{7}}\nm{\grad \rho}_2^{\frac{6}{7}}
+ C_1 (2|\Omega|\rho_M)^{\frac{983}{987}} \nm{\partial_1 \rho}_{\dH_0\pre}^{\frac{8}{987}}\nm{\grad \rho}_2^{\frac{983}{987}},
\end{equation}
under further supposing \Cref{as:secSett_2,as:secSett_3}. Moreover, \Cref{cor:secGWP_AprioriVarianceIneq} and \Cref{cor:secGWP_StratBound} (along with \Cref{rem:secGWP_gNegative}) establish the bounds
\begin{equation}\label{eq:secGWP_ineq3}
\frac{d}{dt}\nm{\rho(\cdot,t) - \rho_M}_2^2 + \nm{\grad \rho(\cdot,t)}_2^2 \le (C_{12} + 2\rho_M)\nm{\rho(\cdot,t)-\rho_M}_2^4
\end{equation}
and
\begin{align*}\label{eq:secGWP_ineq4}
&\int_{T_1}^{T_2} \nm{\partial_1 \rho(\cdot,t)}_{\dH_0^{-1}}^2\, dt \\
&\quad\le C_{15} |g|^{-1}\paren*{\rho_M^{1/4} + \rho_M^{5/3}} \paren*{1 + \int_{T_1}^{T_2} \nm{\grad \rho(\cdot, t)}_2^{3/4}\, dt + \int_{T_1}^{T_2} \nm{\rho(\cdot, t)-\rho_M}_2^{4/3}\, dt} \numberthis
\end{align*}
for time intervals $[T_1, T_2] \subseteq [0, T_*)$ over which $\nm{\rho(\cdot,t) - \rho_M}_2 \ge 1$. We study the system of inequalities (\ref{eq:secGWP_ineq1}, \ref{eq:secGWP_ineq2}, \ref{eq:secGWP_ineq3}, \ref{eq:secGWP_ineq4}) abstractly below, showing that it ensures boundedness of the variance.

\begin{proposition}\label{prop:secGWP_GeneralArgument}
Suppose $X(t)$, $Y(t)$, and $Z(t)$ are nonnegative real-valued smooth functions defined on some interval $[0, T_*)$ such that, for some constants $a_1, a_2, a_3 > 0$ and exponents $0 < p < 1/2$, $0 < q < 1$, it holds that
\begin{equation}\label{eq:secGWP_general1}
X \le a_1 Y^{1/2}
\end{equation}
and
\begin{equation}\label{eq:secGWP_general2}
X \le a_2 Y^p + a_3 Z^q Y^{(1-q)/2}
\end{equation}
for all $0 \le t < T_*$. Suppose further that there exist constants $a_4, a_5 > 0$ and exponents $0 < r  < 1,$ $0< 2s <1$ such that
\begin{equation}\label{eq:secGWP_general3}
\frac{d}{dt}X \le a_4 X^2 - Y \qquad \text{ for all } T_1 \le t \le T_2
\end{equation}
and
\begin{equation}\label{eq:secGWP_general4}
\int_{T_1}^{T_2} Z(t)\, dt \le a_5\paren*{1 + \int_{T_1}^{T_2} Y(t)^s\, dt + \int_{T_1}^{T_2} X(t)^r\, dt}
\end{equation}
for any time interval $[T_1,T_2] \subseteq [0, T_*)$ over which $X(t) \ge 1$. Then, we have that
\[
\sup_{0 \le t < T_*} X(t)
\le C_2\max\curly*{X(0),\, C_3},
\]
where $C_2, C_3 > 0$ are finite constants depending only on $a_1, a_2, a_3, a_4, a_5$ and $p, q, r, s$.
\end{proposition}

\begin{proof}
Assume there is some fixed time interval $[T_1, T_2] \subseteq [0, T_*)$ such that
\begin{equation}\label{eq:secGWP_general5}
X(T_1) = e^A, \qquad X(T_2) = e^{A+B}, \qquad e^A < X(t) < e^{A+B}
\end{equation}
for all $T_1 \le t \le T_2$, where $A, B \ge 1$. Throughout the proof, the positive constants $A$ and $B$ will be chosen sufficiently large until we eventually obtain a contradiction. All computations that follow hold over times $T_1 \le t \le T_2$.

We begin by estimating
\[
\frac{d}{dt} \log(X)
\le a_4 X - \frac{Y}{X}
\le a_4 X - a_1^{-1} Y^{1/2}.
\]
by (\ref{eq:secGWP_general3}) and (\ref{eq:secGWP_general1}). Then, (\ref{eq:secGWP_general1}) and (\ref{eq:secGWP_general5}) ensure that we may sufficiently increase $A$ in order to guarantee that $Y \ge (2a_1a_2a_4)^{2/(1 - 2p)}$, and hence by (\ref{eq:secGWP_general2})
\[
\frac{d}{dt} \log(X)
\le  a_2a_4 Y^p + a_3a_4 Z^q Y^{(1-q)/2} - a_1^{-1} Y^{1/2}
\le a_3a_4 Z^q Y^{(1-q)/2} - \frac12 a_1^{-1} Y^{1/2}.
\]
We next integrate in time to obtain
\[
0 < B = \int_{T_1}^{T_2} \frac{d}{dt} \log(X(t))\, dt \le a_3a_4 \int_{T_1}^{T_2} Z(t)^q Y(t)^{(1-q)/2}\, dt - \frac12 a_1^{-1} \int_{T_1}^{T_2} Y(t)^{1/2}\, dt,
\]
where the first equality holds by (\ref{eq:secGWP_general5}). Rearranging and applying H\"{o}lder's inequality then gives
\[
\int_{T_1}^{T_2} Y(t)^{1/2}\, dt
< 2a_1a_3a_4 \int_{T_1}^{T_2} Z(t)^q Y(t)^{(1-q)/2}\, dt
\le 2a_1a_3a_4 \paren*{\int_{T_1}^{T_2} Z(t)\, dt}^q \paren*{\int_{T_1}^{T_2}  Y(t)^{1/2}\, dt}^{1-q}.
\]
We simplify to get:
\[
\int_{T_1}^{T_2} Y(t)^{1/2}\, dt
< (2a_1a_3a_4)^{1/q} \int_{T_1}^{T_2} Z(t)\, dt.
\]
Hence we have by (\ref{eq:secGWP_general4}) and (\ref{eq:secGWP_general1}) that
\[
\int_{T_1}^{T_2} Y(t)^{1/2}\, dt
< (2a_1a_3a_4)^{1/q} a_5\paren*{1 + \int_{T_1}^{T_2} Y(t)^s\, dt + a_1^r \int_{T_1}^{T_2} Y(t)^{r/2}\, dt}.
\]
Let $\theta := \max\curly*{s,\frac{r}{2}}$ and note $0 < \theta < \frac12$ by assumption. For $A$ sufficiently large so that $Y \ge 1$ (as possible by (\ref{eq:secGWP_general1}) and (\ref{eq:secGWP_general5})), it follows
\[
\int_{T_1}^{T_2} Y(t)^{1/2}\, dt
< (2a_1a_3a_4)^{1/q}a_5 \paren*{1+ (1+a_1^r)\int_{T_1}^{T_2} Y(t)^\theta\, dt}.
\]
We may continue the upper bound by using (\ref{eq:secGWP_general1}) with $0 < \theta < 1/2$ as well as (\ref{eq:secGWP_general5}), to obtain
\begin{align*}
\int_{T_1}^{T_2} Y(t)^{1/2}\, dt
&< (2a_1a_3a_4)^{1/q} a_5 \paren*{1+ (1+a_1^r)\int_{T_1}^{T_2} Y(t)^{1/2} [X(t)/a_1]^{2\theta-1}\, dt} \\
&\le (2a_1a_3a_4)^{1/q} a_5 \paren*{1+ (1+a_1^r)a_1^{1-2\theta}e^{-(1-2\theta)A}\int_{T_1}^{T_2} Y(t)^{1/2}\, dt}.
\end{align*}
Again as $0 < \theta < 1/2$, we may increase $A$ a final time so that $(2a_1a_3a_4)^{1/q} a_5(1+a_1^r)a_1^{1-2\theta}e^{-(1-2\theta)A} \le 1/4$, giving us
\begin{equation}\label{eq:secGWP_general7}
\int_{T_1}^{T_2} Y(t)^{1/2}\, dt < (2a_1a_3a_4)^{1/q} a_5 + \frac{1}{4}\int_{T_1}^{T_2} Y(t)^{1/2}\, dt.
\end{equation}

To complete the proof, we note that Gr\"{o}nwall's inequality applied to (\ref{eq:secGWP_general3}) with $Y \ge 0$ and (\ref{eq:secGWP_general5}) imply that
\[
e^{A+B} = X(T_2) \le X(T_1) e^{a_4\int_{T_1}^{T_2} X(t)\, dt} = e^{A}e^{a_4\int_{T_1}^{T_2} X(t)\, dt}.
\]
By (\ref{eq:secGWP_general1}), it follows
\begin{equation}\label{eq:secGWP_general8}
B \le a_4\int_{T_1}^{T_2} X(t)\, dt \le a_1a_4\int_{T_1}^{T_2} Y(t)^{1/2}\, dt.
\end{equation}
Choosing $B$ so that $B/4 := (2a_1a_3a_4)^{1/q} a_1a_4a_5$, we obtain the contradiction
\[
0 \le \int_{T_1}^{T_2} Y(t)^{1/2}\, dt < \frac{1}{2}\int_{T_1}^{T_2} Y(t)^{1/2}\, dt
\]
by (\ref{eq:secGWP_general7}) and (\ref{eq:secGWP_general8}). In summary, let $A_*$ be the minimal positive constant that is sufficiently large according to the above requirements throughout the proof and let $B_* := 4(2a_1a_3a_4)^{1/q}a_1a_4a_5$. Set $C_2 := e^{B_*}$ and $C_3 := e^{A_*}$. If $X(0) \le C_3$, we can run the above argument with $A = A_*$ and $B = B_*$. If $X(0) > C_3$, then we instead run the above argument with $A = \log(X(0))$ and $B = B_*$.
\end{proof}

From the above general setting, we quickly recover the second main result of this article.

\GWP*

\begin{proof}
Without loss of generality, we further impose \Cref{norm:secSett_h,norm:secSett_eps} on the domain. The classical two-dimensional Nash inequality (\ref{eq:secIntro_Nash}) implies (\ref{eq:secGWP_ineq1}). The Nash stratification inequality \Cref{thm:secIntro_NashTotal} implies (\ref{eq:secGWP_ineq2}). Lastly, for any time interval time $[T_1, T_2] \subseteq [0, T_*)$ over which $\nm{\rho(\cdot,t) - \rho_M}_2 \ge 1$, \Cref{cor:secGWP_AprioriVarianceIneq} and \Cref{cor:secGWP_StratBound} (along with \Cref{rem:secGWP_gNegative}) establish the bounds (\ref{eq:secGWP_ineq3}) and (\ref{eq:secGWP_ineq4}). Altogether, we may apply \Cref{prop:secGWP_GeneralArgument} with
\[
X(t) := \nm{\rho(\cdot, t) - \rho_M}_2^2, \qquad
Y(t) := \nm{\grad\rho(\cdot, t)}_2^2, \qquad
Z(t) := \nm{\partial_1\rho(\cdot, t)}_{\dH_0\pre}^2,
\]
and
\[
p = 3/7, \qquad q = 4/987, \qquad r = 2/3, \qquad s = 3/8,
\]
to conclude that the solution's variance obeys the desired bound on $[0, T_*)$. Then, the blow-up criterion \Cref{thm:secGWP_LWP} completes the proof. We note that, in the application of \Cref{prop:secGWP_GeneralArgument}, the choices of $a_1$, $a_2$, $a_3$, $a_4$, and $a_5$ all depend on $\Omega$ and $\rho_M$ while $a_5$ also depends on $1/|g|$.
\end{proof}

\begin{remark}
In the proof of \Cref{thm:secIntro_GWP}, \Cref{as:secSett_2,as:secSett_3} on the domain are only required in order to invoke the Nash stratification inequality. The classical Nash inequality, \Cref{cor:secGWP_AprioriVarianceIneq}, and \Cref{cor:secGWP_StratBound} only require \Cref{as:secSett_1} on the domain. Thus, any relaxations of the assumptions required for \Cref{thm:secIntro_NashTotal} would automatically guarantee global well-posedness of the PKS-IPM system on such domains by the exact same argument as above.
\end{remark}

\begin{remark}
From the proofs of \Cref{thm:secIntro_GWP} and \Cref{prop:secGWP_GeneralArgument}, one can see that the constants $C_2$ and $C_3$ both depend only on $\Omega$, $|g|$, and $\rho_M$. Indeed, $C_2$ and $C_3$ increase when $\rho_M$ increases and increase when $|g|$ decreases. This matches the natural intuition. However, the manner in which $C_2$ and $C_3$ depend on the domain $\Omega$ cannot be easily seen from our proof, as the universal constant coefficients that we obtained for \Cref{thm:secIntro_NashTotal} do not have simple dependencies on the domain --- boundary geometry, and not just the domain's area, plays a role.
\end{remark}

\section*{Acknowledgements}
\addcontentsline{toc}{section}{\numberline{}Acknowledgements}

AK has been partially supported by the NSF-DMS Grant 2306726. NAS has been partially supported by the NSF-DMS Grant 2038056 and by a Grant in Aid of Research from Sigma Xi, The Scientific Research Honors Society.


\end{document}